\documentclass[generic,10pt,oneside,fleqn]{imsart}

\usepackage{amsmath}
\usepackage{amsthm}
\usepackage{graphicx}
\usepackage{textgreek}
\usepackage{upgreek}
\usepackage{hyperref}
\usepackage{mathrsfs}
\usepackage{color}
\usepackage{bm}
\usepackage{amsfonts}
\usepackage{amssymb}
\usepackage{csquotes}
\usepackage{enumerate}

\RequirePackage[OT1]{fontenc}
\RequirePackage{amsthm,amsmath}
\RequirePackage[numbers]{natbib}
\RequirePackage[colorlinks,citecolor=blue,urlcolor=blue]{hyperref}

\PassOptionsToClass{10pt,oneside,fleqn}{article}
  \PassOptionsToPackage{fleqn}{amsmath}
  \PassOptionsToPackage{numbers,sort&compress}{natbib}
\usepackage{soul}

\startlocaldefs
\numberwithin{equation}{section}
\theoremstyle{plain}

\numberwithin{equation}{section}

\newtheorem{theorem}{Theorem}[section]
\newtheorem{lemma}[theorem]{Lemma}
\newtheorem{proposition}[theorem]{Proposition}
\newtheorem{corollary}[theorem]{Corollary}

\theoremstyle{definition}

\newtheorem{remark}[theorem]{Remark}

\usepackage[all]{xy}
  \setattribute{title}       {skip} {64\p@}
  \setattribute{authors}     {skip} {14\p@}
  \setattribute{address}     {skip} {12\p@}
  \setattribute{abstractname}{skip} {.\enskip}
  \setattribute{history}     {skip} {4\p@}
  \setattribute{abstract}    {skip} {16\p@}

  \setattribute{fline}       {cmd}  {\vskip22\p@
                                     \hrule}
  \setattribute{lline}       {cmd}  {\vskip10\p@
                                     \hrule}

  \setattribute{abstract}   {width} {\textwidth}
  \setattribute{abstract}  {indent} {\z@} 
  \setattribute{keyword}    {width} {\textwidth}
  \setattribute{keyword}   {indent} {\z@} %

  \setattribute{title}      {size} {\huge}
  \setattribute{author}     {size} {\Large}
  \setattribute{abstract}   {size} {\small\upshape}
  \setattribute{keywordname}{size} {\itshape}
  \setattribute{thebibliography}{size}{\footnotesize}

  \def\put@fmt@data{%
    \copyright@fmt%
    \@thanks%
    \history@fmt
    \fline@cmd%
    \abstract@fmt%
    \keyword@fmt%
    \lline@cmd%
    \tableofcontents@fmt}

  \setattribute{keyword}{KWD}{Keywords:}
  \setattribute{email}{text}{E-mail: }

  \renewcommand\subsection{\@startsection {subsection}{2}{\z@}%
                                     {-\medskipamount}%
                                     {\medskipamount}%
                                     {\itshape\raggedright}}

  \renewcommand\subsubsection{\@startsection{subsubsection}{3}{\z@}%
                                     {-\medskipamount}%
                                     {.01\p@}%
                                     {\itshape\raggedright}}

  \renewcommand\paragraph{\@startsection{paragraph}{4}{\z@}%
                                     {\medskipamount}%
                                     {-10pt}%
                                     {\itshape}}

  \renewcommand\subparagraph{\@startsection{subparagraph}{5}{\parindent}%
                                     {0.1pt}%
                                     {-1em}%
                                     {\itshape}}

\def\ims@thmshape{3}

 \setattribute{tablecaption}{size} {\footnotesize}
  \setattribute{figurecaption}{size}{\footnotesize}

\def\ud{\mathrm{d}}
\def\E{{\mathbb E}}
\def\EE{{\mathbb E}}
\def\R{{\mathbb R}}
\def\RR{{\mathbb R}}

\def\PP{{\mathbb P}}

\def\cC{{\mathcal C}}

\def\DTheta{\textrm{\rm grad} \,  \Theta}
\def\HTheta{\textrm{\rm Hess} \,  \Theta}
\endlocaldefs
\begin{document}

\begin{frontmatter}
\title{Zero noise limit for multidimensional SDEs driven by a pointy gradient}
\runtitle{Zero noise limit for multidimensional SDEs}
\begin{aug}
\author{\fnms{Fran\c{c}ois} \snm{Delarue} \thanksref{t1}\ead[label=e1]{delarue@unice.fr}},
\author{\fnms{Mario} \snm{Maurelli} \thanksref{t2}\ead[label=e2]{mario.maurelli@unimi.it}},
\runauthor{F. Delarue and M. Maurelli}

\thankstext{t1}{F.D. is supported by Institut Universitaire de France.}
\thankstext{t2}{M.M. thanks Universit\'e C\^ote d'Azur 
for the hospitality extended to him during his two week visit to Nice.}

\address{Universit\'e C\^ote d'Azur \\
Laboratoire J.-A. Dieudonn\'e
\\
Parc Valrose, 06108 Nice Cedex 02, France
\\
\printead{e1}
}

\address{Universit\`a degli Studi di Milano
\\
Department of Mathematics ``Federigo Enriques'' 
\\
via Saldini 50, 20133 Milano, Italy
\\
\printead{e2}
}

\end{aug}
\begin{abstract}
The purpose of the article is  to address the limiting behavior of the solutions of stochastic differential equations driven by a pointy $d$-dimensional {gradient} as the intensity of the underlying Brownian motion tends to $0$. By pointy {gradient}, we here mean that the {drift derives from a potential that is} ${\mathcal C}^{1,1}$ on any compact subset that does not contain the origin. As a matter of fact, the corresponding deterministic version of the differential equation may have an infinite number of solutions when {initialized} from {$0_{\RR^d}$}, in which case the limit theorem proved in the paper reads as a selection theorem of the solutions to the zero noise system. 

Generally speaking, our result says that, under suitable conditions, the probability that the particle leaves the origin by 
going through regions of higher potential tends to $1$ as the intensity of the noise tends to $0$. 
In particular, our result extends the earlier one due to Bafico and Baldi
\cite{BaficoBaldi} for the zero noise limit of one dimensional stochastic differential equations.
\end{abstract}

\end{frontmatter}

\section{Introduction}

\subsection{An overview of zero noise limits}
In his seminal paper 
\cite{Peano}, Peano addressed the existence of solutions to ordinary differential equations driven by 
continuous but possibly non-Lipschitz coefficients. Meanwhile, he highlighted the fact that, for some initial conditions, the equation could have several solutions. 
Those initial conditions are referred to as \textit{Peano points} and the fact that there exist several solutions is called 
\textit{Peano phenomenon}. A nice introduction to the subject is given in the thesis of Charpentier, 
see \cite{Charpentier2} together with the companion paper \cite{Charpentier1}. 
Therein, she points out several properties of the range of values taken at a given time by all the solutions {initialized} from a common Peano point, among which the so-called Mie and Osgood-Montel theorems: When {the differential equation is set in dimension 1}, the range {formed by all the solutions} at any time is shown to be an interval and extremal values of these intervals are shown to form extremal solutions as the underlying time index varies. 

Peano's phenomenon takes a somewhat  {puzzling} turn when the underlying differential equation is forced by a white noise. As proven in the mid 70's and the early 80's by Zvonkin 
\cite{Zvonkin}
(for one-dimensional equations), 
Veretennikov \cite{Veretennikov} (for higher dimensional equations) and 
{Stroock and Varadhan  \cite{StroockVaradhan} (at least for weak solutions, as addressed through the so-called martingale problem)}, noise restores uniqueness under pretty general boundedness conditions on the velocity field of the equation. In other words, 
Peano's phenomenon disappears in the presence of noise. 
Since these earlier results, several extensions have been addressed, including restoration of uniqueness for a more singular velocity field, restoration of uniqueness for an infinite dimensional stochastic differential equation and restoration of uniqueness under other types of noise than Brownian motion. We refer among others to 
Krylov and R\"ockner 
\cite{kry:roc:05}, Bass and Chen
\cite{BassChen},
Flandoli, Russo and Wolf
\cite{fla:rus:wol:03,fla:rus:wol:04}, 
Flandoli, Issoglio and Russo
\cite{FlandoliIssoglioRusso},
Delarue and Diel \cite{DelarueDiel},
Flandoli, Gubinelli and Priola \cite{FlandoliGubinelliPriola}, Davie \cite{dav:07} and Catellier and Gubinelli \cite{CatellierGubinelli}. The reader may also find a complete overview in the monograph of Flandoli \cite{Flandoli}.

The fact that noise may restore uniqueness {sheds} a new light on Peano's phenomenon. A natural
question is indeed to address the (weak) limit of the solution of the stochastic version of the ordinary differential 
equation as the intensity of the noise tends to $0$. Such a procedure is usually referred to as taking the 
\textit{zero-noise} limit in the corresponding stochastic differential equation. The intuition is that the zero-noise limit should select some \textit{special} solutions
among all the solutions of the original ordinary differential equation: As they are obtained by forcing the dynamics randomly, those \textit{special} solutions should be regarded as being the most meaningful ones from a physical point of view. The main result in this framework is due to Bafico and Baldi \cite{BaficoBaldi}:
It gives a quite complete picture of the zero-noise limit in the case of a one-dimensional ordinary differential equation with an isolated Peano point. Basically, the main result therein asserts that the zero-noise limit 
is a probability measure 
concentrated on extremal solutions; when {there are two extremal solutions that do leave the singularity}, 
the weight of each of them 
 depends on 
the rate at which it leaves the Peano point: the higher the rate, the higher the weight. 

Bafico and Baldi's result has been extended in several ways. In \cite{GradinaruHerrmannRoynette},
the authors addressed large deviations in the zero-noise limit, proving in particular that the rate at which the density of the stochastic solution decays (with the intensity of the noise) at points that do not seat on the extremal paths is not the same outside and inside the ``cone'' formed 
by the two extremal solutions. In 
\cite{DelarueFlandoli}, another proof of \cite{BaficoBaldi} is given for ordinary equations driven by power functions with an exponent between $0$ and $1$. Whilst 
\cite{BaficoBaldi} is mostly based upon PDE arguments (as it makes use of explicit solutions to 
elliptic equations for the exit time of the underlying diffusion process from an interval), the approach of  
\cite{DelarueFlandoli} is based upon the pathwise concept of transition point: 
Roughly speaking, a transition point is a time-space point $(t,x)$ that depends on the intensity of the noise with the following three features: $(i)$ $(t,x)$ converges {in time-space coordinates to the Peano point} as the intensity tends to $0$;
$(ii)$ as the intensity tends to $0$, the noise dominates in the dynamics up to time $t$ and the velocity field dominates after time $t$; (iii) $x$ is a typical position of the process at time $t$. 
In this framework, the strategy used in \cite{DelarueFlandoli} is to compute explicitly the transition points in function of the shape of the velocity field. 
Another probabilistic proof, based on It\^o Tanaka formula and related local times, is given in \cite{Trevisan} for velocity fields of the same form as in \cite{DelarueFlandoli}. 
 Lastly, 
in \cite{AttanasioFlandoli}, the authors investigate a 
$1d$ stochastic linear transport equation by exploiting the fact that the characteristics are precisely given 
by the stochastic differential equation considered in \cite{BaficoBaldi}.

Other extensions of \cite{BaficoBaldi} concern the $d$-dimensional case, but {they are not as precise as the analysis performed in \cite{BaficoBaldi} or they are in some way reminiscent of the 1d case}. In \cite{BuckdahnOuknineQuincampoix}, the authors prove that the zero-noise limit is concentrated on so-called Filippov's solutions of the ordinary differential equation, but this result mainly concerns equations with a discontinuous velocity field; indeed,  
for
a continuous velocity field, the notion of Filippov's solution coincides with the standard notion 
of solution and is then of little use. {The article \cite{PilipenkoProske} addresses the case where the drift has a discontinuity on an hyperplane. The authors are able to characterize the zero-noise limit in several situations (drift pushing {towards or out of} the hyperplane or even parallel to the hyperplane). This is a non-trivial multi-dimensional result, but it seems a quite different problem from drifts with one-point singularity, since it often relies on the time {duration} the process spends on each of the half-spaces generated by the hyperplane.}
In \cite{DelarueFlandoli2} and in 
\cite{JourdainReygner}, the authors address the zero-noise limit for 
specific higher dimensional examples:
\cite{DelarueFlandoli2} deals with a Vlasov-Poisson system of two particles and
 \cite{JourdainReygner} deals with an $n$-dimensional ordinary differential equation with a velocity field that depends on the ordered arrangement of the solution. { We also mention that the techniques in \cite{Trevisan} {may} probably accomodate the multi-dimensional case {when} the drift is non-Lipschitz and strongly repulsive in one point; however, here the strong repulsive assumption induces to consider only one direction (the radial direction), thus reducing morally the multi-dimensional issue to the one-dimensional case ({we exemplify the notion of strong repulsivity below}).}
 
{In contrast, the main scope of this paper is to analyse a case of a drift, which is singular (that is non-Lipschitz) at one point, but not necessarily strongly repulsive: hence the multi-dimensional nature of the problem cannot be disregarded.}

\subsection{Equation under concern}
In this paper, we consider the following equation:
\begin{equation}
\label{eq:1:1}
\ud X_{t}^{\varepsilon} =   \nabla V\bigl(X_{t}^{\varepsilon}\bigr) \ud t +
\varepsilon \ud B_{t}, \quad X_{0}^{\varepsilon}=0,
\end{equation}
where $(B_{t})_{t \geq 0}$ is a $d$-dimensional Brownian motion, 
with $d \geq 2$, 
$\varepsilon$ is a positive intensity parameter, which is intended to be small, and $V$ is a real-valued map defined on $\RR^d$, which is referred to as \textit{a potential}. In particular, $\nabla V$ maps $\RR^d$ into itself and the state $X_{t}^{\varepsilon}$
at time $t \geq 0$ is a vector of $\RR^d$. 
We sometimes regard the dynamics of $(X_{t}^{\varepsilon})_{t \geq 0}$ as the motion of  \textit{a particle}.
 
Importantly, $\nabla V$ is not assumed to be Lipschitz in the neighborhood of $0$ but is required to be continuous at $0$. As a byproduct, the 
zero noise version of the equation may be ill-posed. 
A typical instance is 
\begin{equation*}
V(x) = \vert x \vert^{1+\alpha}, \quad x \in \RR^d,
\end{equation*}
for $\alpha \in [0,1)$, 
where $\vert \cdot \vert$ denotes the Euclidean norm, in which case 
\begin{equation}
\label{eq:1:3}
\nabla V(x) = (1+\alpha) \vert x \vert^{\alpha-1} x, \quad x \in \RR^d. 
\end{equation}
Although the form of the singularity in \eqref{eq:1:3} is somehow representative of the types of potentials addressed in this paper, the radial structure of $V$ makes the example in itself of a somewhat limited scope. Using the rotational invariance of the dynamics, one may indeed investigate the solutions to 
\eqref{eq:1:1}--\eqref{eq:1:3} by following the arguments developed by Bafico and Baldi 
\cite{BaficoBaldi} in the one-dimensional case:
{\eqref{eq:1:3} is a typical instance of what we called above a \textit{strongly repulsive} drift}.

Our objective is thus to go further and to address cases where $V(x)$ behaves like $\vert x \vert^{1+\alpha}$ in some directions only. Typically, the result we prove below applies to potentials of the form
\begin{equation}
\label{eq:1:2}
V(x) = g \bigl( \frac{x}{\vert x \vert} \bigr) \vert x \vert^{1+\alpha}, \quad x \in \RR^d \setminus \{0\}, \quad V(0) =0,
\end{equation}
for a non-negative smooth enough spherical\footnote{i.e. defined on the sphere ${\mathbb S}^{d-1}$.} function $g$ that is non-zero on a sector with a non-empty interior. The detailed conditions that we need are spelled out below, see Section \ref{se:2}. 
In fact, our approach allows for more general potentials of the form
\begin{equation}
\label{eq:1:2:b}
V(x) = g(x) \vert x \vert^{1+\alpha} + h(x) \vert x \vert^{1+\beta}, \quad x \in \RR^d \setminus \{0\}, \quad V(0) =0,
\end{equation}
where $g$ is a non-negative perturbation of a spherical function 
(see \eqref{eq:expansion:V}
for a precise meaning)
and $h$
is an arbitrary function, the supports of the two being included in two disjoint $0$-originated cones of $\RR^d$. 
Importantly, $\alpha$ is assumed to be (strictly) less than $\beta$, which says that, close to the origin, 
the highest values of the potential are imposed by $g$\footnote{We believe that our analysis could be generalized to the case when $\beta < \alpha$ and $h$ is negative, but this would require an additional effort, see Remark \ref{re:beta<alpha}.}.

Equation \eqref{eq:1:1} is called \textit{a gradient flow}. The main feature is that, in the zero noise case, namely when $\varepsilon=0$, the potential is non-decreasing along the solutions of \eqref{eq:1:1}. When the dynamics are {initialized} from a point different from $0$ at which the gradient of $V$ is not zero, the potential is locally increasing (in the strict sense); in particular, when the dynamics start from a point at which $g$ is positive, the particle goes away from $0$ since the gradient of $V$ has a positive radial direction. 
For sure, when the initial condition is $0$, nothing can be said. When the initial condition is in the region 
$\{h \not =0\}$, the motion of the particle heavily depends on the sign of 
$h$, but, actually, the latter does not matter for our purposes. Such a picture slightly differs in the presence of noise: When $\varepsilon >0$, there is a competition between the steepness of the potential and the intensity of the noise. In this regard, the $1d$ case says that, in cases like \eqref{eq:1:2}
or \eqref{eq:1:2:b}, for which the gradient of the potential at $x$ may be much larger than $\vert x \vert$, the noise should only matter when the particle is really close to $0$; as soon as the particle is sufficiently far away from $0$, the deterministic dynamics  {should} dominate. Intuitively, the expected picture should be as follows: 
the steeper the potential the stronger the influence it has on the dynamics. In other words, whilst the particle has an isotropic motion under the sole action of the noise, the way it feels the potential differs from one 
region of the space to another; in particular, at a given distance of the origin, the action of the potential should depend upon the angle of the particle. In this regard, it must be clear that the potential should become effective at a smaller
distance in the region $\{g >0\}$ since it takes higher values there. In this region, the deterministic dynamics should dominate rather quickly and should push
the  
solution to \eqref{eq:1:1} away from the singularity, hence leading to the following guess: The particle leaves the singularity
by going through the region $\{ g >0\}$; certainly, the most likely directions of escape are those along which 
$g$ is maximal.

Here we provide a quite complete picture of this phenomenon for a family of potentials like \eqref{eq:1:2:b}. We first show that the noise dictates the form of the solution to \eqref{eq:1:1} (when $\varepsilon >0$ and $X_{0}^{\varepsilon}=0$) as long as the potential does not exceed a level of order $\varepsilon^2$ (whatever the form of $V$ provided it satisfies our conditions). This is an analogue of the concept of transition 
point used in \cite{DelarueFlandoli}, 
see Section \ref{se:3}.
The key fact to do so is to compare the law of the path of the particle to the law of 
a Brownian motion of intensity $\varepsilon$ by using Girsanov's theorem for martingales with a bounded mean oscillation. 
 Next, we prove that the sites at which the potential hits a level of order $\varepsilon^2$ are located, with probability asymptotically equal to $1$,  in the region $\{g>0\}$, see Section 
 \ref{se:4}. In Section \ref{se:escaping:0}, we prove that, when restarting with a potential of order $\varepsilon^2$ from 
 the region $\{g >0\}$, the particle stays within the region $\{ g >0\}$ with probability asymptotically equal to 
 1 and escapes from the singularity at a macroscopic rate. Lastly, we prove that the directions that the particle follows (while escaping from the singularity) are asymptotically given by the unit vectors $x/\vert x \vert$ of ${\mathbb S}^{d-1}$ that maximize {$g(r \frac{x}{\vert x\vert})$
 for $r$ close to $0$}, see Section \ref{se:6}. 
 Numerical examples are provided in  Section \ref{se:7} to illustrate our results. 
  Assumptions and statements are provided in Section \ref{se:2}. 
 
\section{Assumption and Statements}
\label{se:2}

\subsection{Assumptions} 
Throughout the paper, 
$(\Omega,{\mathcal F},{\mathbb F}=({\mathcal F}_{t})_{t \geq 0},{\mathbb P})$
is a complete filtered probability space satisfying the usual conditions and 
$(B_{t})_{t \geq 0}$ is a $d$-dimensional Brownian motion with respect to ${\mathbb F}${, where $d\ge2$}. 
Moreover, 
we use two sets of assumptions. The first one is defined right below. 
\vskip 10pt

\noindent \textbf{Assumption A.} 
\begin{itemize}
\item[({\textbf{A1}})]
The function $V$ is
a ${\mathcal C}^1$ Lipschitz function on $\RR^d$ which is ${\mathcal C}^{1,1}$ on any compact subsets of $\RR^d \setminus \{0\}$. It has the following writing:
\begin{equation*}
V(x) = g(x) \vert x \vert^{1+\alpha} + h(x) \vert x \vert^{1+\beta}, 
\end{equation*}
for $x \in \RR^d \setminus \{0\}$, where 
$\alpha \in (0,1)$ and $\beta \in (\alpha,1]$. Above, 
$g$ and $h$ are bounded functions that are 
 ${\mathcal C}^{1,1}$ 
on any compact subsets of $\RR^d \setminus \{0\}$. 
In particular, $V(0)=0$. 
The function $g$ is non-negative.
The sets $\{g \not =0\}$ and 
$\{ h \not =0\}$ are disjoint. {We denote by $C_{0}$ a common bound
to $g$ and $h$}.
\item[({\textbf{A2}})]
The gradient and Hessian of $V$ (the latter being defined almost everywhere) satisfy 
\begin{equation*}
\bigl\vert \nabla V(x) \bigr\vert \leq C_{0} \vert   x \vert^{\alpha}, \quad 
{\nabla V(x) \cdot \frac{x}{\vert x \vert} \geq - C_{0} \vert   x \vert^{\beta}}, 
\quad \bigl\vert \nabla^2 V(x) \bigr\vert \leq C_{0} \vert x \vert^{\alpha-1}.
\end{equation*}
In particular, $\nabla V(0)=0$. 
Also, there exists $a_{0}>0$ such that, 
 for any $x \in {\mathbb R}^d \setminus \{0\}$ with $g(x) \in (0,a_{0})$, 
\begin{equation*}
\Delta V(x) \geq 0,
\end{equation*}
{provided that the Hessian of $V$ is well-defined at point $x$.}
\item[({\textbf{A3}})]
There exists a cone ${\mathcal C}$ with $0$ as vertex 
such that $\inf_{x \in {\mathcal C}} g(x) \geq c_{0}$. 
In the interior of the set $g^{-1}(0)=\{g=0\}$, the gradient and Hessian of $V$ (the latter being defined almost everywhere) satisfy
\begin{equation*}
\bigl\vert \nabla V(x) \bigr\vert \leq C_{0} \vert   x \vert^{\beta}, \quad 
\bigl\vert \nabla^2 V(x) \bigr\vert \leq C_{0} \vert x \vert^{\beta-1}.
\end{equation*}
In the region $\{ g >0\}$,
\begin{equation*}
\bigl\vert \nabla V(x) \bigr\vert \geq c_{0} \frac{V(x)}{\vert x \vert}. 
\end{equation*}
\item[({\textbf{A4}})]
The rate at which $V$ and its derivative vanish at the boundary of $\{g >0\}$ is dictated by a function 
$L : {\RR^d  \setminus \{0\}} \rightarrow \RR_{+}$ that vanishes on $\partial \{g>0\}$ and by an exponent $p >0$
such that
\begin{equation*}
 c_{0} L(x)^{p+1}  \vert x \vert^{1+\alpha} \leq V(x) \leq C_{0} L(x)^{p+1}  \vert x \vert^{1+\alpha},
 \quad 
 c_{0} L(x)^p  \vert x \vert^{\alpha} \leq \bigl\vert \nabla V(x)
 \bigr\vert \leq C_{0} L(x)^p  \vert x \vert^{\alpha},
\end{equation*}
for $x $ such that $g(x) \in (0,a_{0})$. 
\vskip 4pt

\noindent
(If $g$ does not vanish, (\textbf{A4}) is satisfied for $a_{0}$ small enough.)
\end{itemize} 

\begin{remark}
Since we are just interested in the local behavior of $X^{\varepsilon}$ in the neighborhood of $0$, 
we can easily relax the above assumptions and assume that $V$ satisfies the above properties in 
a neighborhood of $0$ only.   
\end{remark}

\begin{remark}
\label{re:beta<alpha}
{An interesting extension is the case when $\beta$ is allowed to be strictly less than $\alpha$
and $h$ is negative {and decreases with $\vert x \vert$ (so that $h$ becomes more and more negative when moving away from the origin)}. In that case, we may expect that the particle hardly goes into the region 
$\{h <0\}$ because of the effect of the potential: In the part of the space where $h$ is negative, the potential tends to push back the particle to the origin.
Anyhow, our proof does not apply to this new case: One of our main step in the proof is to show that, until the potential reaches thresholds of order $\varepsilon^2$, 
the particle behaves like a Brownian motion, see Section 
\ref{se:3}. When $V$ is steeper on $\{h \not =0\}$ (the slope being here negative in the radial direction) than 
on $\{g =0\}$ (the slope being positive in the radial direction), this picture may no longer be true since 
$h$ may induce some reflection phenomenon (the precise form of which should be addressed carefully).
 
 Another extension concerns the case of non-potential perturbations. Again, 
 we feel that the tools that we use here could be adapted but, clearly, some care would be needed to make sure that the whole machinery
 indeed works.}
\end{remark}

Here is the second set of assumptions that we shall use. 
\vskip 10pt

\noindent \textbf{Assumption B.} 

\begin{enumerate}
\item[({\textbf{B1}})]
For the same $L$, $p$, {$C_0$} and $a_{0}$ as in (\textbf{\bf A4}), 
but under the additional assumption that $p \geq 1$, 
\begin{equation*}
 \bigl\vert \nabla^2 V(x)
 \bigr\vert \leq C_{0}
L(x)^{p-1}  \vert x \vert^{\alpha-1},
\end{equation*}
for almost every $x$ such that $g(x) \in (0,a_{0})$.

\item[({\textbf{B2}})]
There exist a differentiable map $\Theta : {\mathbb S}^{d-1} \rightarrow  {{\mathbb R}_{+}}$ and a function 
$\eta : \RR^d \rightarrow \RR$ such that 
\begin{equation}
\label{eq:expansion:V}
g(x) = \Theta\bigl( \frac{x}{\vert x \vert} \bigr) + \eta(x), \quad x \in \{ g >0\},
\end{equation}
where $\eta$ satisfies 
$\limsup_{\delta \searrow 0} \sup_{\vert x \vert \leq \delta} (\vert \eta(x) \vert / \vert x \vert )< \infty$. 
Moreover, the tangential direction to $\nabla V$ in the region $\{ g >0\}$ may be decomposed in the form
\begin{equation}
\label{eq:expansion:nablaV}
\begin{split}
&\nabla V(x) - 
\bigl( \nabla V(x) \cdot \frac{x}{\vert x \vert}
\bigr) \frac{x}{\vert x \vert}
= \Bigl[ \text{grad} \,  \Theta \bigl( \frac{x}{\vert x \vert} \bigr)
  + \eta'(x) \Bigr] \vert x \vert^{\alpha}, \quad x \in \{g >0\},
  \end{split}
\end{equation}
where, for $u \in {\mathbb S}^{d-1}$, $\DTheta(u)$ is the gradient of $\Theta$
at point $u \in {\mathbb S}^{d-1}$, which is a vector of the tangent space
$T_{u} {\mathbb S}^{d-1}$  
to the sphere ${\mathbb S}^{d-1}$ at $u$ that we regard as a vector of $\RR^d$, the tangent space to the sphere reading as the 
space of vectors of $\RR^d$ that are orthogonal to $u$. 
In the above expansion, the function $\eta'$ {(which should not be confused with the derivative of $\eta$)} is required to satisfy 
$\limsup_{\delta \searrow 0} \sup_{\vert x \vert \leq \delta} (\vert \eta'(x) \vert / \vert x \vert) < \infty$.
Lastly, we require that the function 
\begin{equation}
\label{eq:C11}
{\mathbb S}^{d-1} \ni u \mapsto 
 \DTheta(u)  \in {T_{u} {\mathbb S}^{d-1} \subset \RR^d}
\end{equation}
is Lipschitz. 
\end{enumerate}
\vskip 5pt

Assumption \eqref{eq:expansion:V} plays a crucial role in our analysis: Up to the perturbation $\eta$, 
it permits to separate the radius and the angular structures of the potential. In this regard, it is worth noting that, for $x$ small, the condition $g(x)>0$ is \textit{essentially} equivalent to 
$\Theta(x/\vert x \vert) >0$ (which, for a given $u \in {\mathbb S}^{d-1}$, may written in the form $\lim_{\lambda \searrow 0} g(\lambda u) >0$ if and only if $\Theta(u)>0$). 
Of course, the condition
\eqref{eq:expansion:nablaV}
for the gradient of the potential
has a similar role. It
is obtained by differentiating formally 
\eqref{eq:expansion:V} {(which can be made rigorous if $\eta$ is differentiable)},  {by means of} the useful relationship
\begin{equation}
\label{eq:DTheta}
\nabla \Bigl[ \Theta \bigl( \frac{x}{\vert x\vert} \bigr) \Bigr]=
\frac{1}{\vert x \vert}
\DTheta \bigl( \frac{x}{\vert x \vert} \bigr)
=
\frac{1}{\vert x \vert} \Bigl( I_{d} - \frac{x}{\vert x \vert} \otimes \frac{x}{\vert x \vert}
\Bigr) \nabla \Theta \bigl( \frac{x}{\vert x \vert} \bigr),
\end{equation}
for $x \in {\mathbb R}^d \setminus \{0\}$, where 
$\nabla$ in the left and right terms is the usual $d$-dimensional Euclidean gradient. 
 {In the last term, ${I_{d}} - \frac{x}{\vert x \vert} \otimes \frac{x}{\vert x \vert}$ is the orthogonal projection onto 
the orthogonal vector space to $x$ and $\nabla \Theta$ makes sense if $\Theta$ is extended to a neighborhood of the sphere.}

Condition \eqref{eq:C11}
says that $\Theta$ is ${\mathcal C}^{1,1}$.  
It
 implies in particular that 
$\sup_{u \in {\mathbb S}^{d-1}} 
\vert  \DTheta(u) \vert < \infty$. 
It is extremely useful as it permits to consider the 
ODE (with values in $\RR^d$)
\begin{equation}
\label{eq:ODE:2:2}
\dot{\phi}_{t}  = \DTheta\bigl(\frac{\phi_{t}}{\vert \phi_{t} \vert}\bigr), \quad t \geq 0 \ ; \quad \phi_{0} = u, 
\end{equation}
for an initial condition $u \in {\mathbb S}^{d-1}$. 
By \eqref{eq:C11}, the solution is locally well-defined. In fact, since $u \in {\mathbb S}^{d-1}$
and $\DTheta$ is orthogonal to $u$, the solution remains on the sphere and 
is (uniquely) defined over the entire $[0,\infty)$. It is denoted by $(\phi_{t}^u)_{t \geq 0}$.  
Actually, we set \eqref{eq:ODE:2:2} on $\RR^d$ for convenience. Since $\DTheta$ may be regarded as a vector field on ${\mathbb S}^{d-1}$,  we  {can also consider 
the equation on the sphere, in which case it takes the form}
\begin{equation}
\label{eq:ODE:2:2:sphere}
\dot{\phi}_{t}  = \DTheta(\phi_{t}), \quad t \geq 0 \ ; \quad \phi_{0} = u. 
\end{equation}
Below, we denote by ${\mathcal S}$ the collection of  {local maxima}
$u \in {\mathbb S}^{d-1}$ of 
\eqref{eq:expansion:nablaV} satisfying 
$\Theta(u) >0$. Given these objects, we also require the following properties:

\begin{enumerate}
\item[({\textbf{B3}})] 
For any $a>0$ small enough, there exists a basin of attraction ${\mathcal B}_{a}\subset 
\{u \in {\mathbb S}^{d-1} : \Theta(u) >0  \}$ 
such that
\vskip 2pt

\begin{enumerate}
\item[({\textbf{B3-a}})]
For all $r>0$, we can find $T_{0}:=T_{0}(r)>0$ {(which may depend on $r$)} satisfying  
\begin{equation*}
\forall t \geq T_{0}, \ \forall u \in {\mathcal B}_{a}, \quad 
\textrm{\rm dist}\bigl(\phi_{t}^u,{\mathcal S} \bigr) \leq r;
\end{equation*}
\item[({\textbf{B3-b}})] \textit{For all $r>0$,
there exists $r'>0$ such that, for any $u \in {\mathbb S}^{d-1}$ with $\textrm{\rm dist}(u,{\mathcal S}) < r'$, $u$ belongs to ${\mathcal B}_{a}$ and  
$(\phi_{t}^u)_{t \geq 0}$ 
remains at distance less than $r$ of ${\mathcal S}$
\footnote{The reader must be aware of the fact that
the second part of the assumption 
({\textbf{B3-b}}), namely
$(\phi_{t}^u)_{t \geq 0}$ 
remains at distance less than $r$ of ${\mathcal S}$, 
 follows from 
the first part of the assumption itself and from ({\textbf{B3-a}}).
Indeed, for a given $r>0$, we can choose $T_{0}$ as in 
({\textbf{B3-a}}) and then choose, by the first part of {({\textbf{B3-b}})}, $r'$ small enough such that
$u$ belongs to ${\mathcal B}_{a}$, which shows that 
$\textrm{\rm dist}\bigl(\phi_{t}^u,{\mathcal S} \bigr) \leq r$
for $t \geq T_{0}$. In order to guarantee that the same holds true 
for $t \in [0,T_{0}]$, we notice that, since ${\mathcal S}$ is preserved by 
$\phi$, we have 
$\sup_{t \in [0,T_{0}]}
\textrm{\rm dist}(\phi_{t}^u,{\mathcal S}) \leq C_{T_{0}} 
\textrm{\rm dist}(u,{\mathcal S})$ for a constant $C_{T_{0}}$ depending on 
$T_{0}$ but independent of $u$.};}
\item[({\textbf{B3-c}})] \textit{For any $a>0$ small enough,  {$\Theta^{-1}( [a,3a])$}
is included in ${\mathcal B}_{a}$.}
\end{enumerate}
\item[({\textbf{B4}})] \textit{The 
set ${\mathcal L}$ of local minima of the function 
$\Theta$ on $\{\Theta >0\}$
is finite (it may be empty). 
Moreover, }
\vskip 2pt

\begin{enumerate}
\item[({\textbf{B4-a}})]
\textit{For each local minimum $u_{w} \in {\mathcal L}$, 
there exists an open ball (on the sphere) $B(u_{w},e_{w}):=\{Êu \in {\mathbb S}^{d-1} :  {\vert u - u_{w} \vert} < e_{w}\}$, 
for some $e_{w}>0$, on which the function $\Theta$ is twice continuously differentiable and 
uniformly convex, meaning that the Hessian of $\Theta$ on the sphere, denoted by $\HTheta$, 
satisfies, for any $u \in B(u_{w},e_{w})$ and any $v \in T_{u} {\mathbb S}^{d-1}$, 
$v \cdot \HTheta(u) v \geq c \vert v \vert^2$, the inner product being regarded as the inner product  {on} 
$T_{u} {\mathbb S}^{d-1}$;}
\item[({\textbf{B4-b}})]
\textit{For any $a>0$ small enough
and
for each local minimum $u_{w} \in {\mathcal L}$, 
there exists $a_{w}>0$ such that $a_{w} < \inf_{\vert u - u_{w} \vert =e_{w}} \Theta(u)$
and 
\begin{equation*}
\bigl\{u \in {\mathbb S}^{d-1} : \Theta(u) \geq a \bigr\} \setminus 
\biggl( \bigcup_{u_{w} \in {\mathcal L}}
\bigl\{ u \in B(u_{w},e_{w}) :  \Theta(u) < a_{w}\bigr\} \biggr) \subset {\mathcal B}_{a}.
\end{equation*}
}
\end{enumerate}
\end{enumerate}

 {Assumptions $(\textbf{B3})$ and $(\textbf{B4})$ 
provide a strong form of asymptotic stability.} Condition $(\textbf{B3-a})$  
says that
${\mathcal B}_{a}$ is a subset of $\{\Theta >0 \}$
such that the distance between the flow $\phi$  and the set of local maxima 
${\mathcal S}$ converges to $0$ as time increases, uniformly in 
the initial condition $u \in {\mathcal B}_{a}$. 
Condition $(\textbf{B3-b})$ says that the neighborhood of ${\mathcal S}$ is included in 
${\mathcal B}_{a}$ while $(\textbf{B3-c})$ says that
the neighborhood of $\{ \Theta =  {2a} \}$, {for $a$ small, is   {also} in ${\mathcal B}_{a}$ (whence the dependence on $a$)}.  
Lastly, $(\textbf{B4-a})$ says that the local minima of $\Theta$ on $\{\Theta >0\}$ are located at bottoms of uniformly convex 
wells 
while $(\textbf{B4-b})$ guarantees that, starting sufficiently far away from the local minima (now including $0$ as a local minimum),
the flow is uniformly attracted to ${\mathcal S}$ (as it starts from ${\mathcal B}_{a}$ and ${\mathcal B}_{a}$ is uniformly
attracted to ${\mathcal S}$, see 
$(\textbf{B3-a})$). 

If useful, we remind the reader of the definition of $\HTheta$ in 
(\textbf{\bf B4-a}). For $u \in {\mathbb S}^{d-1}$, $\HTheta(u)$
is given by the covariant derivative of $\DTheta(u)$. On the sphere, the covariant derivative may be easily computed as the 
orthogonal projection of the (standard) Euclidean derivative onto $T_{u} {\mathbb S}^{d-1}$. In other words, 
if $\Theta$ extends to a local neighborhood of the sphere, then, 
with the same notation as in 
\eqref{eq:DTheta}, 
$\HTheta(u) = (I_{d} - u \otimes u) \nabla [\DTheta(u)]$, where $\nabla$ is the standard Euclidean derivative. In turn, 
$\DTheta(u)$ reads as $\DTheta(u)=
 (I_{d} - u \otimes u) \nabla \Theta(u)$
 and we end up with the formula
\begin{equation}
\label{eq:Hess:theta}
\HTheta(u) = (I_{d} - u \otimes u) \bigl[\nabla^2 \Theta(u)
- (u \cdot \nabla \Theta(u)) I_{d} \bigr],
\end{equation} 
where $\nabla^2$ is the standard Hessian matrix in ${\mathbb R}^d$. 
In this regard, it is worth noticing that the convexity property in \textbf{(B4-a)} may be stated in other ways without any use 
of the Hessian tensor. Still, we feel it simpler to use this definition in the text
and we let the interest reader refer to standard references on \textit{geodesic convexity}
for more details on the subject. In the special case of the sphere, the reader may find a complete overview of 
\textit{geodesic convexity} (including the definition that we use here) in \cite{Ferreira2014}. 
{Actually},
in the analysis below, we invoke convexity in two forms. {Of course}, we make use of the lower bound on the Hessian  tensor, 
as it is stated in \textbf{(B4-a)}. We also claim  {(and it is useful for us)} that we can find a constant $c'>0$  such that, for each local minimum $u_{w} \in {\mathcal L}$, 
for any $u \in B(u_{w},e_{w})$, 
\begin{equation}
\label{eq:convexity:2}
\bigl\vert \DTheta(u) \bigr\vert^2 \geq c' \bigl( \Theta(u) - \Theta(u_{w}) \bigr).
\end{equation}
Here the constant $c$' depends on $c$ in (\textbf{\bf B4-a}) and of the Lipschitz constant 
in 
\eqref{eq:C11}. Property \eqref{eq:convexity:2} is completely standard in the Euclidean setting. 
The proof on the sphere is similar, replacing \textit{Euclidean straight lines}
by \textit{geodesics on the sphere} (whence the notion of \textit{geodesic convexity}): We refer for instance to \cite[Proposition 8]{Ferreira2014} for the details.

\subsection{Example in dimension $d=2$}
\label{subse:2:2}
We here provide a quite generic example of a $2$-dimensional potential $V$ satisfying the above 
requirements. Take indeed 
\begin{equation*}
V(x) = g(\theta) \vert x \vert^{1+\alpha}, \quad x \in {\mathbb R}^2, 
\end{equation*}
where $\theta$ is the angle (or argument) in the polar decomposition of $x$ and 
$g$ is a function 
from ${\mathbb R}/(2 \pi {\mathbb Z})$ into $\RR_{+}$ 
which is positive on an open interval 
$U = (\theta_{0},\theta_{1}) + 2 \pi {\mathbb Z}$
of 
the torus, where $-\pi < \theta_{0} < \theta_{1} \leq \pi$, and which is $0$ outside. On the closure of $U$, $g$ extends into a twice continuously differentiable function, the extension of $g$ and its derivative
vanishing at the boundary of $U$. In particular,
$g$ is ${\mathcal C}^{1,1}$ on the torus.  

Clearly, for $x \not =0$,  
\begin{equation}
\label{eq:formula:2d:V}
\begin{split}
&\nabla V(x) = (1+ \alpha) g(\theta) \vert x\vert^{\alpha} u_{{\rho}}+
g'(\theta) \vert x \vert^{\alpha} u_{\theta}, 
\\
&\nabla^2 V(x) = (1+ \alpha) \alpha g(\theta) \vert x\vert^{\alpha-1} u_{{\rho}}
\otimes u_{{\rho}} +
 \alpha g'(\theta) \vert x\vert^{\alpha-1} u_{{\rho}} \otimes u_{\theta}
\\
&\hspace{50pt}+
\alpha g'(\theta) \vert x \vert^{\alpha-1} u_{\theta} \otimes u_{{\rho}}
+
\bigl( g''(\theta) + (1+\alpha) g(\theta) \bigr) \vert x \vert^{\alpha-1} u_{\theta} \otimes u_{\theta}, 
\\
&\Delta V(x) = (1+\alpha)^2  g(\theta) \vert x \vert^{\alpha-1} + g''(\theta) \vert x \vert^{\alpha-1}. 
\end{split}
\end{equation}
where
\begin{center}
 $u_{{\rho}}= \left(\begin{array}{c} \cos(\theta) \\ \sin(\theta) \end{array} \right)$ \quad and \quad $u_{\theta} = 
\left(\begin{array}{c} - \sin(\theta) \\ \cos(\theta) \end{array} \right)$.
\end{center}
So, for those $x$ whose argument is between $\theta_{0}$ and $\theta_{1}$, $\Delta V(x) \geq 0$ if and only if (identifying $g$ with its extension to $[\theta_{0},\theta_{1}]$)
\begin{equation*}
(1+\alpha)^2 g(\theta) +g''(\theta) \geq 0.
\end{equation*}
For $\theta>\theta_{0}$ close to $\theta_{0}$, 
$g(\theta) \approx \tfrac12 (\theta - \theta_{0})^2 g''(\theta_{0})$. Therefore, if 
$g''(\theta_{0})>0$, the above inequality is indeed true in the neighborhood of
$\theta_{0}$. The same holds in the neighborhood of $\theta_{1}$, which shows that (\textbf{A2})
and  
(\textbf{A3})  are indeed satisfied. Condition (\textbf{A4}) is easily satisfied as well 
with 
$L(x) = \min[(\theta - \theta_{0})_{+},(\theta_{1}-\theta)_{+}]$
and  $p=1$.

It now remains to discuss Assumption (\textbf{B}). (\textbf{B1}) is true with $p=1$  {possibly increasing $C_0$ }and 
(\textbf{B2}) holds with 
$\Theta(u) = g(\theta)$
{(and so $\DTheta(u) =  g'(\theta) u_{\theta}$)}, where $\theta$ is the argument of 
$u \in {\mathbb S}^1$. In particular, 
\eqref{eq:ODE:2:2} rewrites in the form
\begin{equation*}
\phi_{t}^u = \left( \begin{array}{c}
\cos \bigl( \vartheta_{t}^{\theta} \bigr) 
\\
\sin \bigl( \vartheta_{t}^{\theta} \bigr) 
\end{array}
\right)
, 
\quad 
\dot{\vartheta}_{t}^\theta = g(\vartheta_{t}^\theta), \quad t \geq 0 \ ; \quad \vartheta_{0}^\theta = \theta. 
\end{equation*}

Now we call ${\mathcal S}$ the collection of local maxima of $g$ on $(\theta_{0},\theta_{1})$
and ${\mathcal L}$ the collection of local minima (also on $(\theta_{0},\theta_{1})$). For simplicity, we assume 
that ${\mathcal S}$ is a finite union of closed intervals of the form $\cup_{i=1}^n [m_{0}^i,m_{1}^i]$, 
with $m_{0}^i \leq m_{1}^i < m_{0}^{i+1}$. Each of these intervals may reduce to one point. 
Following (\textbf{B4}), we assume that ${\mathcal L}$ is finite (and may be empty) and that elements 
of 
${\mathcal L}$ (if not empty) are located at bottoms of uniformly convex wells; in the latter case, 
we enumerate the elements of ${\mathcal L}$ in the form $u_{1}< \cdots <u_{\ell}$, for $\ell \geq 1$,
and we denote by $([u_{j}-e_{j},u_{j}+e_{j}])_{1 \leq j \leq \ell}$ non-trivial disjoint intervals on which 
$g$ is uniformly convex. 
Moreover, we assume that the zeros of $g'$ on $(\theta_{0},\theta_{1})$ are 
restricted to ${\mathcal S} \cup {\mathcal L}$. 
For $a$ small enough such that {$3a$} is less than the minimum of $g$ over ${\mathcal L}$, 
we let 
\begin{equation*}
{\mathcal B}_{a} = \bigl\{ \theta \in (\theta_{0},\theta_{1}) : g(\theta) \geq  {a}
\bigr\} \setminus \biggl( \bigcup_{j=1}^{\ell} \bigl\{ \theta \in (u_{j}-e_{j},u_{j}+e_{j}) : g(\theta) < g(u_{j}) + a
\bigr\} \biggr). 
\end{equation*}
Obviously, we may choose $a$ small enough such that, for each $j=1,\cdots,\ell$, 
$g(u_{j} \pm e_{j}) > g(u_{j})+a$. The latter guarantees that the level set $\{ \theta \in (u_{j}-e_{j},u_{j}+e_{j})  : g(\theta)=g(u_{j})+a\}$ separates the bottom of the well from ${\mathcal B}_{a}$.
This shows (\textbf{B4}).
 
It remains to check 
(\textbf{B3}). 
For sure, ${\mathcal B}_{a}$ contains $g^{-1}({[a,3a]})$, which is 
(\textbf{B3-c}). 
Also, for $\theta \in {\mathcal B}_{a} \setminus {\mathcal S}$,
$\theta$ lies in an interval of the form $[u+\epsilon,m)$ or $(m,u-\epsilon]$, where $u$ is a local minimum 
of $g$ ($u$ being possibly equal to $\theta_{0}$ or $\theta_{1}$)
and $m$ is equal to some $m_{0}^i$ or $m_{1}^i$, for some $i \in \{1,\cdots,n\}$. 
Here, $\epsilon$ is a small enough positive real that merely accounts for the fact that the distance between ${\mathcal B}_{a}$
and ${\mathcal L} \cup \{\theta_{0},\theta_{1}\}$ is not zero. 
In case when the interval has the form 
$[u+\epsilon,m)$, $g$ must be increasing on it
and $g'$ has to be (uniformly) positive on any compact subinterval.
In particular, $(\vartheta_{t}^{u+\epsilon})_{t \geq 0}$
reaches any neighborhood of $m$ in finite time. Since 
$\vartheta_{t}^{u+\epsilon} \leq 
\vartheta_{t}^{\theta}$ for any $\theta \in [u+\epsilon,m)$ and for any $t \geq 0$, 
we deduce that $(\vartheta_{t}^{\theta})_{t \geq 0}$
reaches any neighborhood of $m$ in finite time, uniformly in 
$\theta \in [u+\epsilon,m)$ ({and then never comes back as otherwise $g$ would decrease}). 
Proceeding similarly 
when the interval has the form 
$(m,u-\epsilon]$, 
this shows 
({\textbf{B3-a}}). 
Condition
({\textbf{B3-b}})  {is obvious}.

\begin{remark}
For sure, we could address more general $2$-dimensional examples of the form 
$V(x) = g(\theta) \vert x \vert^{1+\alpha} + h(x) \vert x \vert^{1+\beta}$, 
by choosing $g$ as before and by choosing $h$ as a ${\mathcal C}^{1,1}$ function 
with a support of empty intersection with the closure of the $2$-dimensional 
sector induced by $\{ g >0\}$. In order to completely match
\eqref{eq:expansion:V},
we could also consider perturbations of $g$ of the form $\ell(x) g(\theta)$ where 
$\ell$ is a smooth function satisfying $\ell(0) \not =0$.

Numerical examples are given in Section \ref{se:7}. 
\end{remark}

\subsection{Higher dimensional example}
\label{subse:2:3}

{The natural extension of the above $2d$ case is 
\begin{equation}
\label{ex:2:3}
V(x) = g \bigl( \frac{x}{\vert x \vert} \bigr) \vert x \vert^{1+\alpha}, \quad x \in {\mathbb R}^d \setminus \{0\},
\end{equation}
for a function $g$ defined (say) on a neighborhood ${\mathcal U}$ of the sphere
${\mathbb S}^{d-1} := \{ x \in {\mathbb R}^d : \vert x \vert=1\}$.
In that case,
\begin{equation*}
\begin{split}
&\nabla V(x) = \vert x \vert^{\alpha} {\bigl( I_{d} - \frac{x}{\vert x \vert} \otimes \frac{ x}{\vert x \vert} \bigr) \nabla g \bigl( \frac{x}{\vert x \vert}
\bigr) } + (1+\alpha) g\bigl( \frac{x}{\vert x \vert} \bigr) \vert x \vert^{\alpha}
\frac{x}{\vert x \vert},
\end{split}
\end{equation*}
for $x \in {\mathbb R}^d \setminus \{0\}$, see
\eqref{eq:DTheta}. 
Very much in the spirit of 
\eqref{eq:formula:2d:V}
(see also \eqref{eq:Hess:theta}), a similar formula may be derived for the second order derivatives of $V$.
 We may deduce  
that (\textbf{\bf A1}), (\textbf{\bf A2}) and (\textbf{\bf A3}) are satisfied if $g$ takes non-negative values, is not identically zero
on ${\mathbb S}^{d-1}$, and is ${\mathcal C}^{1,1}$ on ${\mathcal U}$
and if there exists $a_{0}>0$ such that the following holds true for the $x$'s in ${\mathcal U}$ 
for which 
$\nabla^2 g(x)$ is well-defined and $g(x) \in (0,a_{0})$:
\begin{equation}
\label{eq:Laplace:spheric}
\begin{split}
\Delta V(x) 
&= \vert x \vert^{\alpha-1}
\biggl( \textrm{\rm Trace} \biggl[
\bigl( I_{d} - \frac{x}{\vert x \vert} \otimes \frac{x}{\vert x \vert}
\bigr) 
\Bigl[
\nabla^2 g \bigl(\frac{x}{\vert x \vert}\bigr)
- \bigl(\frac{x}{\vert x \vert} \cdot \nabla g( \frac{x}{\vert x \vert}) \bigr) 
 \Bigr]
\bigl( I_{d} - \frac{x}{\vert x \vert} \otimes \frac{x}{\vert x \vert}
\bigr) 
\biggr] 
\\
&\hspace{15pt} + (1+ \alpha) \bigl(d-1+\alpha\bigr) g\bigl( \frac{x}{\vert x \vert} \bigr) \biggr) 
\\
&\geq 0.
\end{split}
\end{equation}
Conditions (\textbf{\bf A4}) and (\textbf{\bf B1}) are satisfied 
if, for $x$ in the neighborhood of 
${\mathbb S}^{d-1}$, $g(x) = \ell(x)^{p+1} {\mathbf 1}_{\{x \in {\mathcal U}_{+}\}}$ 
for some  {$p \geq 1$}, 
where ${\mathcal U}_{+} \subset \R^d$ is an open set of non-empty intersection with ${\mathbb S}^{d-1}$,
and $\ell$ is a positive-valued twice continuously differentiable function on 
${\mathcal U}_{+}$, such that 
$\ell$, $\nabla \ell$ and $\nabla^2 \ell$ extend by continuity to the closure $\overline {\mathcal U}_{+}$
of ${\mathcal U}_{+}$, the extensions satisfying, for any 
$u \in \partial {\mathcal U}_{+} \cap {\mathbb S}^{d-1}$,
\begin{equation*}  
\ell(u) = 0, \quad 
\textrm{\rm grad} \, \ell(u) = 
\bigl( I_{d} - u \otimes u \bigr) \nabla \ell(u) \not =0.
\end{equation*}
 {Indeed, 
with such a choice, 
$g$ is ${\mathcal C}^{1,1}$ on ${\mathcal U}$. 
Moreover,}
by compactness of $\overline {\mathcal U}_{+} \cap {\mathbb S}^{d-1}$ and by continuity of $\nabla \ell$, we can find $\varepsilon >0$ small enough 
such that $\inf_{u \in \overline{\mathcal U}_{+} \cap {\mathbb S}^{d-1} : \textrm{\rm dist}(u,\partial {\mathcal U}_{+}) \leq \varepsilon}
\vert \textrm{\rm grad} \, \ell(u) \vert >0$. Then, 
$a_{0}:=\inf_{u \in \overline{\mathcal U}_{+} \cap {\mathbb S}^{d-1} : \textrm{\rm dist}(u,\partial {\mathcal U}_{+}) \geq \varepsilon}
\ell^{p+1}(u) >0$, and (\textbf{\bf A4}) and (\textbf{\bf B1})
hold true with $L(x) = \ell(x/\vert x \vert) {\mathbf 1}_{\{x/\vert x \vert \in {\mathcal U}_{+}\}}$
and for these values of 
$p$
and 
$a_{0}$.

%

{Within this framework},   
we recover 
\eqref{eq:expansion:V}
and
\eqref{eq:expansion:nablaV} 
in 
 (\textbf{\bf B2})
 with
\begin{equation*}
\begin{split}
\Theta(u) = g(u), \quad 
 \DTheta(u) =  \Bigl( I_{d} - u \otimes u \Bigr) \nabla g( u), \ u \in {\mathbb S}^{d-1} \ ; \quad \eta(x) = \eta'(x) =  0, \quad x \in {\mathbb R}^d.
\end{split}
\end{equation*}
Obviously, 
(\textbf{\bf B2}) holds true if 
$g$ is ${\mathcal C}^{1,1}$ on ${\mathcal U}$. 

In order to guarantee
  (\textbf{\bf B3})
  and 
   (\textbf{\bf B4}), we
  may proceed as in the $2d$ example. 
  We call ${\mathcal S}$ and ${\mathcal L}$ the sets of local maxima 
  and of local minima  
   of $\Theta$ on $\{\Theta >0\}$
   and we require that, on 
   $\{\Theta >0\} \setminus 
  ({\mathcal S} \cup {\mathcal L})$, 
  $\vert \DTheta \vert >0$. Moreover,  
   following (\textbf{\bf B4-a}),
  we require 
   that   
${\mathcal L}$ 
 is finite and that each element of ${\mathcal L}$ is located at the bottom of a uniformly convex well.
 As in dimension 2, we ask ${\mathcal S}$ to write as the finite union of closed disjoint connected subsets (the value of the local maximum 
 on each connected component being hence constant).
As the proofs of (\textbf{\bf B3})
  and 
   (\textbf{\bf B4}) are here similar to the $2d$ case, we just give a sketch of them. 
The idea is as follows. Whenever the flow starts from 
a compact subset ${\mathcal K} \subset \{\Theta>0\} \setminus ({\mathcal L} \cup {\mathcal S})$,
it stays within ${\mathcal K}$ for a finite time only, uniformly with respect to the initial point 
(as long as the latter is in ${\mathcal K}$). Indeed, 
$(\ud/\ud t)[\Theta(\phi_{t})] = \vert \DTheta(\phi_{t}) \vert^2$
remains lower bounded by a positive constant, only depending on ${\mathcal K}$, as long as $\phi_{t}$
remains in ${\mathcal K}$. Hence, 
for any $r>0$, the distance between the flow and ${\mathcal S}$ becomes less than $r$ in finite time
(uniformly with respect to the initial point in ${\mathcal K}$)
as otherwise $\Theta$ would blow up along $\phi$.
In order to complete the proof, it suffices to show that, for a given $r>0$,
the flow remains at distance to ${\mathcal S}$ less than 
$r$
provided that $\textrm{\rm dist}(\phi_{0},{\mathcal S})$ is small enough. In fact, 
since ${\mathcal S}$ writes as the union of a finite number of closed disjoint connected sets, say 
${\mathcal S}_{1},\cdots,{\mathcal S}_{n}$, we just have to show that,
 for a given $r>0$, for any $i \in \{1,\cdots,n\}$, 
 $\phi$ remains at distance to ${\mathcal S}_{i}$ less than $r$ provided that 
$\textrm{\rm dist}(\phi_{0},{\mathcal S}_{i})$ is small enough.
For $r$ small enough, we have $\sup_{u : \textrm{\rm dist}(u,{\mathcal S}_{i}) = r}
 g \leq g({\mathcal S}_{i})-\varepsilon$, for some $\varepsilon >0$ depending on $r$. 
Choose now $r'>0$ such that 
$\inf_{u : \textrm{\rm dist}(u,{\mathcal S}_{i}) \leq r'}
 g \geq g({\mathcal S}_{i})-\varepsilon/2$.
 Then, whenever $\textrm{\rm dist}(\phi_{0},{\mathcal S}_{i})
 \leq r'$, we have $\Theta(\phi_{t}) \geq g({\mathcal S}_{i})- \varepsilon/2$
 for all $t \geq 0$. In particular, $\phi$ cannot go at distance to ${\mathcal S}_{i}$
 greater than or equal to $r$, as otherwise $\Theta(\phi)$ would pass below 
 $g({\mathcal S}_{i})- \varepsilon$.
 This provides a similar picture to the $2d$ case and we let the reader define 
 ${\mathcal B}_{a}$ accordingly.}

\subsection{Statements}
By Veretennikov \cite{Veretennikov}, we have the following proposition:

\begin{proposition}
Let Assumption \textbf{\bf A} be in force. Then, 
for all $\varepsilon >0$, the equation \eqref{eq:1:1}, with $X_{0}^{\varepsilon}=0$ as initial condition, has a unique  {(strong) }solution, which is denoted by $(X_{t}^{\varepsilon})_{t \geq 0}$.
\end{proposition}

Here are the two statements we prove below.

\begin{theorem}
\label{thm:2:1}
Under Assumption \textbf{\bf A}, there exist two constants $c>0$ and $\psi \in (0,1)$
and a collection of positive times $(t_{\varepsilon})_{\varepsilon >0}$
converging to $0$ with $\varepsilon$ 
 such that, whenever
$X_{0}^{\varepsilon} = 0$ for all $\varepsilon >0$, it holds that  
\begin{equation*}
\liminf_{\varepsilon \searrow 0}
{\mathbb P}
\biggl(\forall 
t\in [t_{\varepsilon},  t_{\star}], 
\ 
V\bigl( X_{t}^\varepsilon\bigr)  \geq 
\Bigl( (1-\psi) c(t  - t_{\varepsilon})_{+} \Bigr)^{\tfrac{1}{1-\psi}}, 
\ 
g\bigl( X_{t}^{\varepsilon} \bigr) >0
\biggr) =1,
\end{equation*}
where 
$\displaystyle t_{\star} := \inf \Bigl\{t \geq 0 : \Bigl( (1-\psi) c t \Bigr)^{\frac{1}{1-\psi}} \geq \tfrac12 \Bigr\} \wedge 1$. 
\end{theorem}
\begin{theorem}
\label{thm:2:2}
Assume Assumptions \textbf{\bf A}
and \textbf{\bf B} are in force. 
Then, for given $\pi \in (0,1)$ and $r>0$, we can find 
a collection of positive times $(t_{\varepsilon})_{\varepsilon >0}$
converging to $0$ with $\varepsilon$
 together with a positive time ${\boldsymbol \epsilon} >0$ such that
\begin{equation*}
\liminf_{\varepsilon \searrow 0} {\mathbb P}
\biggl( \forall t \in [t_{\varepsilon},{\boldsymbol \epsilon}], \
\vert X_{t}^{\varepsilon} \vert >0, 
\ 
\textrm{\rm dist}\Bigl(\frac{X_{t}^{\varepsilon}}{\vert X_{t}^{\varepsilon} \vert},{\mathcal S}\Bigr) \leq r 
\biggr) \geq 1- \pi. 
\end{equation*}
\end{theorem}

In the text, we call 
a collection of positive times $(t_{\varepsilon})_{\varepsilon >0}$
converging to $0$ with $\varepsilon$
an \textit{infinitesimal time} (or \textit{infinitesimal sequence}).

The interpretation of Theorem \ref{thm:2:1} is pretty clear: With probability converging to $1$, the potential 
grows polynomially fast after some infinitesimal time $t_{\varepsilon}$ (at least up until time $t_{\star}$); meanwhile $g(X^{\varepsilon})$ remains positive. 
In other words, with probability converging to $1$, the particle goes away from the singularity while staying within the region $\{g >0\}$ (observe indeed that the particle cannot stay close to $0$ as the potential increases). 

Although the statement does not say anything about the behavior of the particle 
after time $t_{\star}$, we can guess it quite easily. When starting away from the singularity (say from 
$X_{0}^{\varepsilon}=x_{0}$ for a fixed $x_{0}$ with $\vert x_{0} \vert >0$), the solution of the 
SDE 
\eqref{eq:1:1}
converges to the solution of the ODE $\dot{x}_{t} = \nabla V(x_{t})$ (with $x_{0}$ as initial condition) up until it reaches some fixed neighbordhood ${\mathcal O}$ of the origin. In fact, we can easily choose 
the neighborhood ${\mathcal O}$ in such a way that the solution of the limiting ODE does not reach it,  
since the
potential cannot decrease along the limiting ODE. Hence, as $\varepsilon$ tends to $0$, the probability that the solution of  
the
SDE 
\eqref{eq:1:1} with $x_{0}$ as initial condition reaches ${\mathcal O}$ gets smaller and smaller. In other words, the particle stays away from the singularity and follows, asymptotically as $\varepsilon$ tends to $0$, the solution $(x_{t})_{t \geq 0}$ of the deterministic version of 
\eqref{eq:1:1} with $x_{0}$ as initial condition. 

{The interpretation of Theorem \ref{thm:2:2} is also obvious: With probability converging to $1$, and under the required conditions, 
the paths that the particle follows to escape from the singularity are locally (i.e., on a piece of time $[0,{\boldsymbol \epsilon}]$ with 
${\boldsymbol \epsilon}$ independent of $\varepsilon$) directed by the local maxima in ${\mathbb S}^{d-1} \cap \{ \Theta >0\}$ 
of $\Theta$. In other words, the particle follows directions that maximize the potential, or equivalently, steepest lines. Somehow, this is the multi-dimensional generalization of Bafico and Baldi's result. 
}
\vskip4pt

\noindent \textbf{Notations.}
Throughout the proofs, we specify when needed the parameters on which the various constants do depend.
We often write $C(\textbf{\bf A})$ (or $C(\textbf{\bf A},\textbf{\bf B})$) to 
stress the fact that the constant $C$ in hand may depend on the parameters in 
Assumption \textbf{\bf A}
(or in 
Assumptions
\textbf{\bf A}
and
\textbf{\bf B}).

\subsection{Submartingale dynamics of the potential}

Part of the analysis relies on the submartingale properties of the potential process $(V_{t}^{\varepsilon} := V(X_{t}^{\varepsilon}))_{t \geq 0}$. The latter is pretty easy to check whenever the potential function $V$ is convex. When $V$ is not convex, we need to prove the submartingale property for a perturbation of the potential process, which is the precise purpose of the next two lemmas. 

\begin{lemma}
\label{lem:supermartingale}
Call $(X_{t}^{\varepsilon})_{t \geq 0}$ the solution of \eqref{eq:1:1} with $X_{0}^{\varepsilon}=0$ as initial condition and let
\begin{equation*}
V_{t}^{\varepsilon} := V(X_{t}^{\varepsilon}),
\quad 
\nabla V_{t}^{\varepsilon} := \nabla V(X_{t}^{\varepsilon}), 
\quad 
R_{t}^{\varepsilon} := \vert X_{t}^{\varepsilon} \vert, 
\quad t \geq 0.
\end{equation*}
Letting for any $\varepsilon >0$
\begin{equation*}
\kappa^{\varepsilon} := \inf \Bigl\{ t\geq 0 : R_{t}^{\varepsilon}  \geq  {\varepsilon^{\tfrac{2}{1+\beta}+\tfrac{\beta-\alpha}{2}}}  \Bigr\},
\end{equation*}
there {exist two positive constants $\eta:=\eta(\textbf{\bf A})$ and $\varepsilon_{\star}:=\varepsilon_{\star}(\textbf{\bf A})$}, such that, {for any 
$\varepsilon \in (0, \varepsilon_{\star}]$}, the process
\begin{equation*}
\biggl(  {V_{t \wedge \kappa^{\varepsilon}}^{\varepsilon}
+ \eta \bigl( R_{t \wedge \kappa^{\varepsilon}}^{\varepsilon} \bigr)^{1+\alpha}}
-  \int_{0}^{t \wedge \kappa^{\varepsilon}} \vert \nabla V_{s}^{\varepsilon}
\vert^2 \ud s   \biggr)_{t \geq 0} 
\end{equation*}
is a sub-martingale. 

If needed, we can choose $\eta$ in such a way that the function $x \mapsto V(x) + \tfrac12 \eta \vert x \vert^{1+\alpha}$ is non-negative on the ball of center $0$ and radius $1$. 
\end{lemma}


\begin{proof}
We first notice that, for a given $\varepsilon$, 
the law of $(X_{t}^{\varepsilon})_{0 \leq t \leq T}$
is equivalent  to the law of $(\varepsilon B_{t})_{0 \leq t \leq T}$, for any $T>0$. 
In particular, for any $\varepsilon >0$,
\begin{equation*}
{\mathbb P} \Bigl( \forall t >0, \quad X_{t}^{\varepsilon} \neq 0 \Bigr) = 1. 
\end{equation*}
Hence, for any $t >0$, it makes sense to expand $(V_{s}^{\varepsilon})_{s \geq t}$ since 
$V$ is ${\mathcal C}^{1,1}$ on any compact subset of $\RR^d \setminus \{0\}$. Equivalently, 
it makes sense to write $\ud V_{t}^{\varepsilon}$, for $t>0$; then, the only difficulty is integrate those microscopic variations between $0$ and some positive time. We shall come back to this point when necessary. 
\vskip 4pt

\textit{First step.}
We start with the following two simple computations:
\begin{equation}
\label{eq:V:R2}
\begin{split}
&\ud V_{t}^{\varepsilon} = \Bigl( \vert \nabla V_{t}^{\varepsilon} \vert^2  + 
\tfrac12
\varepsilon^2 \Delta V_{t}^{\varepsilon} \Bigr) \ud t + \varepsilon \nabla V_{t}^{\varepsilon} \cdot \ud B_{t},
\quad t >0,
\\
&\ud \bigl( R_{t}^{\varepsilon}
\bigr)^2 = \Bigl( 2 X_{t}^{\varepsilon} \cdot \nabla V_{t}^{\varepsilon}
+ d \varepsilon^2 \Bigr) \ud t
+ 2 \varepsilon X_{t}^{\varepsilon} \cdot \ud B_{t}, \quad t >0.
\end{split}
\end{equation}
In particular, It\^o's formula, with the function $f(x)=x^{\frac{1+\alpha}{2}}$ yields
\begin{equation*}
\begin{split}
\ud \bigl( R_{t}^{\varepsilon}
\bigr)^{1+\alpha} &= \Bigl( (1+\alpha) \bigl(R_{t}^{\varepsilon}\bigr)^{\alpha} \frac{X_{t}^{\varepsilon}}{R_{t}^{\varepsilon}} \cdot \nabla V_{t}^{\varepsilon}
+ \varepsilon^2 (1+\alpha)  \frac{d+\alpha-1}{2}   \bigl( R_{t}^{\varepsilon} \bigr)^{\alpha-1}\Bigr) \ud t
+\varepsilon (1+\alpha)   \bigl( R_{t}^{\varepsilon} \bigr)^{\alpha} \frac{X_{t}^{\varepsilon}}{R_{t}^{\varepsilon}} \cdot \ud B_{t}.
\end{split}
\end{equation*}
Hence, for any constant $\eta>0$, 
\begin{align}
&\ud \Bigl[ V_{t}^{\varepsilon} + \eta \bigl( R_{t}^{\varepsilon} \bigr)^{1+\alpha} 
\Bigr]
=\Bigl(  \vert \nabla V_{t}^{\varepsilon} \vert^2  + \tfrac12
\varepsilon^2 \Delta V_{t}^{\varepsilon}
+
\eta (1+\alpha) \bigl( R_{t}^{\varepsilon} \bigr)^{\alpha} \frac{X_{t}^{\varepsilon}}{R_{t}^{\varepsilon}} \cdot \nabla V_{t}^{\varepsilon}
+ \varepsilon^2
\eta (1+\alpha) 
\frac{d+\alpha-1}{2} 
\bigl( R_{t}^{\varepsilon} \bigr)^{\alpha-1} 
\Bigr) \ud t \nonumber
\\
&\hspace{15pt}
+
\varepsilon \Bigl( \nabla V_{t}^{\varepsilon}
+ 
\eta(1+\alpha)  \bigl( R_{t}^{\varepsilon} \bigr)^{\alpha} \frac{X_{t}^{\varepsilon}}{R_{t}^{\varepsilon}}
\Bigr) \cdot \ud B_{t}. 
\label{eq:proof:submartingale}
\end{align}
\vskip 4pt

\textit{Second step.}
{Notice from (\textbf{A2}) that} there exists a constant $C$ such that, for 
{$t < \kappa^{\varepsilon}$}, 
\begin{equation*}
\eta (1+\alpha) \bigl( R_{t}^{\varepsilon} \bigr)^{\alpha} \frac{X_{t}^{\varepsilon}}{R_{t}^{\varepsilon}} \cdot \nabla V_{t}^{\varepsilon} \geq - \eta C  \bigl( R_{t}^{\varepsilon} \bigr)^{{\alpha+\beta}} \geq -   \eta C \varepsilon^2 \varepsilon^{\tfrac{(1+\beta)(\beta-\alpha)}2} \bigl( R_{t}^{\varepsilon} \bigr)^{\alpha-1},
\end{equation*}
{where we used the fact that, for $t < \kappa^{\varepsilon}$, 
$(R_{t}^{\varepsilon})^{\alpha+\beta} = (R_{t}^{\varepsilon})^{\alpha-1}
(R_{t}^{\varepsilon})^{1+\beta} \leq \varepsilon^2 \varepsilon^{\frac{(1+\beta)(\beta-\alpha)}2} (R_{t}^{\varepsilon})^{\alpha-1}$.}
{Now, by assumption (\textbf{A2})  on the shape of the potential function},
we can choose {$\eta:=\eta(\textbf{A})$ large enough {and $\varepsilon_{\star}:=\varepsilon_{\star}(\textbf{\bf A})$ small enough} such that, for
$\varepsilon \leq \varepsilon_{\star}$} and for
 $t>0$,
\begin{equation*}
\tfrac12
 \varepsilon^2 \Delta V_{t}^{\varepsilon}
 +
\eta (1+\alpha) \bigl( R_{t}^{\varepsilon} \bigr)^{\alpha} \frac{X_{t}^{\varepsilon}}{R_{t}^{\varepsilon}} \cdot \nabla V_{t}^{\varepsilon} 
+\varepsilon^2 \eta (1+\alpha) \frac{d+\alpha-1}{2} \bigl( R_{t}^{\varepsilon} \bigr)^{\alpha-1}  
\geq 0.
\end{equation*}
We deduce that  
\begin{equation*}
\biggl(  V_{t \wedge \kappa^{\varepsilon}}^{\varepsilon}
+ \eta \bigl( R_{t \wedge \kappa^{\varepsilon}}^{\varepsilon} \bigr)^{1+\alpha}
-  \int_{0}^{t \wedge \kappa^{\varepsilon}} \vert \nabla V_{s}^{\varepsilon}
\vert^2 \ud s   \biggr)_{t > 0} 
\end{equation*}
is a sub-martingale. We easily include time $t=0$ by a continuity argument.
Choosing $\eta$ such that 
$V(x) + \tfrac12 \eta \vert x \vert^{1+\alpha} 
 \geq 0$ for $\vert x \vert \leq 1$ (which is possible 
since $V(x) \leq C_{0} \vert x \vert^{1+\alpha}$ for $\vert x \vert \leq 1$),
we complete the proof.
\end{proof}

In fact, and this is a crucial point in the rest of the analysis, we can do better when we restart away from the boundary. 

\begin{lemma}
\label{lem:supermartingale:2}
For a given $r_{0}>0$, assume that, for any $\varepsilon >0$, 
the starting point $X_{0}^{\varepsilon}$ of the diffusion process is located 
at distance $R_{0}^{\varepsilon} \geq  \varepsilon^{\frac2{(1+\alpha)}} r_{0}$ from the origin. 

Then, there exist {two positive constants $r_{\star}:=r_{\star}(\textbf{\bf A})$ and $\varepsilon_{\star}=\varepsilon_{\star}(\textbf{\bf A})$} such that, for $\varepsilon \in (0, \varepsilon_{\star}]$, we can find a positive constant $\eta_{\varepsilon}:=\eta(\textbf{\bf A},\varepsilon)$, depending on $\varepsilon$, such that, for $r_{0} \geq r_{\star}$, the process
\begin{equation*}
\biggl(  V_{t \wedge \kappa^{\varepsilon,\prime}}^{\varepsilon} + \eta_\varepsilon \bigl( R_{t \wedge \kappa^{\varepsilon,\prime}}^{\varepsilon} \bigr)^{1+\alpha}
  - \frac12 \int_{0}^{t \wedge \kappa^{\varepsilon,\prime}} \vert \nabla V_{s}^{\varepsilon}
\vert^2 \ud s
\biggr)_{t \geq 0}
\end{equation*}
is a sub-martingale, where we used here the notations:
\begin{equation*}
\kappa^{\varepsilon,\prime} := \inf \Bigl\{ t \geq 0 : R_{t}^{\varepsilon} \leq \tfrac12 \varepsilon^{\tfrac{2}{1+\alpha}} r_{\star} \Bigr\} \wedge \inf \Bigl\{ t \geq 0 : R_{t}^{\varepsilon} \geq \varepsilon^{\tfrac{2}{1+\beta}+\tfrac{\beta-\alpha}{2}}\Bigr\}. 
\end{equation*}
Moreover, we can choose $\eta_{\varepsilon}$ {converging to $0$ with $\varepsilon$}
such that the function $x \mapsto V(x) + \tfrac12 \eta_{\varepsilon}
\vert x \vert^{1+\alpha}$ is positive on the ball of center $0$ and of radius $\varepsilon^{\frac{2}{1+\beta}+\frac{\beta-\alpha}{2}}$.
\end{lemma}

Notice that, for $\varepsilon \in (0,1)$, 
\begin{equation}\label{eq:comparison_barriers}
\begin{split}
\varepsilon^{\tfrac{2}{1+\alpha}}
=
\varepsilon^{\tfrac{2}{1+\beta}+\tfrac{\beta-\alpha}{2}}
\varepsilon^{\tfrac{2 (\beta-\alpha)}{(1+\alpha)(1+\beta)}-\tfrac{\beta-\alpha}{2}}
&=
\varepsilon^{\tfrac{2}{1+\beta}+\tfrac{\beta-\alpha}{2}}
\varepsilon^{(\beta-\alpha) \tfrac{4 - (1+\alpha)(1+\beta)}{2(1+\alpha)(1+\beta)}}
= o \bigl( \varepsilon^{\tfrac{2}{1+\beta}+\tfrac{\beta-\alpha}{2}} \bigr),
\end{split}
\end{equation}
where we used the Landau notation $o(\cdot)$ in the last line. This proves that, for $\varepsilon$ small enough,
the stopping time $\kappa^{\varepsilon,\prime}$ makes sense. 

\begin{proof}
The proof is pretty much the same as the proof of Lemma 
\ref{lem:supermartingale}.
The point is to investigate the sign of the drift in the semi-martingale expansion 
\eqref{eq:proof:submartingale}. In this new framework, we can no longer choose 
$\eta$ as large as needed and the main difficulty
for lower bounding the drift in 
\eqref{eq:proof:submartingale}
 is to guarantee that the following condition is indeed satisfied
 {for $\varepsilon$ small enough}
{({notice that} the coefficient $\tfrac18$ in front of $\vert \nabla V_{t}^{\varepsilon} \vert^2$ is for later use, here we could work with 
$\tfrac12$ instead of $\tfrac18$)}
\begin{equation*}
\tfrac18 \vert \nabla V_{t}^{\varepsilon} \vert^2
+\tfrac12
\varepsilon^2 \Delta V_{t}^{\varepsilon}
+
\eta (1+\alpha) \bigl( R_{t}^{\varepsilon} \bigr)^{\alpha} \frac{X_{t}^{\varepsilon}}{R_{t}^{\varepsilon}} \cdot \nabla V_{t}^{\varepsilon}
+ 
\varepsilon^2
\eta (1+\alpha) 
\frac{d+\alpha-1}{2} 
\bigl( R_{t}^{\varepsilon} \bigr)^{\alpha-1} 
\geq 0,
\end{equation*}
with the constraint that $\eta$ becomes small with $\varepsilon$. Equivalently,
we want to check, for any $x$ with
$\vert x \vert \in [\tfrac12 \varepsilon^{\frac{2}{1+\alpha}} r_{{\star}},{\varepsilon^{\frac{2}{1+\beta}+\frac{\beta-\alpha}{2}}}]$ ({the value of $r_{\star}$ being made preciser later on in the proof}), the inequality 
\begin{equation}
\label{eq:condition:submartingale:0}
\tfrac18 \vert \nabla V(x) \vert^2
+\tfrac12
\varepsilon^2 \Delta V(x)
+
\eta (1+\alpha) \vert x \vert^{\alpha} \frac{x}{\vert x \vert} \cdot \nabla V(x)
+ 
\varepsilon^2
\eta (1+\alpha) 
\frac{d+\alpha-1}{2} 
\vert x \vert^{\alpha-1} 
\geq 0,
\end{equation}
{holds} for $\varepsilon$ small enough.

To start with, we may proceed as in the second step of the proof of Lemma
\ref{lem:supermartingale} to prove that
\begin{equation*}
\eta (1+\alpha) \vert x \vert^{\alpha} \frac{x}{\vert x \vert} \cdot \nabla V(x) \geq - \eta C  \varepsilon^2 \varepsilon^{\tfrac{(1+\beta)(\beta-\alpha)}2} \vert x \vert^{\alpha-1},
\end{equation*}
for $\vert x \vert \leq  {\varepsilon^{\frac{2}{1+\beta}+\frac{\beta-\alpha}{2}}}$ and for {a constant $C:=C(\textbf{A})$. This says that, for $\varepsilon \leq \varepsilon_{\star}(\textbf{\bf A})$}
and whatever {the values of $\eta$ and $r_{\star}$}, 
for 
$\vert x \vert \in [\tfrac12 \varepsilon^{\frac{2}{1+\alpha}} r_{\star}, \varepsilon^{\frac{2}{1+\beta}+\frac{\beta-\alpha}{2}}]$, 
\begin{equation*}
\eta (1+\alpha) \vert x \vert^{\alpha} \frac{x}{\vert x \vert} \cdot \nabla V(x) 
+
\varepsilon^2
\eta (1+\alpha) 
\frac{d+\alpha-1}{4} 
\vert x \vert^{\alpha-1} 
\geq 0.
\end{equation*}
Hence, in order to 
 to prove 
\eqref{eq:condition:submartingale:0}, it suffices to prove
\begin{equation}
\label{eq:condition:submartingale}
\tfrac18 \vert \nabla V(x) \vert^2
+\tfrac12
\varepsilon^2 \Delta V(x)
+ 
\varepsilon^2
\eta (1+\alpha) 
\frac{d+\alpha-1}{4} 
\vert x \vert^{\alpha-1} 
\geq 0,
\end{equation}
for  {$x$ as therein (and with an appropriate value of $r_{\star}$)}.

We shall distinguish three cases in order to check condition \eqref{eq:condition:submartingale}. 
\vskip 4pt

\textit{First case.}
We first work on the domain $g^{-1}(0,a_{0})=\{ g \in (0,a_{0})\}$, for $a_{0}$ as in the assumption. 
By assumption, we know that $\Delta V(x) \geq 0$, for $g(x) \in (0,a_{0})$. In that case, 
\eqref{eq:condition:submartingale} is obvious. 
\vskip 4pt

\textit{Second case.}
Take now $x$ such that $g(x) \geq a_{0}$ and $\vert x \vert \geq \tfrac12 \varepsilon^{\frac2{1+\alpha}} r_{\star}$. Then, by assumption, we know that 
$\vert \nabla V(x) \vert^2 \geq c_{0} \vert x \vert^{2 \alpha}$. Meanwhile we have, for $\vert x \vert \geq \tfrac12 \varepsilon^{\frac{2}{1+\alpha}} r_{\star}$,
\begin{equation*}
\varepsilon^2 \vert \Delta V(x) \vert \leq C_{0} \Bigl( \frac{2 \vert x \vert}{r_{\star}} \Bigr)^{1+\alpha} \vert x \vert^{\alpha-1} = C_02^{1+\alpha} \frac{1}{{r_{\star}}^{1+\alpha}} \vert x\vert^{2\alpha} .
\end{equation*}
So choosing {$r_{\star}:=r_{\star}(\textbf{\bf A})$} large enough, we get
\begin{equation*}
\tfrac18 \vert \nabla V(x) \vert^2
+
\tfrac12
\varepsilon^2 \Delta V(x)
\geq 0,
\end{equation*}
which obviously suffices to complete the proof.
\vskip 4pt

\textit{Third case.} It remains to see what happens for $x$ such that $g(x) =0$ and 
 {$\vert x \vert \leq \varepsilon^{\frac{2}{1+\beta}+\frac{\beta-\alpha}{2}}$}. Observe that, so far, we have not used $\eta$ yet; we make its role explicit in this step.
{In this regard, we invoke Lemma \ref{lem:boundary:zero:measure} below,
from which  
 the boundary of the set $\{Êg = 0\}$ has zero Lebesgue measure.}
 Hence, it suffices to check 
\eqref{eq:condition:submartingale} on a full Lebesgue subset, {because $X^\varepsilon_t$ has absolutely continuous law with respect to the Lebesgue measure. Therefore,} we may assume that $x$ belongs to the interior of the set $g^{-1}(0)$, in which case we have
\begin{equation*}
\varepsilon^2 \vert \Delta V(x) \vert \leq C_{0} \varepsilon^2 \vert x \vert^{\beta-1}.
\end{equation*}
So, for $x \in \overset{\circ}{g^{-1}(0)}$ (which denotes the interior of $g^{-1}(0)$), with ${\vert x \vert \leq \varepsilon^{\frac{2}{1+\beta}+\frac{\beta-\alpha}{2}}}$, we have
\begin{equation*}
\begin{split}
\tfrac12
\varepsilon^2 \Delta V(x) 
+ \varepsilon^2 \eta 
\bigl(1 + \alpha \bigr)
\frac{d+\alpha-1}{2}
 \vert x \vert^{\alpha-1}
&\geq - \tfrac12 C_{0} \varepsilon^2 \vert x \vert^{\beta-1}
+ \varepsilon^2 \eta 
\bigl(1 + \alpha \bigr) 
\frac{d+\alpha-1}{2}
\vert x \vert^{\alpha-1}
\\
&\geq  \varepsilon^2 
\vert x \vert^{\alpha-1}
\Bigl( - \tfrac12 C_{0} {\varepsilon^{\tfrac{2(\beta-\alpha)}{1+\beta}+\tfrac{(\beta-\alpha)^2}{2}}}
+ \eta 
\bigl(1 + \alpha \bigr) 
\frac{d+\alpha-1}{2} \Bigr).
\end{split}
\end{equation*}
For a given {$\varepsilon$}, we can choose {$\eta:=\eta(\textbf{\bf A},\varepsilon)$} such that the right-hand side is non-negative.
Obviously we can choose $\eta$ as small as we want by choosing ${\varepsilon}$ as small as needed.  
Similarly, we can assume that  
the function $x \mapsto V(x) + \tfrac12 \eta \vert x \vert^{1+\alpha}$ is non-negative on the 
ball of center $0$ and of radius ${\varepsilon^{\frac{2}{1+\beta}+\frac{\beta-\alpha}{2}}}$. 
\vskip 4pt

\textit{Conclusion.} By combining the three cases, we get that 
\eqref{eq:condition:submartingale:0} is satisfied almost everywhere on 
the region $\{x : \vert x \vert \in 
[\tfrac12 \varepsilon^{\frac2{1+\alpha}} r_{\star}, {\varepsilon^{\frac{2}{1+\beta}+\frac{\beta-\alpha}{2}}}]\}$. 
Since $(X_{t}^{\varepsilon})_{t \geq 0}$ does not see sets of zero measure, we can duplicate the proof of Lemma  \ref{lem:supermartingale} to conclude. 
\end{proof}

\section{Close to zero}
\label{se:3}
The first part of our analysis is devoted to the study of the process $(X_{t}^{\varepsilon})_{t \geq 0}$ when $(V_{t}^{\varepsilon})_{t \geq 0}$ is less than $v_{0} \varepsilon^2$ for some $v_{0} >0$. 
Basically, we show that, as long as the potential
(or more precisely the perturbation of the potential as we considered in the previous section)
remains less than $v_{0} \varepsilon^2$, the diffusion process behaves like $(\varepsilon B_{t})_{t \geq 0}$.
This fact reads as a rewriting of the transition property exhibited in the one-dimensional 
setting in \cite{DelarueFlandoli}. In particular, it permits to prove next that, in infinitesimal time, 
the potential reaches values of order $\varepsilon^2$.

Using the same notations as in Lemmas \ref{lem:supermartingale} and 
\ref{lem:supermartingale:2}, we introduce the following three stopping times (for a given value of $\delta >0$):
\begin{equation}
\label{eq:3:1}
\begin{split}
&\tau^{\varepsilon}(v_{0})
= \inf \bigl\{ t \geq 0 : V_{t}^{\varepsilon} \geq v_{0} \varepsilon^2 \bigr\},
\\
&\nu^{\varepsilon}(v_{0}) = \inf \bigl\{ t \geq 0 : V_{t}^{\varepsilon} + \eta \bigl( R_{t}^{\varepsilon} \bigr)^{1+\alpha} \geq v_{0} \varepsilon^2 \bigr\},
\quad \nu^{\varepsilon,\prime}(v_{0}) = \inf \bigl\{ t \geq 0 : V_{t}^{\varepsilon} + \eta_{\varepsilon} \bigl( R_{t}^{\varepsilon} \bigr)^{1+\alpha} \geq v_{0} \varepsilon^2 \bigr\}.
\end{split}
\end{equation}
In the second line, $\eta$ is chosen as in the statement of Lemma \ref{lem:supermartingale}. 
In the third line, $\eta_{\varepsilon}$ is chosen as in the statement of Lemma \ref{lem:supermartingale:2}.

\subsection{BMO martingale}
In order to prove that $X^{\varepsilon}$ behaves like $(\varepsilon B_{t})_{t \geq 0}$ (as long as the perturbed potential remains small enough), we shall use Girsanov's theorem. Of course, the difficulty is that, the 
diffusion coefficient getting smaller and smaller with $\varepsilon$, it becomes more and more difficulty to  bound (from above and from below) the density of the law of $X^{\varepsilon}$ with respect to 
the law of $\varepsilon B$. This is precisely where the assumption that 
the potential remains small comes in: We manage to control the BMO norm (see below for the definition) of the martingale entering the Girsanov density in terms of the sole (perturbed) potential. 

In this regard, we recall that a martingale $(\int_{0}^t Z_{s}
\cdot dB_{s})_{t \geq 0}$ is said to be BMO 
if there exists a constant $K \geq 0$ such that, for any stopping time $\sigma$, it holds with probability 1: 
\begin{equation*}
\EE 
\biggl[ \int_{\sigma}^{\infty} \vert Z_{s} \vert^2 ds \, \big\vert \, {\mathcal F}_{\sigma}
\biggr] \leq K^2.
\end{equation*}
The smallest constant $K$ that achieves the above condition is called the BMO norm of 
$(\int_{0}^t Z_{s}
\cdot \ud B_{s})_{t \geq 0}$. We refer to the textbook \cite{Kazamaki} for an overview.

\begin{lemma}
\label{lem:BMO:perturbed:1}
Assume that, for any $\varepsilon >0$, $X_{0}^{\varepsilon} = 0$. Then,  {for any $\varepsilon \in (0,\varepsilon_\star)$ (with the notation of Lemma \ref{lem:supermartingale}), }for any $v_{0} >0$ and any $T>0$, 
the martingale $\bigl( \varepsilon^{-1}\int_{0}^{\nu^{\varepsilon}(v_{0}) \wedge \kappa^{\varepsilon} \wedge T \wedge t} \nabla V_{s}^{\varepsilon}\cdot \ud B_{s} \bigr)_{t \geq 0}$ has a BMO norm less than $\sqrt{v_{0}}$.
\end{lemma}

\begin{proof}
It suffices to observe that, from Lemma 
\ref{lem:supermartingale} (and with the same notations as therein), for any $t \geq 0$ and for any stopping time $\sigma$,  on the event $\{ \sigma < T \wedge \nu^{\varepsilon}(v_{0}) \wedge \kappa^{\varepsilon} \}$,
\begin{equation*}
{\mathbb E} \Bigl[ \Bigl( V_{T \wedge \nu^{\varepsilon}(v_{0}) \wedge \kappa^{\varepsilon} }^{\varepsilon}
+ \eta \Bigl( R_{T \wedge \nu^{\varepsilon}(v_{0})  \wedge \kappa^{\varepsilon}}^{\varepsilon} \Bigr)^{1+\alpha}
\Bigr) \vert {\mathcal F}_{\sigma}
\Bigr] \geq {\mathbb E}
\biggl[ \int_{\sigma \wedge T \wedge \nu^{\varepsilon}(v_{0})  \wedge \kappa^{\varepsilon}}^{T \wedge \nu^{\varepsilon}(v_{0})  \wedge \kappa^{\varepsilon}}
\vert \nabla V_{s}^{\varepsilon} \vert^2 ds \, \vert \, {\mathcal F}_{\sigma}
\biggr],
\end{equation*}
which completes the proof, since the right-hand side is less than $v_{0} \varepsilon^2$. 
Above, we used a value of $\eta$ such that the process 
$( V_{t \wedge \kappa^{\varepsilon}}^{\varepsilon}
+ \eta ( R_{t \wedge \kappa^{\varepsilon}}^{\varepsilon} )^{1+\alpha}
)_{t \geq 0}$ is non-negative.  
\end{proof}

In a similar manner, we have the following statement:
\begin{lemma}
\label{lem:3:2}
For a given $r_{0}\geq r_{\star}$,
with $r_{\star}$ as in the statement of Lemma 
\ref{lem:supermartingale:2}, assume that, for any $\varepsilon >0$, 
the starting point $X_{0}^{\varepsilon}$ of the diffusion process is located 
at distance $R_{0}^{\varepsilon} = \varepsilon^{2/(1+\alpha)} r_{0}$ from the origin. 

Then, 
with the same notations as in the statement of Lemma 
\ref{lem:supermartingale:2}, for any $\varepsilon \in (0,\varepsilon_\star)$, 
for any $v_{0} >0$ and any $T>0$, 
the martingale $\bigl( \varepsilon^{-1}\int_{0}^{\nu^{\varepsilon,\prime}(v_{0}) \wedge \kappa^{\varepsilon,\prime} \wedge T \wedge t} \nabla V_{s}^{\varepsilon}\cdot \ud B_{s} \bigr)_{t \geq 0}$ has a BMO norm less than 
 {$\sqrt{2v_0}$}.
\end{lemma}

 \subsection{Change of measure}
For a fixed time $T>0$, an intensity $\varepsilon >0$ and a stopping time $\sigma$ with values in 
$[0,T]$, define the probability measure ${\mathbb Q}^{\varepsilon,\sigma}$ by:
\begin{equation*}
\frac{\ud {\mathbb Q}^{\varepsilon,\sigma}}{\ud {\mathbb P}}{\vert}_{{\mathcal F}_{t}}
= \exp \biggl( - \varepsilon^{-1} \int_{0}^{t \wedge \sigma} \nabla V_{s}^{\varepsilon} \cdot \ud B_{s} - \frac12
 \varepsilon^{-2} \int_{0}^{t \wedge \sigma} \vert 
 \nabla V_{s}^{\varepsilon} \vert^2 
 \ud s \biggr), \quad t \geq 0.
\end{equation*} 
By the BMO property proved right above, we obtain the following lemma:

\begin{lemma}
\label{lem:appli:BMO}
Take a collection of stopping times $(\sigma^{\varepsilon})_{\varepsilon >0}$ such the collection of BMO norms of the martingales
$\bigl( \varepsilon^{-1} \int_{0}^{t \wedge \sigma^{\varepsilon}} \nabla V_{s}^{\varepsilon} \cdot \ud B_{s}
\bigr)_{t \geq 0}$ 
is bounded by a constant $K$. Then, there exist two positive constants $\lambda(K)$ and $\Lambda(K)$
and two exponents $p(K) >1$ and $q(K)>0$, {all of them only depending on $K$},  such that:
\begin{equation}
\label{lem:BMO:1}
\forall \varepsilon >0, \quad 
\EE \biggl[ 
\Bigl( \frac{\ud {\mathbb Q}^{\varepsilon,\sigma^{\varepsilon}}}{\ud {\mathbb P}} \Bigr)^{p(K)} \biggr] \leq \lambda(K)
\quad \textrm{\rm and}
\quad 
\EE \biggl[ 
\Bigl( \frac{\ud {\mathbb P}}{\ud {\mathbb Q}^{\varepsilon,\sigma^{\varepsilon}}} \Bigr)^{q(K)} \biggr] \leq \Lambda(K).
\end{equation}

In particular, if, for any $\varepsilon >0$, $\sigma^{\varepsilon}$ is a stopping time with respect to the filtration 
generated by $X^{\varepsilon}$ and, hence, can be put in the form 
$S^{\varepsilon}(X^{\varepsilon})$, then, for any Borel subset $A \subset \cC([0,T];\RR^d)$,
\begin{equation}
\label{lem:BMO:2}
\begin{split}
&\lambda(K)^{-\tfrac1{p(K)-1}}
\PP \Bigl( (X_{0}^{\varepsilon} + \varepsilon B_{t \wedge S^{\varepsilon}(X_{0}^{\varepsilon} +\varepsilon B)})_{0 \leq t \leq T} \in A 
\Bigr)^{\tfrac{p(K)}{p(K)-1}}
\\
&\hspace{15pt} \leq 
\PP \Bigl( (X_{t \wedge \sigma^{\varepsilon}}^{\varepsilon})_{0 \leq t \leq T} \in A 
\Bigr)   \leq \Lambda(K)^{\tfrac1{1+q(K)}}
{\mathbb P} \Bigl( (X_{0}^{\varepsilon} + {\boldsymbol \epsilon} B_{t \wedge S^{\varepsilon}(X_{0}^{\varepsilon} +\varepsilon B)})_{0 \le t \le T} \in A 
\Bigr)^{\tfrac{q(K)}{q(K)+1}}.
\end{split}
\end{equation}
\end{lemma}

\begin{proof}
The first part of the statement, i.e. \eqref{lem:BMO:1}, is a direct consequence of the theory of BMO martingales, see \cite[Chapter 3, Theorems 3.1 and 3.3]{Kazamaki}. 

We turn to the second part of the proof.
By Girsanov's theorem, the law of 
$(X_{t \wedge \sigma^{\varepsilon}}^{\varepsilon})_{0 \leq t \leq T}$ under 
{${\mathbb Q}^{\varepsilon,\sigma^{\varepsilon}}$}
 is the same as the law of 
$(X_{0}^{\varepsilon} + \varepsilon B_{t \wedge S^{\varepsilon}(X_{0}^{\varepsilon} + \varepsilon B)})_{0 \le t \le T}$ under 
{$\PP$}. Hence, for any $A$ as in the statement, 
\begin{equation*}
\begin{split}
&\PP \Bigl( \bigl(X_{0}^{\varepsilon} + \varepsilon B_{t \wedge S^{\varepsilon}(X_{0}^{\varepsilon} + \varepsilon B)}\bigr)_{0 \le t \le T} \in A 
\Bigr) 
\\
&=
{\mathbb Q}^{\varepsilon,\sigma^{\varepsilon}} \Bigl( \bigl(X_{t \wedge \sigma^{\varepsilon}}^{\varepsilon} \bigr)_{0 \le t \leq T} \in A 
\Bigr) 
= \EE \Bigl[ \frac{\ud {\mathbb Q}^{\varepsilon,\sigma^{\varepsilon}}}{\ud \PP}
{\mathbf 1}_{A}
\bigl(
(X_{t \wedge \sigma^{\varepsilon}})_{0 \le t \le T}\bigr) \Bigr]
\leq \lambda(K)^{\tfrac{1}{p(K)}} 
{\mathbb P} \Bigl( \bigl(X_{t \wedge \sigma^{\varepsilon}}^{\varepsilon}\bigr)_{0 \le t \leq T} \in A 
\Bigr)^{\tfrac{p(K)-1}{p(K)}}.
\end{split}
\end{equation*}
Similarly,
\begin{equation*}
\begin{split}
&\PP \Bigl( \bigl(X^{\varepsilon}_{t \wedge \sigma^{\varepsilon}})_{0 \le t \le T} \in A 
\Bigr) 
 = \EE^{{\mathbb Q}^{\varepsilon,\sigma^{\varepsilon}}} \Bigl[ \frac{\ud\PP}{\ud  {\mathbb Q}^{\varepsilon,\sigma^{\varepsilon}}}
{\mathbf 1}_{A}
\Bigl(
\bigl(X_{t \wedge \sigma^{\varepsilon}}^{{\varepsilon}}\bigr)_{0\le t \le T}\Bigr) \Bigr]
\\
&\leq 
\EE^{{\mathbb Q}^{\varepsilon,\sigma^{\varepsilon}}} \Bigl[ \Bigl( \frac{\ud {\mathbb P}}{\ud  {\mathbb Q}^{\varepsilon,\sigma^{\varepsilon}}} \Bigr)^{1+q(K)} \Bigr]^{\tfrac1{1+q(K)}}
{\mathbb P} \Bigl( \bigl(X_{0}^{\varepsilon} + \varepsilon B_{t \wedge S^{\varepsilon}(X_{0}^{\varepsilon} + \varepsilon B)}\bigr)_{0 \leq t \le T} \in A 
\Bigr)^{\tfrac{q(K)}{q(K)+1}}
\\
&\leq \Lambda(K)^{\tfrac1{1+q(K)}}
{\mathbb P} \Bigl( \bigl(X_{0}^{\varepsilon} + \varepsilon B_{t \wedge \sigma^{\varepsilon}(X_{0}^{\varepsilon} + \varepsilon B)}\bigr)_{0 \le t \le T} \in A 
\Bigr)^{\tfrac{q(K)}{q(K)+1}}.
\end{split}
\end{equation*}
\end{proof}

\subsection{Reaching a sufficiently high radius}
As we already explained, our goal in this section is to show that the potential reaches any level of order 
$\varepsilon^2$ in infinitesimal time. To do so, we proceed in two steps: The first one is to prove that the radius reaches any level of order $\varepsilon^{\frac{2}{1+\alpha}}$ and the second one is to prove that, restarting from a radius of order $\varepsilon^{\frac2{1+\alpha}}$, the particle can reach a potential of order $\varepsilon^{2}$. 
%
%

This subsection is dedicated to the first part of the proof. We show that the particle can indeed reach a radius of order $\varepsilon^2$ in infinitesimal time:
\begin{lemma}
\label{lem:hitting:first:barrier}
Assume that, for any $\varepsilon >0$, $X_{0}^{\varepsilon} = 0$.
For any $r_{0}>0$, we have:
\begin{equation*}
\liminf_{A \nearrow \infty}
\liminf_{\varepsilon \searrow 0}
\PP \biggl( \sup_{0\le t \le A \varepsilon^{2(1-\alpha)/(1+\alpha)}} \vert X_{t}^{\varepsilon} \vert \geq r_{0} \varepsilon^{\tfrac{2}{1+\alpha}}  \biggr) = 1.
\end{equation*}
\end{lemma}

\begin{proof}
The proof is based on the perturbed potential introduced in the statement of Lemma 
\ref{lem:supermartingale}. Indeed, we observe that, for $\eta$ as therein,
\begin{equation*}
V_{t}^{\varepsilon} + \eta \bigl( R_{t}^{\varepsilon} \bigr)^{1+\alpha}
\leq \bigl( C_{0} + \eta \bigr) \bigl( R_{t}^{\varepsilon} \bigr)^{1+\alpha}.
\end{equation*}
In particular, if the left-hand side is greater than ${v_{0}} \varepsilon^2$ for some
$v_{0} >0$, then 
\begin{equation*}
 \bigl( R_{t}^{\varepsilon} \bigr)^{1+\alpha} \geq 
 \frac{{v_{0}}}{ C_{0} + \eta} \varepsilon^2. 
\end{equation*}
Moreover, by definition of $\kappa^\varepsilon$ and the bound \eqref{eq:comparison_barriers}, if $\kappa^\varepsilon\leq t_\varepsilon^A:= {A \varepsilon^{\frac{2(1-\alpha)}{1+\alpha}}}
$, then
\begin{equation*}
\sup_{0\leq t\leq t_\varepsilon^A}R^\varepsilon_t \geq  \varepsilon^{\tfrac{2}{1+\beta}+\tfrac{\beta-\alpha}{2}} > r_{0} \varepsilon^{\tfrac{2}{1+\alpha}}
\end{equation*}
for $\varepsilon$ sufficiently small. Hence, in order to prove the statement, it suffices to prove that (with the same notation as in \eqref{eq:3:1}):
\begin{equation*}
\liminf_{A \nearrow \infty} \liminf_{\varepsilon \searrow 0}{\mathbb P}
\bigl( \nu^{\varepsilon}(v_{0}) \leq 
t_{\varepsilon}^A \text{ or } \kappa^\varepsilon\leq t^\varepsilon_A \bigr)  = 1,
\end{equation*}
with
${v_{0}} = \bigl( C_{0}+ \eta \bigr) r_{0}^{1+\alpha}$.  
Obviously, it suffices to prove
\begin{equation*}
\limsup_{A \nearrow \infty} \limsup_{\varepsilon \searrow 0}{\mathbb P}
\Bigl( \nu^{\varepsilon}(v_{0}) > 
t_{\varepsilon}^A, \ \kappa^\varepsilon > t_\varepsilon^A \Bigr)  = 0,
\end{equation*}
{that is, by definition of $\kappa^\varepsilon$,
\begin{equation*}
\limsup_{A \nearrow \infty} \limsup_{\varepsilon \searrow 0}{\mathbb P}
\Bigl( \nu^{\varepsilon}(v_{0}) > 
t_{\varepsilon}^A, \ \sup_{0\leq t\leq t_\varepsilon^A}R^\varepsilon_t < \varepsilon^{\tfrac{2}{1+\beta}+\tfrac{\beta-\alpha}{2}} \Bigr)  = 0.
\end{equation*}
}Therefore,
by Lemmas
\ref{lem:BMO:perturbed:1}
and 
\ref{lem:appli:BMO} with {$\sigma^{\varepsilon} = \nu^{\varepsilon}(v_{0})\wedge \kappa^\varepsilon\wedge T$} and $T=1$, we just need to prove
that
\begin{equation*}
\limsup_{A \nearrow \infty}
\limsup_{\varepsilon \searrow 0}
{\mathbb P} \Bigl(\forall t \leq t_{\varepsilon}^A,
\ V(\varepsilon B_{t}) + \eta \vert \varepsilon B_{t} \vert^{1+\alpha} < {v_{0}} \varepsilon^{2},\ \sup_{0 \leq t \leq t_{\varepsilon}^{A}} \vert \varepsilon B_{t}^{\varepsilon}
\vert < \varepsilon^{\tfrac{2}{1+\beta}+\tfrac{\beta-\alpha}{2}}
 \Bigr) =0.
\end{equation*}
Notice 
 that a crucial fact to pass from the law of $X^{\varepsilon}$ to the law 
of $\varepsilon B$ is the fact that, for 
$t_{\varepsilon}^A \leq 1$ (which is true when $\varepsilon$ is small enough with respect to 
$A$), the BMO norm of {$\varepsilon^{-1} \bigl( \int_{0}^{t \wedge t_{\varepsilon}^A \wedge \nu^{\varepsilon}(v_{0}) \wedge \kappa^\varepsilon} \nabla V_{s}^{\varepsilon}\cdot \ud 
B_{s} \bigr)_{t \geq 0}$} is controlled independently of $A$ and $\varepsilon$.

Since $V(x) + (\eta/2) \vert x \vert^2$ is non-negative on the ball of center $0$ and radius 1, it is sufficient to prove 
that
\begin{equation*}
\limsup_{A \nearrow \infty}
\limsup_{\varepsilon \searrow 0}
{\mathbb P} \Bigl(\forall t \leq t_{\varepsilon}^A, \ \vert \varepsilon B_{t} \vert^{1+\alpha} < 2\eta^{-1} v_{0}\varepsilon^{2} \Bigr) =0,
\end{equation*}
Using the fact that $(B_{t})_{t \geq 0}$
has the same law as $\bigl(A^{\frac12} \varepsilon^{\frac{1-\alpha}{1+\alpha}}
 B_{A^{-1}\varepsilon^{-2(1-\alpha)/(1+\alpha)} t}\bigr)_{t \geq 0}$, we have:
\begin{equation*}
{\mathbb P} \Bigl(\forall t \leq t_{\varepsilon}^A,\  \vert \varepsilon B_{t} \vert^{1+\alpha} \leq
2\eta^{-1} v_{0}\varepsilon^{2} \Bigr)
=
{\mathbb P} \Bigl(\forall t \leq 1,\ \vert  B_{t} \vert^{1+\alpha} \leq 
2\eta^{-1} v_{0}
A^{-\tfrac{1+\alpha}2}
\Bigr).
\end{equation*}
Obviously, the right hand side tends to 1 as $A$ tends to $\infty$. 
\end{proof}

\subsection{Reaching a sufficiently high potential}
Now that the particle is known to reach
$\varepsilon^{2/(1+\alpha)} r_{0}$ in infinitesimal time, we can prove the announced result:

\begin{proposition}
\label{prop:hitting:potential}
Assume that, for any $\varepsilon >0$, $X_{0}^{\varepsilon} = 0$.
Then, 
for any $v_{0}>0$, 
\begin{equation}
\label{eq:hitting:proof:step:1}
\liminf_{A \nearrow \infty}
\liminf_{\varepsilon \searrow 0}
\PP \Bigl( \sup_{0 \leq 
t \leq A  \varepsilon^{2(1-\alpha)/(1+\alpha)
}} V_{t}^{\varepsilon} \geq  v_{0} \varepsilon^2 \Bigr) = 1.
\end{equation}
In particular, 
for any collection $(A_{\varepsilon})_{\varepsilon >0}$ that converges to $\infty$ as $\varepsilon$ tends to $0$, we have:
\begin{equation}
\label{eq:hitting:proof:step:2}
\liminf_{\varepsilon \searrow 0}
\PP \Bigl( \sup_{0 \leq 
t \leq A_{\varepsilon}  \varepsilon^{2(1-\alpha)/(1+\alpha)
}} V_{t}^{\varepsilon} \geq v_{0} \varepsilon^2 \Bigr) = 1.
\end{equation}
\end{proposition}

The proof is split in several steps. Most of the difficulty is to show the following lemma:

\begin{lemma}
\label{lem:restarting}
For any $v_{0} >0$, there exists a collection $(\lambda_{S},A_{S})_{S >0}$ such that 
$\lim_{S \nearrow \infty} \lambda_{S} = \lim_{S \nearrow \infty} A_{S} = \infty$
and 
\begin{equation}
\label{eq:restarting:proof:step:1}
\liminf_{S \nearrow \infty}
\liminf_{\varepsilon \searrow 0}
\PP \Bigl( \sup_{0 \leq 
t \leq A_{S}   \varepsilon^{2(1-\alpha)/(1+\alpha)
}} V_{t}^{\varepsilon} \geq v_{0} \varepsilon^2
\, \big\vert \, 
R_{0}^{\varepsilon} = \varepsilon^{\tfrac{2}{1+\alpha}} \lambda_{S} r_{\star}
 \Bigr) = 1,
\end{equation}
where we used the notation $\PP ( \cdot \, \vert \, 
R_{0}^{\varepsilon} = \varepsilon^{\frac{2}{1+\alpha}} \lambda_{S} r_{\star})$ to indicate the fact that the initial condition is forced to satisfy
$R_{0}^{\varepsilon} = \varepsilon^{\frac{2}{1+\alpha}} \lambda_{S} r_{\star}$, 
{with $r_{\star}$ as in the statements of Lemmas \ref{lem:supermartingale:2}
and
\ref{lem:3:2}.}
\end{lemma}

Taking Lemma 
\ref{lem:restarting} for granted for a while, we prove:

\begin{proof}[Proof of Proposition \ref{prop:hitting:potential}]
\textit{First step.}
We first check that 
\eqref{eq:hitting:proof:step:1}
implies 
\eqref{eq:hitting:proof:step:2}. 
Indeed, if \eqref{eq:hitting:proof:step:1} is true, then, for any {$\pi>0$}, we have, for some $A(\pi) >0$, 
\begin{equation*}
\liminf_{\varepsilon \searrow 0}
\PP \Bigl( \sup_{0 \leq 
t \leq A(\pi)  \varepsilon^{2(1-\alpha)/(1+\alpha)
}} V_{t}^{\varepsilon} \geq v_{0} \varepsilon^2 \Bigr) \geq 1 - \pi. 
\end{equation*}
Hence, we can find $\varepsilon(\pi) >0$ such that, for $\varepsilon \in (0,\varepsilon(\pi))$, 
\begin{equation*}
\PP \Bigl( \sup_{0 \leq 
t \leq A(\pi)  \varepsilon^{2(1-\alpha)/(1+\alpha)
}} V_{t}^{\varepsilon} \geq v_{0} \varepsilon^2 \Bigr) \geq 1 - 2\pi.
\end{equation*}
Modifying $\varepsilon(\pi)$ if necessary, we can assume that $A_{\varepsilon} > A(\pi)$ 
for $\varepsilon \in (0,\varepsilon(\pi))$ (with $A_\varepsilon$ as in \eqref{eq:hitting:proof:step:2}). This proves that, for $\varepsilon \in (0,\varepsilon(\pi))$, 
\begin{equation*}
\PP \Bigl( \sup_{0 \leq 
t \leq A_{\varepsilon}  \varepsilon^{2(1-\alpha)/(1+\alpha)
}} V_{t}^{\varepsilon} \geq v_{0} \varepsilon^2 \Bigr) \geq 1 - 2\pi,
\end{equation*}
which is 
\eqref{eq:hitting:proof:step:2}. 

\vskip 2pt

\textit{Second step.}
We now prove 
\eqref{eq:hitting:proof:step:1}. 
We make use of 
\eqref{eq:restarting:proof:step:1}. For a given $\pi >0$, we can find $A(\pi) >0$ and 
$\lambda(\pi) >0$ such that
\begin{equation}
\label{lem:markov:1}
\liminf_{\varepsilon \searrow 0}
\PP\Bigl( \sup_{0 \leq 
t \leq A(\pi)   \varepsilon^{2(1-\alpha)/(1+\alpha)
}} V_{t}^{\varepsilon} \geq v_{0} \varepsilon^2 \, \vert \, 
R_{0}^{\varepsilon} = \lambda(\pi) r_{\star}
\varepsilon^{\tfrac2{1+\alpha}} 
\Bigr) \geq 1 - \pi.
\end{equation}

We now invoke Lemma 
\ref{lem:hitting:first:barrier} with $r_{0} = \lambda(\pi) r_{\star}$ therein. 
It says that there exists $A'(\pi) >0$ such that, whenever $X_{0}^{\varepsilon}=0$ for all $\varepsilon >0$,
\begin{equation}
\label{lem:markov:2}
\liminf_{\varepsilon \searrow 0}
\PP \biggl( \sup_{0\le t \le A'(\pi) \varepsilon^{2(1-\alpha)/(1+\alpha)}} \vert X_{t}^{\varepsilon} \vert \geq \lambda(\pi) r_\star \varepsilon^{\tfrac{2}{1+\alpha}}  \biggr) \geq 1- \pi.
\end{equation}

Combining 
\eqref{lem:markov:1}
and 
\eqref{lem:markov:2} together with Markov property, we deduce that, under the prescription that
$X_{0}^{\varepsilon}=0$ for all $\varepsilon >0$, 
\begin{equation*}
\liminf_{\varepsilon \searrow 0}
\PP\Bigl( \sup_{0 \leq 
t \leq (A(\pi) + A'(\pi))  \varepsilon^{2(1-\alpha)/(1+\alpha)
}} V_{t}^{\varepsilon} \geq v_{0} \varepsilon^2 \Bigr) \geq \bigl( 1 -  \pi \bigr)^2. 
\end{equation*} 
Therefore, for any $A > A(\pi)+A'(\pi)$, 
\begin{equation*}
\liminf_{\varepsilon \searrow 0}
\PP\Bigl( \sup_{0 \leq 
t \leq A  \varepsilon^{2(1-\alpha)/(1+\alpha)
}} V_{t}^{\varepsilon} \geq v_{0} \varepsilon^2 \Bigr) \geq \bigl( 1 -  \pi \bigr)^2. 
\end{equation*} 
In particular, for any $\pi  >0$, 
\begin{equation*}
\liminf_{A \nearrow \infty} \liminf_{\varepsilon \searrow 0}
\PP\Bigl( \sup_{0 \leq 
t \leq A  \varepsilon^{2(1-\alpha)/(1+\alpha)
}} V_{t}^{\varepsilon} \geq v_{0} \varepsilon^2 \Bigr) \geq \bigl( 1 - \pi \bigr)^2,
\end{equation*} 
under the initial condition $X_{0}^{\varepsilon} = 0$, for all $\varepsilon >0$. 
The conclusion easily follows by sending $\pi$ to $0$. 
\end{proof}

\begin{proof}[Proof of Lemma \ref{lem:restarting}]\vskip 4pt

\textit{First step.}
We consider the event 
$\bigl\{ \sup_{0 \leq 
t \leq A  \varepsilon^{2(1-\alpha)/(1+\alpha)
}} V_{t}^{\varepsilon} \geq  v_{0} \varepsilon^2 \bigr\}$,
with the constraint that $R_{0}^{\varepsilon} = \lambda r_{\star} \varepsilon^{\frac2{1+\alpha}}$, {for some $\lambda \geq 2$}.
We recall that there exists a cone ${\mathcal C}$ such that $g$ is above $c_{0}$ on
${\mathcal C}$. In particular, if $x \in {\mathcal C}$ and $V(x) + \eta_{\varepsilon} \vert x \vert^{1+\alpha} \geq 2 v_{0} \varepsilon^2$, for {$\varepsilon$} small enough such that $\eta_{\varepsilon} < c_{0}$, with $\eta_{\varepsilon}$
as in the statement of 
Lemma \ref{lem:supermartingale:2}, then $V(x) \geq v_{0} \varepsilon^2$. 
As a consequence, the above event 
contains the following one:
\begin{equation*}
\Bigl\{ \exists t \leq t^A_{\varepsilon} : V_{t}^{\varepsilon} + \eta_{\varepsilon} \bigl( R_{t}^{\varepsilon} \bigr)^{1+\alpha} \geq 2  v_{0} \varepsilon^2, \  X_{t}^{\varepsilon} \in {\mathcal C} \Bigr\} \cap \Bigl\{ \inf_{s \leq t^A_{\varepsilon}} R_{s}^{\varepsilon}
\geq r_{\star} \varepsilon^{\tfrac{2}{1+\alpha}}
\Bigr\} \cap 
\Bigl\{ \sup_{s \leq t_{\varepsilon}^A} R_{s}^{\varepsilon} 
\leq \varepsilon^{\tfrac{2}{1+\beta}+\tfrac{\beta-\alpha}{2}}
 \Bigr\},
\end{equation*}
where we recall the notation
$t_{\varepsilon}^A = A \varepsilon^{\frac{2(1-\alpha)}{(1+\alpha)}}$. 
Therefore Lemmas
\ref{lem:3:2}
and
\ref{lem:appli:BMO} with {$\sigma^{\varepsilon} = \nu^{\varepsilon,\prime}(2v_{0}) \wedge \kappa^{\varepsilon,\prime}$} and $T=1$
say that, in order to prove
\begin{equation*}
\limsup_{(A,\lambda) \nearrow (\infty,\infty)}
\liminf_{\varepsilon \searrow 0}
\PP\Bigl( \sup_{0 \leq 
t \leq t_{\varepsilon}^A} V_{t}^{\varepsilon} \geq  v_{0} \varepsilon^2 
\Bigr)=1,
\end{equation*}
it is sufficient to prove that 
 $\limsup_{(A,\lambda) \nearrow (\infty,\infty)} \liminf_{\varepsilon \searrow 0} p^{\varepsilon,A,\lambda} =1$, 
with
\begin{equation*}
\begin{split}
p^{\varepsilon,A,\lambda} :={\mathbb P} \biggl( &\Bigl\{ \exists  t \leq t_{\varepsilon}^A : 
V\bigl(X_{0}^{\varepsilon} + \varepsilon B_{t}\bigr) + \eta_{\varepsilon}
\bigl\vert X_{0}^{\varepsilon} + \varepsilon B_{t}
 \bigr\vert^{1+\alpha} \geq 2v_{0}\varepsilon^2, \ X_{0}^{\varepsilon} + \varepsilon B_{t} \in {\mathcal C} 
\Bigr\}
\\
&\hspace{10pt}
 \cap \Bigl\{ \inf_{s \leq t^A_{\varepsilon}} \bigl\vert X_{0}^{\varepsilon} + \varepsilon B_{s}
 \bigr\vert
\geq r_{\star} \varepsilon^{\tfrac{2}{1+\alpha}}
\Bigr\}
\cap 
\Bigl\{ \sup_{s \leq t_{\varepsilon}^A} \bigl\vert X_{0}^{\varepsilon} + \varepsilon B_{s}
 \bigr\vert
\leq \varepsilon^{\tfrac{2}{1+\beta}+\tfrac{\beta-\alpha}{2}}
 \Bigr\}
\biggr),
\end{split}
\end{equation*}
where $R_{0}^{\varepsilon} = \lambda r_{\star} \varepsilon^{\frac2{1+\alpha}}${ (and $\lambda\ge 2$)}.
Notice that, as in the proof of Lemma 
\ref{lem:hitting:first:barrier}, a crucial fact to pass from the law of $X^{\varepsilon}$ to the law 
of $X_{0}^{\varepsilon} + \varepsilon B$ is the fact that, for 
$t_{\varepsilon}^A \leq 1$ (which is true when $\varepsilon$ is small enough with respect to 
$A$), the BMO norm of $\varepsilon^{-1} \bigl( \int_{0}^{t
{\wedge \kappa^{\varepsilon,\prime}}  \wedge t_{\varepsilon}^A \wedge \nu^{\varepsilon,\prime}(2v_{0})} \nabla V_{s}^{\varepsilon} \cdot \ud 
B_{s} \bigr)_{t \geq 0}$ is controlled independently of $A$, $\varepsilon$ and $\lambda$.
\vskip 4pt

\textit{Second step.}
We first aim at lower bounding the probability
$p^{\varepsilon,A,\lambda}$. To do so, we use again the fact that $g$ is greater than $c_{0}$ in the cone
${\mathcal C}$. Hence, whenever $X_{0}^{\varepsilon}
+ \varepsilon B_{t}
 \in {\mathcal C}$, 
 $V\bigl( X_{0}^{\varepsilon} + \varepsilon B_{t} \bigr) \geq c_{0} \bigl\vert X_{0}^{\varepsilon}
+ \varepsilon B_{t}
 \bigr\vert^{1+\alpha}$. 
In particular, it suffices to lower bound the probability
\begin{equation*}
\PP \Bigl( \exists t \leq t^A_{\varepsilon} : X_{0}^{\varepsilon}
+ \varepsilon B_{t}
 \in {\mathcal C}, \quad \inf_{s \leq t^A_{\varepsilon}} 
 \bigl\vert 
 X_{0}^{\varepsilon}
+ \varepsilon B_{s}
\bigr\vert
\geq \lambda_{0} r_{\star} \varepsilon^{\tfrac{2}{1+\alpha}}, \quad \sup_{s \leq t^A_{\varepsilon}} 
\bigl\vert 
 X_{0}^{\varepsilon}
+ \varepsilon B_{s}
\bigr\vert
\leq \varepsilon^{\tfrac{2}{1+\beta}+\tfrac{\beta-\alpha}{2}}
 \Bigr),
\end{equation*}
with 
$\lambda_{0} = \max \Bigl( 1, \frac1{r_{\star}} \bigl( \frac{2v_{0}}{c_{0}} \bigr)^{\frac1{1+\alpha}} \Bigr)$.
Hence,
\begin{equation}
\label{eq:pepsilonA}
\begin{split}
p^{\varepsilon,A,\lambda} &\geq \PP \Bigl(\exists 
t \in [0,t^A_{\varepsilon}] : X_{0}^{\varepsilon} + \varepsilon B_{t} \in {\mathcal C} \Bigr)
- \PP \Bigl(\inf_{s \leq t_{\varepsilon}^A}
\vert X_{0}^{\varepsilon} + \varepsilon B_{s} \vert < \lambda_{0} r_{\star} \varepsilon^{\tfrac{2}{1+\alpha}}
\Bigr)
\\
&\hspace{15pt}
 - \PP 
\Bigl( \sup_{s \leq t_{\varepsilon}^A} 
\vert X_{0}^{\varepsilon} + \varepsilon B_{s} \vert
> \varepsilon^{\tfrac{2}{1+\beta}+\tfrac{\beta-\alpha}{2}}
 \Bigr).
\end{split}
\end{equation}
\vskip 4pt

\textit{Third step.} Now, using the fact that 
$X_{0}^{\varepsilon}$ is independent of
$(B_{t})_{t \geq 0}$
and that 
$(B_{t})_{t \geq 0}$
has the same law as $((t_{\varepsilon}^A)^{1/2}
B_{(t_{\varepsilon}^A)^{-1} t})_{t \geq 0}$, we have
\begin{equation*}
\begin{split}
\PP \Bigl(\exists 
t \in [0,t_{\varepsilon}^A] : X_{0}^{\varepsilon} + \varepsilon B_{t} \in {\mathcal C} \Bigr)
&= {\mathbb P}
\Bigl(\exists 
t \in [0,1] : X_{0}^{\varepsilon} + A^{1/2}
\varepsilon^{\tfrac{2}{1+\alpha}} B_{t} \in {\mathcal C}
\Bigr)
\\
&=
{\mathbb P}
\Bigl(\exists 
t \in [0,1] :  \lambda A^{-1/2}  \tilde X_{0}^{\varepsilon} + 
 B_{t} \in {\mathcal C}
\Bigr),
\end{split}
\end{equation*}
where we have let
$\tilde X_{0}^{\varepsilon} := \lambda^{-1} \varepsilon^{-\frac{2}{1+\alpha}} X_{0}^{\varepsilon}$ 
{(recalling that $\vert X_{0}^{\varepsilon} \vert = \lambda r_{\star} \varepsilon^{\frac2{1+\alpha}}$)}.

Using the invariance in law of $B$ by rotation, we get
\begin{equation*}
\begin{split}
\PP \Bigl(\exists 
t \in [0,t_{\varepsilon}^A] : X_{0}^{\varepsilon} + \varepsilon B_{t} \in {\mathcal C} \Bigr)
&\geq \inf_{\vert u \vert = r_{\star}}
{\mathbb P}
\Bigl(\exists 
t \in [0,1] : \lambda A^{-1/2} u + 
 B_{t} \in {\mathcal C}
\Bigr)
\\
&= \inf_{\vert u \vert = r_{\star}}
{\mathbb P}
\Bigl(\exists 
t \in [0,1] : \lambda A^{-1/2} e + 
 B_{t} \in {\mathcal C}^u
\Bigr),
\end{split}
\end{equation*}
where ${\mathcal C}^u$ is a new cone obtained by rotating ${\mathcal C}$ by a rotation that permits to pass from $u$ to $e$ where $e$ is a fixed vector such that $\vert e \vert =r_{\star}$ (the rotation matrix can be constructed in a canonical way by a Gram Schmidt procedure). Since all the $({\mathcal C}^u)_{\vert u \vert = r_{\star}}$ are isometric, we can find a finite covering of $\RR^d$ by closed cones $({\mathcal C}^{0,i})_{i=1,\cdots,N}$ with $0$ as common vertex such that 
each ${\mathcal C}^u$ contains at least one ${\mathcal C}^{0,i}$. Hence,
\begin{equation*}
\begin{split}
\PP \Bigl(\exists 
t \in [0,t^A_{\varepsilon}] : X_{0}^{\varepsilon} + \varepsilon B_{t} \in {\mathcal C} \Bigr)
&\geq \inf_{i=1,\cdots,N}
{\mathbb P}
\Bigl(\exists 
t \in [0,1] : \lambda  A^{-1/2} e + 
 B_{t} \in {\mathcal C}^{0,i}
\Bigr)
\\
&\geq \inf_{i=1,\cdots,N}
{\mathbb P}
\Bigl(\exists 
t \in [0,1] : \lambda A^{-1/2} e + 
 B_{t} \in \overset{\circ}{{\mathcal C}^{0,i}}
\Bigr),
\end{split}
\end{equation*}
where $\overset{\circ}{{\mathcal C}^{0,i}}$ is the interior of 
${{\mathcal C}^{0,i}}$.

For a given $i \in \{1,\cdots,N\}$, 
the set of continuous functions $f : [0,1] \rightarrow {\mathbb R}^d$ such that there exists $t \in [0,1]$ for which $f_{t} \in \overset{\circ}{{\mathcal C}^{0,i}}$ is open for the uniform topology on $[0,1]$. 
Hence, by the Portemanteau theorem,
\begin{equation*}
\begin{split}
\liminf_{\lambda A^{-1/2} \searrow 0}
{\mathbb P}
\Bigl(\exists 
t \in [0,1] : \lambda A^{-1/2} e + 
 B_{t} \in \overset{\circ}{{\mathcal C}^{0,i}}
\Bigr)
\geq
{\mathbb P}
\Bigl(\exists 
t \in [0,1] :  B_{t} \in \overset{\circ}{{\mathcal C}^{0,i}}
\Bigr)
= 1.
\end{split}
\end{equation*}
{To derive the second equality, we used the fact that, for any $a>0$, 
$${\mathbb P}
\bigl(\exists 
t \in [0,1] :  B_{t} \in \overset{\circ}{{\mathcal C}^{0,i}}
\bigr)
={\mathbb P}
\bigl(\exists 
t \in [0,a] :  B_{t} \in \overset{\circ}{{\mathcal C}^{0,i}}
\bigr) 
={\mathbb P}
\Bigl( \bigcap_{a'>0} \{ \exists 
t \in [0,a'] :  B_{t} \in \overset{\circ}{{\mathcal C}^{0,i}}
\}
\Bigr).$$ By Blumenthal's zero-one law, the last term (hence the first term as well) is equal to $0$ or $1$. Since it is obviously non-zero, we deduce that it is equal to $1$.} We deduce that
\begin{equation*}
\liminf_{\lambda A^{-1/2} \searrow 0}
\liminf_{\varepsilon \searrow 0}
\PP \Bigl(\exists 
t \in [0,t^A_{\varepsilon}] : X_{0}^{\varepsilon} + \varepsilon B_{t} \in {\mathcal C} \Bigr)
=1.
\end{equation*}
\vskip 4pt

\textit{Fourth step.} 
We now have a look at the second term in the right-hand side of 
\eqref{eq:pepsilonA}. By a new scaling argument, we  {get}:
\begin{equation*}
\begin{split}
\PP \Bigl(\inf_{0 \leq t \leq t^A_{\varepsilon}}
\vert X_{0}^{\varepsilon} + \varepsilon B_{t} \vert < \lambda_{0}  r_{\star}
\varepsilon^{\tfrac{2}{1+\alpha}}
\Bigr)
&= \PP \Bigl(\inf_{0 \leq t \leq \lambda^{-2} A}
\vert 
\tilde X_{0}^{\varepsilon}
+
 B_{t} \vert  <
\lambda_{0}  
\lambda^{-1}
 r_{\star}
\Bigr).
\end{split}
\end{equation*}
By a rotation argument (using the fact that $\tilde X_{0}^{\varepsilon}$
and $(B_{t})_{t \geq 0}$ are independent {together with the identity $\vert \tilde X_{0}^{\varepsilon} \vert = r_{\star}$}),
\begin{equation*}
\begin{split}
\PP \Bigl(\inf_{0 \leq t \leq t^A_{\varepsilon}}
\vert X_{0}^{\varepsilon} + \varepsilon B_{t} \vert  < \lambda_{0}
r_{\star}
\varepsilon^{\tfrac{2}{1+\alpha}}
\Bigr)
&\leq \PP \Bigl(\inf_{0 \leq t \leq \lambda^{-2} A }
\vert 
r_{\star} e
+
 B_{t} \vert < \lambda_{0} \lambda^{-1}r_{\star}
\Bigr)
\end{split}
\end{equation*}
where $e$ is an arbitrary unit vector. 
Observe that, for a fixed $S >0$, 
\begin{equation*}
\lim_{\lambda \nearrow \infty}
\PP \Bigl(\inf_{0 \leq t \leq S}
\vert 
r_{\star} e
+
 B_{t} \vert^{1+\alpha} < \lambda_{0} \lambda^{-1}r_{\star}
\Bigr) = 0.
\end{equation*}
Hence, we can find a collection $(\lambda_{S})_{S>0}$, with 
$\lim_{S \nearrow \infty} \lambda_{S} = \infty$ such that 
\begin{equation*}
\lim_{S \nearrow \infty}
\PP \Bigl(\inf_{0 \leq t \leq S}
\vert 
r_{\star} e
+
 B_{t} \vert^{1+\alpha} < \lambda_{0} \lambda_{S}^{-1}r_{\star}
\Bigr) = 0.
\end{equation*}
So, letting $A_{S} = S \lambda_{S}^2$, we get 
\begin{equation*}
\begin{split}
\lim_{S \nearrow \infty} \PP \Bigl(\inf_{0 \leq t \leq t^{A_{S}}_{\varepsilon}}
\vert X_{0}^{\varepsilon} + \varepsilon B_{t} \vert < \lambda_{0}  r_{\star}
\varepsilon^{\tfrac{2}{1+\alpha}}
\,
\Big\vert \, R_{0}^{\varepsilon} = \varepsilon^{\tfrac{2}{1+\alpha}} \lambda_{S}
r_{\star}
\Bigr)
&= 0.
\end{split}
\end{equation*}
%
%
%
%

\vskip 4pt

\textit{Fifth step.}
We now have a look at the last term in the right-hand side of 
\eqref{eq:pepsilonA}.
As in the previous step, we have:
\begin{equation*}
{\mathbb P} \Bigl( 
 \sup_{0 \le t \leq t_{\varepsilon}^A} \bigl\vert 
 X_{0}^{\varepsilon} + \varepsilon B_{t} \bigr\vert 
> \varepsilon^{\frac{2}{1+\beta}+\frac{\beta-\alpha}{2}}
\Bigr) 
=
{\mathbb P} \Bigl( 
 \sup_{0 \le t\leq \lambda^{-2} A} \bigl\vert 
r_{\star} e+ B_{t} \bigr\vert 
>  \lambda^{-1 }\varepsilon^{\tfrac{2}{1+\beta}+\tfrac{\beta-\alpha}{2}}
\varepsilon^{-\tfrac{2}{1+\alpha}}
\Bigr).
\end{equation*}
{By formula \eqref{eq:comparison_barriers}, $\varepsilon^{\frac{2}{1+\beta}+\frac{\beta-\alpha}{2}} \varepsilon^{-\frac{2}{1+\alpha}}$ tends to $\infty$ as $\varepsilon$ tends to $0$, hence the above probability tends to $0$}. 
This shows that 
\begin{equation*}
\liminf_{\lambda A^{-1/2} \searrow 0}
\liminf_{\varepsilon \searrow 0}
{\mathbb P} \Bigl( 
 \sup_{0 \le t \leq t_{\varepsilon}^A} \bigl\vert 
 X_{0}^{\varepsilon} + \varepsilon B_{t} \bigr\vert 
> \varepsilon^{\frac{2}{1+\beta}+\frac{\beta-\alpha}{2}}
\Bigr) 
=0.
\end{equation*}
\vskip 4pt

\textit{Conclusion.}
Collecting all the five steps and choosing the pair $(A,\lambda)$ in the third and fifth steps as
$(A,\lambda)=(S \lambda_{S}^2,\lambda_{S})$, for $S$ large, we complete the proof.
\end{proof}
 
\section{Hitting points of the level sets}
\label{se:4}

The conclusion of the previous section is that the particle hits in small time the level sets of the potential of order $\varepsilon^2$. We now prove that
the law of the hitting point of the level set $\{ V \geq \varepsilon^2 v_{0}\}$ is asymptotically supported by the region where
$g$ is positive, which is in fact well-understood: Given a value of the potential, the larger $g$, the closer
the point to the origin.

Here is the precise statement:

\begin{proposition}
\label{prop:hitting:g>0}
Assume that, for any $\varepsilon >0$, $X_{0}^{\varepsilon} = 0$.
Then, for a given $v_{0} >0$ and with the same notations as in 
\eqref{eq:3:1},
\begin{equation*}
\liminf_{\mu \searrow 0}
\liminf_{\varepsilon \searrow 0} \PP \bigl( g( X^{\varepsilon}_{\tau^{\varepsilon}(v_{0})}) > \mu
\bigr) = 1.
\end{equation*}
\end{proposition}
\vskip 4pt

\begin{proof}
\textit{First step.} We start with the following remark. If we assume
\begin{equation}
\label{eq:lower:bound:g}
g \bigl( X_{\tau^{\varepsilon}(v_{0})} \bigr) \leq \mu,
\end{equation}
then,
\begin{equation}
\varepsilon^{2} v_{0} \leq \mu \bigl( R^{\varepsilon}_{\tau^{\varepsilon}(v_{0})} \bigr)^{1+\alpha}, 
\quad \textrm{\rm and thus} \quad
R^{\varepsilon}_{\tau^{\varepsilon}(v_{0})} \geq r_{0} \varepsilon^{\tfrac{2}{1+\alpha}},
\label{eq:lower:bound:g:2}
\end{equation}
with 
\begin{equation}
\label{eq:additional:label:r0}
r_{0} = \bigl( \frac{v_{0}}{\mu} \bigr)^{\tfrac{1}{1+\alpha}}.$$
\end{equation}
We also have, for any $\eta >0$ and any $t \in [0,\tau^{\varepsilon}(v_{0})]$,
\begin{equation*}
V_{t}^{\varepsilon} + 
\eta \bigl( R_{t}^{\varepsilon}
\bigr)^{1+\alpha} \leq \varepsilon^2 v_{0} + 
\eta \bigl( R_{t}^{\varepsilon}
\bigr)^{1+\alpha}.
\end{equation*}
In particular, if the left-hand side is greater than $2 \varepsilon^2 v_{0}$, then 
 $\bigl( R_{t}^{\varepsilon}
\bigr)^{1+\alpha} \geq \varepsilon^2 v_{0}/{\eta}$. 

With the same notations as in 
\eqref{eq:3:1}, choose now {$\eta=\eta_\varepsilon$ from Lemma \ref{lem:supermartingale:2}, with $\varepsilon$ small enough to have $\eta_{\varepsilon} < \mu$}. Then, with the same notation for $r_{0}$ as above, we 
also have
\begin{equation}
\label{eq:lower:bound:g:3}
R^{\varepsilon}_{t} \geq r_{0} \varepsilon^{\tfrac{2}{1+\alpha}},
\end{equation}
We deduce from 
\eqref{eq:lower:bound:g:2}
and
\eqref{eq:lower:bound:g:3}
that, under 
\eqref{eq:lower:bound:g}, 
at time $\tau^{\varepsilon}(v_{0}) \wedge \nu^{\varepsilon,\prime}(2v_{0})$,
\begin{equation*}
R^{\varepsilon}_{\tau^{\varepsilon}(v_{0}) \wedge \nu^{\varepsilon,\prime}(2v_{0})} \geq r_{0} \varepsilon^{\tfrac{2}{1+\alpha}}.
\end{equation*}
Here, the key fact is that {we can choose} $\mu(v_{0}) >0$ such that, for $\mu \leq \mu(v_{0})$ {and $\varepsilon$ such that $\eta_\varepsilon<\mu$}, $r_{0} \geq 2 r_{\star}$, 
with $r_{\star}$ as in the statement of Lemma 
\ref{lem:supermartingale:2}.
\vskip 4pt

\textit{Second step.}
Another point is that, {by Proposition \ref{prop:hitting:potential}}, for a given ${\pi} >0$, we can choose $A>0$ such that 
\begin{equation}
\label{eq:g:not:small:1}
\liminf_{\varepsilon \searrow 0}
{\mathbb P}
\bigl( 
\tau^{\varepsilon}(v_{0}) \leq t_{\varepsilon}^A
\bigr) \geq 1 - \pi, 
\end{equation}
{where we recall that $t_{\varepsilon}^A=A \varepsilon^{\frac{2(1-\alpha)}{1+\alpha}}$}.
Combined 
with the conclusion of the first step, this leads to
\begin{equation*}
\begin{split}
{\mathbb P}
\Bigl( g \bigl( X_{\tau^{\varepsilon}(v_{0})} \bigr) \leq \mu
\Bigr)
&\leq 
{\mathbb P}
\Bigl( g \bigl( X_{\tau^{\varepsilon}(v_{0})} \bigr) \leq \mu , 
\tau^{\varepsilon}(v_{0}) \wedge \nu^{\varepsilon,\prime}(2v_{0}) \leq t_{\varepsilon}^A
\Bigr) + {\mathbb P}
\Bigl( 
\tau^{\varepsilon}(v_{0}) \wedge \nu^{\varepsilon,\prime}(2v_{0}) > t_{\varepsilon}^A
\Bigr)
\\
&\leq {\mathbb P}
\Bigl( 
R^{\varepsilon}_{\tau^{\varepsilon}(v_{0}) \wedge \nu^{\varepsilon,\prime}(2v_{0})} \geq r_{0} \varepsilon^{\tfrac{2}{1+\alpha}}, 
\tau^{\varepsilon}(v_{0}) \leq t_{\varepsilon}^A
\Bigr) + {\mathbb P}
\Bigl( 
\tau^{\varepsilon}(v_{0}) > t_{\varepsilon}^A
\Bigr),
\end{split}
\end{equation*}
where $r_{0} \geq 2 r_{\star}$, for $\mu \leq \mu(v_{0})$ {and $\eta_\varepsilon<\mu$}.
Therefore,  
\begin{equation}
\label{eq:conclusion:second:step:proposition:4:1}
\begin{split}
\limsup_{\varepsilon \searrow 0}{\mathbb P}
\Bigl( g \bigl( X_{\tau^{\varepsilon}(v_{0})} \bigr) \leq \mu
\Bigr)
&\leq \limsup_{\varepsilon \searrow 0} 
{\mathbb P}
\Bigl( 
R^{\varepsilon}_{\tau^{\varepsilon}(v_{0}) \wedge \nu^{\varepsilon,\prime}(2v_{0})} \geq r_{0} \varepsilon^{\tfrac{2}{1+\alpha}}, 
\tau^{\varepsilon}(v_{0}) \wedge \nu^{\varepsilon,\prime}(2v_{0}) \leq t_{\varepsilon}^A
\Bigr)
 + \pi.
\end{split}
\end{equation}
\vskip 2pt

\textit{Third step.} For $r_{0}$ as in the previous step, we aim at upper bounding
\begin{equation*}
\limsup_{\varepsilon \searrow 0} {\mathbb P}
\Bigl( 
R^{\varepsilon}_{\tau^{\varepsilon}(v_{0})
\wedge 
 \nu^{\varepsilon,\prime}(2v_{0})} \geq r_{0} \varepsilon^{\tfrac{2}{1+\alpha}}, 
\tau^{\varepsilon}(v_{0}) \wedge \nu^{\varepsilon,\prime}(2v_{0}) \leq t_{\varepsilon}^A
\Bigr).
\end{equation*}
Here $X_{0}^{\varepsilon}=0$. 
Since $r_{0}> 2 r_{\star}$, we deduce from Markov property that 
the above term is less than
\begin{equation*}
\limsup_{\varepsilon \searrow 0} {\mathbb P}
\Bigl( 
R^{\varepsilon}_{\tau^{\varepsilon}(v_{0}) \wedge \nu^{\varepsilon,\prime}(2v_{0})} \geq r_{0} \varepsilon^{\tfrac{2}{1+\alpha}}, 
\tau^{\varepsilon}(v_{0}) \wedge \nu^{\varepsilon,\prime}(2v_{0}) \leq t_{\varepsilon}^A
\, \big\vert \, 
R_{0}^{\varepsilon} = \tfrac{1}{2} (r_{0}+ r_{\star}) \varepsilon^{\tfrac2{1+\alpha}}
\Bigr).
\end{equation*}
Notice indeed that, on the event $\{
R^{\varepsilon}_{\tau^{\varepsilon}(v_{0}) \wedge \nu^{\varepsilon,\prime}(2v_{0})} \geq r_{0} \varepsilon^{{2}/(1+\alpha)}, 
\tau^{\varepsilon}(v_{0}) \leq t_{\varepsilon}^A
\}$, the first time when $R^{\varepsilon}$ hits 
$\tfrac{1}{2} (r_{0}+ r_{\star}) \varepsilon^{2/(1+\alpha)}$
is strictly less than $\tau^{\varepsilon}(v_{0}) \wedge \nu^{\varepsilon,\prime}(2v_0)$.

Hence, 
with the prescription 
that
$R_{0}^{\varepsilon} = \tfrac{1}{2} (r_{0}+ r_{\star}) \varepsilon^{2/(1+\alpha)}$
for all $\varepsilon >0$, we use the same strategy as in the proof of Lemma \ref{lem:restarting}. 
\begin{align}
&\limsup_{\varepsilon \searrow 0} {\mathbb P}
\Bigl( 
R^{\varepsilon}_{\tau^{\varepsilon}(v_{0}) \wedge \nu^{\varepsilon,\prime}(2v_{0})} \geq r_{0} \varepsilon^{\tfrac{2}{1+\alpha}}, 
\tau^{\varepsilon}(v_{0}) 
\wedge \nu^{\varepsilon,\prime}(2v_{0})
\leq t_{\varepsilon}^A
\Bigr) 
\nonumber 
\\
&\leq \limsup_{\varepsilon \searrow 0} {\mathbb P}
\Bigl( 
R^{\varepsilon}_{\tau^{\varepsilon}(v_{0}) \wedge  \nu^{\varepsilon,\prime}(2v_{0})} \geq r_{0} \varepsilon^{\tfrac{2}{1+\alpha}}, 
\tau^{\varepsilon}(v_{0}) \wedge \nu^{\varepsilon,\prime}(2v_{0}) \leq t_{\varepsilon}^A, \inf_{0 \leq t \leq  \tau^{\varepsilon}(v_{0}) \wedge \nu^{\varepsilon,\prime}(2v_{0})}
R_{t}^{\varepsilon} \geq r_{\star} \varepsilon^{\tfrac{2}{1+\alpha}},\nonumber
\\
&\hspace{300pt} \sup_{0 \leq t \leq t_{\varepsilon}^A} R_{t}^{\varepsilon} \leq \varepsilon^{\tfrac{2}{1+\beta}+\tfrac{\beta-\alpha}{2}}
\Bigr) 
\nonumber
\\
&\hspace{15pt}
+\limsup_{\varepsilon \searrow 0} {\mathbb P}
\Bigl( 
  \inf_{0 \leq t \leq t_{\varepsilon}^A \wedge \tau^{\varepsilon}(v_{0}) \wedge  \nu^{\varepsilon,\prime}(2v_{0})}
R_{t}^{\varepsilon} < r_{\star} \varepsilon^{\tfrac{2}{1+\alpha}},
\sup_{0 \leq t \leq t_{\varepsilon}^A} R_{t}^{\varepsilon} \leq \varepsilon^{\tfrac{2}{1+\beta}+\tfrac{\beta-\alpha}{2}}
\Bigr)\label{eq:mu:cestpetit:1}
\\
&\hspace{15pt} +
\limsup_{\varepsilon \searrow 0} {\mathbb P}
\Bigl(
\sup_{0 \leq t \leq t_{\varepsilon}^A} R_{t}^{\varepsilon} >\varepsilon^{\tfrac{2}{1+\beta}+\tfrac{\beta-\alpha}{2}}
\Bigr).\nonumber
\end{align}
{As for the last term in the right-hand side, we use the equation for $X^\varepsilon$ and Assumption ({\textbf{A2}}) and get:
\begin{align*}
\sup_{0\le t\le t_\varepsilon^A} R_{t}^{\varepsilon} & \leq \frac{1}{2} (r_{0}+ r_{\star}) \varepsilon^{\tfrac2{1+\alpha}} +C_0\int^{t_\varepsilon^A}_0 (R_s^\varepsilon)^\alpha ds + \varepsilon \sup_{0\le t\le t_\varepsilon^A}|B_t|\\
& \leq \frac{1}{2} (r_{0}+ r_{\star}) \varepsilon^{\frac2{1+\alpha}} +\frac12 \sup_{0\le t\le t_\varepsilon^A} R_{t}^{\varepsilon} +\bigl(2^{\alpha}C_0\bigr)^{\tfrac1{1-\alpha}} \bigl( t_\varepsilon^A\bigr)^{\tfrac{1}{1-\alpha}} + \varepsilon \sup_{0\le t\le t_\varepsilon^A}|B_t|.
\end{align*}
Therefore, we have
\begin{align*}
{\mathbb E}\Bigl[\sup_{0\le t\le t_\varepsilon^A} R_{t}^{\varepsilon}\Bigr] &\leq \bigl(r_{0}+ r_{\star} +
 (2C_0 A)^{\tfrac1{1-\alpha}}\bigr) \varepsilon^{\tfrac{2}{1+\alpha}} +4\varepsilon \E[B_{t_\varepsilon^A}^2]^{1/2}
 \leq \bigl(r_{0}+ r_{\star} +2(2C_0 A)^{\tfrac1{1-\alpha}}+4 A^{\tfrac12} \bigr) \varepsilon^{\tfrac{2}{1+\alpha}}.
\end{align*}
Applying Markov inequality and recalling that $\varepsilon^{\frac{2}{1+\alpha}}=o\bigl( \varepsilon^{\frac{2}{1+\beta}+\frac{\beta-\alpha}{2}}\bigr)$ by formula \eqref{eq:comparison_barriers}, we get that, for $A$ fixed,
\begin{equation}\label{eq:mu:cestpetit:4}
\limsup_{\varepsilon \searrow 0} {\mathbb P}
\Bigl(
\sup_{0 \leq t \leq t_{\varepsilon}^A} R_{t}^{\varepsilon} >\varepsilon^{\tfrac{2}{1+\beta}+\tfrac{\beta-\alpha}{2}}
\Bigr) = 0.
\end{equation}
}

As for the second term in the right-hand side {of \eqref{eq:mu:cestpetit:1}}, we can invoke Lemmas
\ref{lem:3:2}
and
\ref{lem:appli:BMO} with $\sigma^{\varepsilon} = \nu^{\varepsilon,\prime}(2v_{0}) \wedge \kappa^{\varepsilon,\prime}$ and $T=1$.
We deduce that there exists a constant $C$ (only depending on $v_{0}$)
such that
\begin{equation*}
\begin{split}
&\limsup_{\varepsilon \searrow 0} {\mathbb P}
\Bigl( 
  \inf_{0 \leq t \leq t_{\varepsilon}^A \wedge \tau^{\varepsilon}(v_{0}) \wedge  \nu^{\varepsilon,\prime}(2v_{0})}
R_{t}^{\varepsilon} < r_{\star} \varepsilon^{\tfrac{2}{1+\alpha}},
\sup_{0 \leq t \leq t_{\varepsilon}^A} R_{t}^{\varepsilon} \leq \varepsilon^{\tfrac{2}{1+\beta}+\tfrac{\beta-\alpha}{2}}
\Bigr)
\\
&\leq C \limsup_{\varepsilon \searrow 0} {\mathbb P}
\Bigl( 
  \inf_{0 \leq t \leq t_{\varepsilon}^A}
\bigl\vert X_{0}^{\varepsilon}
+ \varepsilon B_{t}
\bigr\vert
 < r_{\star} \varepsilon^{\tfrac{2}{1+\alpha}} 
\Bigr),
\end{split}
\end{equation*}
with $\vert X_{0}^{\varepsilon} \vert = [(r_{0}+r_{\star})/2] \varepsilon^{\frac2{1+\alpha}}$. 

Following the 
fourth step in the proof of 
Lemma \ref{lem:restarting}, we claim (at least for $r_{0}>2$):
\begin{equation*}
\begin{split}
\PP \Bigl(\inf_{0 \leq t \leq t^A_{\varepsilon}}
\vert X_{0}^{\varepsilon} + \varepsilon B_{t} \vert <    r_{\star}
\varepsilon^{\tfrac{2}{1+\alpha}}
\Bigr)
&\leq \PP \Bigl(\inf_{0 \leq t \leq A}
\bigl\vert 
\tfrac2{r_{0}+r_{\star}}
\varepsilon^{-\tfrac2{1+\alpha}}
X_{0}^{\varepsilon}
+
 B_{t} \bigr\vert  <
\tfrac2{r_{0}+r_{\star}}
 r_{\star}
\Bigr)
\\
&= \PP \Bigl(\inf_{0 \leq t \leq A }
\vert 
 e
+
 B_{t} \vert < \tfrac2{r_{0}+r_{\star}} r_{\star}
\Bigr),
\end{split}
\end{equation*}
where $e$ an arbitrary unit vector. 
When $\mu$ in \eqref{eq:lower:bound:g} tends to $0$, $r_{0}$ 
{in 
\eqref{eq:additional:label:r0} (and in the prescription 
$R_{0}^{\varepsilon} = \tfrac{1}{2} (r_{0}+ r_{\star}) \varepsilon^{2/(1+\alpha)}$)}
tends to $\infty$ and the right hand side tends to $0$. 
We deduce 
\begin{equation}
\label{eq:mu:cestpetit:2}
\limsup_{\mu \searrow 0}
\limsup_{\varepsilon \searrow 0} {\mathbb P}
\Bigl( 
  \inf_{0 \leq t \leq t_{\varepsilon}^A \wedge \tau^{\varepsilon}(v_{0}) \wedge  \nu^{\varepsilon,\prime}(2v_{0})}
R_{t}^{\varepsilon} < r_{\star} \varepsilon^{\tfrac{2}{1+\alpha}},
\sup_{0 \leq t \leq t_{\varepsilon}^A} R_{t}^{\varepsilon} \leq \varepsilon^{\tfrac{2}{1+\beta}+\tfrac{\beta-\alpha}{2}}
\Bigr)
= 0.
\end{equation}

By a similar argument, {the first term in the right-hand side of \eqref{eq:mu:cestpetit:1} satisfies:}
\begin{align*}
&\limsup_{\varepsilon \searrow 0} {\mathbb P}
\Bigl(  
R^{\varepsilon}_{\tau^{\varepsilon}(v_{0}) \wedge  \nu^{\varepsilon,\prime}(2v_{0})} \geq r_{0} \varepsilon^{\tfrac{2}{1+\alpha}}, 
\tau^{\varepsilon}(v_{0}) \wedge \nu^{\varepsilon,\prime}(2v_{0}) \leq t_{\varepsilon}^A, \ \inf_{0 \leq t \leq  \tau^{\varepsilon}(v_{0}) \wedge \nu^{\varepsilon,\prime}(2v_{0})}
R_{t}^{\varepsilon} \geq r_{\star} \varepsilon^{\tfrac{2}{1+\alpha}}, 
\\
&\hspace{300pt} \sup_{0 \leq t \leq t_{\varepsilon}^A} R_{t}^{\varepsilon} \leq \varepsilon^{\tfrac{2}{1+\beta}+\tfrac{\beta-\alpha}{2}}
\Bigr)
\\
&\leq C \limsup_{\varepsilon \searrow 0} {\mathbb P}
\Bigl( 
  \sup_{0 \leq t \leq t_{\varepsilon}^A}
\bigl\vert X_{0}^{\varepsilon}
+ \varepsilon B_{t}
\bigr\vert
 \geq r_{0} \varepsilon^{\tfrac{2}{1+\alpha}} 
\Bigr)
\\
&\leq C \limsup_{\varepsilon \searrow 0} {\mathbb P}
\Bigl( 
  \sup_{0 \leq t \leq t_{\varepsilon}^A}
\bigl\vert  \varepsilon B_{t}
\bigr\vert
 \geq \tfrac{r_{0}-r_{\star}}2 \varepsilon^{\tfrac{2}{1+\alpha}} 
\Bigr)
= C {\mathbb P}
\Bigl( 
  \sup_{0 \leq t \leq A}
\vert  B_{t}
\vert
 \geq \tfrac{r_{0}-r_{\star}}2  
\Bigr),
\end{align*}
and the last term in the right-hand side tends to $0$ as $\mu$ tends to $0$. 

We deduce
\begin{equation}
\label{eq:mu:cestpetit:3}
\begin{split}
&\limsup_{\mu \searrow 0}
\limsup_{\varepsilon \searrow 0} {\mathbb P}
\Bigl(  
R^{\varepsilon}_{\tau^{\varepsilon}(v_{0}) \wedge  \nu^{\varepsilon,\prime}(2v_{0})} \geq r_{0} \varepsilon^{\tfrac{2}{1+\alpha}}, 
\tau^{\varepsilon}(v_{0}) \leq t_{\varepsilon}^A, \ \inf_{0 \leq t \leq  \tau^{\varepsilon}(v_{0})}
R_{t}^{\varepsilon} \geq r_{\star} \varepsilon^{\tfrac{2}{1+\alpha}}, 
\\
&\hspace{250pt} \sup_{0 \leq t \leq t_{\varepsilon}^A} R_{t}^{\varepsilon} \leq \varepsilon^{\frac{2}{1+\beta}+\frac{\beta-\alpha}{2}}
\Bigr)  = 0.
\end{split}
\end{equation}
\vskip 2pt

\textit{Conclusion.}
Collecting 
\eqref{eq:conclusion:second:step:proposition:4:1}, 
\eqref{eq:mu:cestpetit:1},  {\eqref{eq:mu:cestpetit:4}, }\eqref{eq:mu:cestpetit:2}
and
\eqref{eq:mu:cestpetit:3}
and inserting in the conclusion of the second step, we deduce 
\begin{equation*}
\begin{split}
\limsup_{\mu \searrow 0}
\limsup_{\varepsilon \searrow 0}{\mathbb P}
\Bigl( g \bigl( X_{\tau^{\varepsilon}(v_{0})} \bigr) \leq \mu
\Bigr)
&\leq  \pi.
\end{split}
\end{equation*}
Since $\pi$ is arbitrary, this completes the proof.
\end{proof}

\section{Escaping from zero: Proof of Theorem \ref{thm:2:1}}
\label{se:escaping:0}

The goal of this section is to prove that the particle does leave the origin and that it does so by staying {with high probability} within the region where 
$g$ is positive. Throughout, we assume that Assumption \textbf{A} is in force.

\subsection{Lower bound for the potential}
\label{subse:lower:bound:pot}
The first step is to get a lower bound for the potential in terms of the radius, which is the precise purpose of this section. In order to state the result properly, we introduce the new notations:
\begin{equation}
\label{eq:5:1}
\begin{split}
&g_{t}^{\varepsilon} := g(X_{t}^{\varepsilon}), \quad t \geq 0,
\\
&\gamma^{\varepsilon} := \inf \bigl\{ t \geq 0 : g_{t}^{\varepsilon} = 0 \bigr\} \ ; \quad  
\xi^{\varepsilon} := \inf \Bigl\{ t \geq 0 : V_{t}^{\varepsilon} \leq \frac{V_{0}^{\varepsilon}}{2}  \Bigr\}, \ ; 
\quad 
\Xi^{\varepsilon} := \inf \bigl\{t \geq 0 : V_{t}^{\varepsilon} \geq 1 \bigr\}. 
\end{split}
\end{equation}

We start with the following statement:

\begin{lemma}
\label{lem:5:1}
Consider a collection of initial conditions $(X_{0}^{\varepsilon})_{\varepsilon >0}$ 
such that, for some $v_{0} >0$ {and some} $\varepsilon >0$,
\begin{equation*}
V_{0}^{\varepsilon} \geq v_{0} \varepsilon^2, \quad g_{0}^{\varepsilon} >0.
\end{equation*}
Then, there exists 
$v_{\star}:=v_{\star}(\textbf{\bf A})$
such that, for $v_{0} \geq v_{\star}(\textbf{\bf A})$,
\begin{equation*}
{\mathbb P} \biggl( \exists t \in (0,\gamma^{\varepsilon}] : V_{t}^{\varepsilon} \leq \frac{V_{0}^{\varepsilon}}{2} \biggr) 
\leq 
\exp \bigl( - \frac{v_{0}}{2}
\bigr). 
\end{equation*}
In particular, 
\begin{equation*}
{\mathbb P} \bigl( \gamma^{\varepsilon} < \infty \bigr) 
\leq 
\exp \bigl( - \frac{v_{0}}{2}
\bigr),
\quad \textrm{\rm and} \quad 
{\mathbb P} \biggl( \forall t >0, \ V_{t}^{\varepsilon} \geq \frac{V_{0}^{\varepsilon}}{2}, \ g_{t}^{\varepsilon} >0 \biggr) 
\geq
1- 
\exp \bigl( - \frac{v_{0}}{2}
\bigr). 
\end{equation*}
\end{lemma}

\begin{proof}
Recalling that there exists a constant $C:=C(\textbf{\bf A})$ such that
$V_{t}^{\varepsilon} \leq C \bigl( R_{t}^{\varepsilon} \bigr)^{1+\alpha}$, $t \geq 0$,
we deduce that, for $t \in [0,\xi^{\varepsilon}]$, 
\begin{equation*}
 R_{t}^{\varepsilon} \geq \bigl( \frac{v_{0}}{2 C} \bigr)^{\frac1{1+\alpha}} \varepsilon^{\frac{2}{1+\alpha}}. 
\end{equation*}
In particular, for $v_{0} \geq v_{\star}:=v_{\star}(\textbf{\bf A})$, we can follow the second case in the proof of 
Lemma \ref{lem:supermartingale:2} {(note that, {thanks to $\gamma^{\varepsilon}$}, we only need to work on the region $g>0$, in particular we do not need the additional term $\eta_\varepsilon (R^\varepsilon_t)^{1+\alpha}$)}. Combined with the first half of the same proof, we deduce that,
for $v_{0} \geq v_{\star}$, for 
$t \in [0,\xi^{\varepsilon} \wedge \gamma^{\varepsilon}]$, 
\begin{equation}
\label{eq:minoration:V:varepsilon:78}
V_{t}^{\varepsilon} \geq  V_{0}^{\varepsilon} + \int_{0}^t \frac78 \vert \nabla V_{s}^{\varepsilon} \vert^2 \ud s + \varepsilon \int_{0}^t \nabla V_{s}^{\varepsilon} \cdot \ud B_{s}.
\end{equation}
So, 
\begin{equation*}
\varepsilon^{-2} \Bigl( V_{t}^{\varepsilon} - \frac{V_{0}^{\varepsilon}}{2} \Bigr) \geq \frac{v_{0}}{2}
+ \frac{1}2 \varepsilon^{-2} \int_{0}^t \vert \nabla V_{s}^{\varepsilon} \vert^2 \ud s + \varepsilon^{-1} \int_{0}^t \nabla V_{s}^{\varepsilon} \cdot \ud B_{s}.
\end{equation*}
Therefore, for any $t >0$,
\begin{equation*}
\begin{split}
{\mathbb P} \Bigl( \xi^{\varepsilon} \leq t \wedge \gamma^{\varepsilon}  \Bigr) 
&\leq {\mathbb P}
\Bigl(
\sup_{s \in [0,t \wedge \gamma^{\varepsilon}]}\exp \biggl( - \varepsilon^{-1} \int_{0}^s \nabla V_{r}^{\varepsilon} \cdot 
\ud B_{r} - \frac12 \varepsilon^{-2} \int_{0}^s \vert \nabla V_{r}^{\varepsilon} \vert^2 \ud r \biggr) \geq \exp \bigl( \frac{v_{0}}{2} \bigr)
\Bigr)
\leq \exp \bigl( - \frac{v_{0}}{2} \bigr),
\end{split}
\end{equation*}
where we used Doob's {inequality} for the Dol\'eans-Dade martingale of $(- \varepsilon^{-1} \int_{0}^s \nabla V_{r}^{\varepsilon}
\cdot \ud B_{r} )_{s \geq 0}$. Letting $t$ tend to $\infty$, we deduce that 
\begin{equation*}
 {\mathbb P} \bigl( \xi^{\varepsilon} \leq \gamma^{\varepsilon} \bigr) 
\leq \exp \bigl( - \frac{v_{0}}{2} \bigr), 
\end{equation*}
which completes the first part of the proof.

The second part of the proof is to observe that, on the event $\{ \forall t \in (0,\gamma^{\varepsilon}],  \ 
V_{t}^{\varepsilon} \geq \tfrac12 V_{0}^{\varepsilon} \}$, the stopping time $\gamma^{\varepsilon}$
is infinite. If not, we have $0 = V_{\gamma^{\varepsilon}}^{\varepsilon} \geq \tfrac12 {V_{0}^{\varepsilon}}>0$, which is a contradiction. The last {two assertions easily follow}. 
\end{proof}

We now provide another lower bound for the potential in terms of the radius. 

\begin{proposition}
\label{prop:5:2}
Consider a sequence of initial positions $(X_{0}^{\varepsilon})_{\varepsilon >0}$ such that, for some 
$\varphi>0$, some $v_{0}>0$ {and some} $\varepsilon >0$, 
\begin{equation*}
V_{0}^{\varepsilon} \geq  \max \Bigl( v_{0} \varepsilon^2, \varphi \bigl(R_{0}^{\varepsilon}\bigr)^2 \Bigr), \quad 
g_{0}^{\varepsilon} >0. 
\end{equation*}
Then, there exist $v_\star:=v_{\star}(\textbf{\bf A})>0$, $\varphi_\star:=\varphi_{\star}(\textbf{\bf A})>0$ and $\varpi:=\varpi(\textbf{\bf A}) >0$, such that, for $v_{0} \geq v_\star$ and $\varphi \geq \varphi_{\star}$, 
\begin{equation*}
\begin{split}
\PP \biggl( \exists t \in [0, 1] : 
 V_{t}^{\varepsilon} \leq  \varpi  (R_{t}^{\varepsilon})^2 
\biggr)\leq 2 \exp \bigl( - \frac{v_{0}}{2} \bigr).
\end{split}
\end{equation*}
\end{proposition}

\begin{proof}
\textit{First step.}
We use the same notation as in \eqref{eq:5:1}.
Then, we recall that, for $v_{0} \geq v_{\star}(\textbf{\bf A})$, for $0 < t < \xi^{\varepsilon} \wedge \gamma^{\varepsilon}$, 
\begin{equation*}
\ud V_{t}^{\varepsilon} \geq \frac78 \vert \nabla V_{t}^{\varepsilon} \vert^2 \ud t + 
\varepsilon \nabla V_{t}^{\varepsilon} \cdot \ud B_{t}.
\end{equation*}
(Here and below, we use the shorten notation $\ud C_{t} \geq 0$ for 
 {a non-decreasing process $(C_{t})_{t}$}; the notation $\ud C_{t}^1 \geq \ud C_{t}^2$ is understood as $\ud (C_{t}^1 - C_{t}^2) \geq 0$.)
As for $(R^{\varepsilon})^2$, we have
\begin{equation*}
\ud (R_{t}^{\varepsilon})^2 = 2 \nabla V_{t}^{\varepsilon} \cdot X_{t}^{\varepsilon} \ud t + 
d \varepsilon^2  \ud t + 2 \varepsilon X_{t}^{\varepsilon} \cdot \ud B_{t},
\end{equation*}
where the first $d$ {in the second term of the right-hand side} denotes the dimension. 
We now use the fact that
\begin{equation*}
\begin{split}
2 \nabla V_{t}^{\varepsilon} \cdot X_{t}^{\varepsilon} &\leq 2 \vert \nabla V_{t}^{\varepsilon} \vert R_{t}^{\varepsilon} 
\leq  \vert \nabla V_{t}^{\varepsilon} \vert^2 +  (R_{t}^{\varepsilon})^2.
\end{split}
\end{equation*}
Therefore,
\begin{equation*}
\ud (R_{t}^{\varepsilon})^2 \leq  \vert \nabla V_{t}^{\varepsilon} \vert^2 
\ud t +  (R_{t}^{\varepsilon})^2 \ud t +
d \varepsilon^2  \ud t + 2 \varepsilon X_{t}^{\varepsilon} \cdot \ud B_{t},
\end{equation*}
and then
\begin{equation*}
\ud \Bigl[ \exp(-2t) (R_{t}^{\varepsilon})^2 \Bigr] \leq   \vert \nabla V_{t}^{\varepsilon} \vert^2 
\ud t   -
 \exp(-2t)
 (R_{t}^{\varepsilon})^2
  \ud t +
d \varepsilon^2 \exp(-2t)
  \ud t + 2 \varepsilon \exp(-2t)  X_{t}^{\varepsilon} \cdot \ud B_{t}.
\end{equation*}
Hence, for any $\varpi >0$ and for $0 < t < \xi^{\varepsilon} \wedge \gamma^{\varepsilon}$, 
\begin{equation*}
\begin{split}
\ud \bigl( V_{t}^{\varepsilon} - \varpi \exp(-2t) (R_{t}^{\varepsilon})^2 \bigr) 
& \geq 
(\tfrac78 - \varpi) \vert \nabla V_{t}^{\varepsilon} \vert^2 \ud t 
+
 \varpi \exp(-2t)
 (R_{t}^{\varepsilon})^2
  \ud t 
-  d \varpi \varepsilon^2 \exp(-2t)  \ud t 
\\
&\hspace{15pt} + \varepsilon\bigl( \nabla V_{t}^{\varepsilon} - 2 \varpi \exp(-2t) X_{t}^{\varepsilon} \bigr) \cdot \ud B_{t}.
\end{split}
\end{equation*}
We now recall that, for $0 < t < \xi^{\varepsilon} \wedge \gamma^{\varepsilon}$,
\begin{equation*}
\vert \nabla V_{t}^{\varepsilon} \vert^2 \geq c_{0} \frac{(V_{t}^{\varepsilon})^2}{(R_{t}^{\varepsilon})^2}
\geq \frac{c_{0} v_{0} \varepsilon^2}{2} \frac{V_{t}^{\varepsilon}}{(R_{t}^{\varepsilon})^2},
\end{equation*}
from which we get
\begin{equation*}
\begin{split}
&\ud \bigl( V_{t}^{\varepsilon} - \varpi 
\exp(- {2}t)
(R_{t}^{\varepsilon})^2 \bigr) 
\\
&\geq 
\bigl(\tfrac78 - (1+c_{0}^{-1})\varpi \bigr)
\vert \nabla V_{t}^{\varepsilon} \vert^2  \ud t +  \frac{ \varpi v_{0}}{2}
\varepsilon^2
 \frac{V_{t}^{\varepsilon}}{(R_{t}^{\varepsilon})^2} \ud t
 +
 \varpi
 \exp(-2t)
 (R_{t}^{\varepsilon})^2
  \ud t 
  - d \varpi \varepsilon^2  \exp(-2t)\ud t 
 \\
&\hspace{15pt}  +  \varepsilon \bigl( \nabla V_{t}^{\varepsilon} - 2 \varpi \exp(-2t) X_{t}^{\varepsilon} \bigr) \cdot \ud B_{t}
\\
&\geq \bigl(\tfrac78 - (1+c_{0}^{-1})\varpi\bigr)
\vert \nabla V_{t}^{\varepsilon} \vert^2 \ud t + 
 \frac{ \varpi v_{0}}{2}
\varepsilon^2
 \frac{V_{t}^{\varepsilon} - \varpi \exp(-2t)
 (R_{t}^{\varepsilon})^2}{(R_{t}^{\varepsilon})^2} \ud t  
 +
 \varpi
 \exp(-2t)
 (R_{t}^{\varepsilon})^2
  \ud t 
\\
 &\hspace{15pt} 
 + \varepsilon^2 \bigl(\frac{v_{0}}{2} \varpi^2  - d \varpi  \bigr)
 \exp(-2t) \ud t 
 +  \varepsilon \bigl( \nabla V_{t}^{\varepsilon} - 2 \varpi \exp(-2t) X_{t}^{\varepsilon} \bigr) \cdot \ud B_{t}.
\end{split}
\end{equation*}
Therefore, letting
\begin{equation*}
\Pi_{t}^{\varepsilon} := 
\exp \biggl( - \int_{0}^t \varepsilon^2 \frac{\varpi v_{0}}{2 (R_{s}^{\varepsilon})^2} \ud s \biggr),
\end{equation*}
we get, for $0 < t < \xi^{\varepsilon} \wedge \gamma^{\varepsilon}$, 
\begin{equation*}
\begin{split}
&\ud \Bigl[ \Pi_{t}^{\varepsilon}   \bigl( V_{t}^{\varepsilon} -  \varpi \exp(-2t) (R_{t}^{\varepsilon})^2 \bigr) \Bigr] 
\\
&\geq 
\Pi_{t}^{\varepsilon} \bigl(\tfrac78 - (1+c_{0}^{-1})\varpi\bigr)
\vert \nabla V_{t}^{\varepsilon} \vert^2 \ud t + \varepsilon^2 \Pi_{t}^{\varepsilon}
\bigl( \frac{v_{0}}{2} \varpi^2 - d \varpi   \bigr) \exp(-2t)  \ud t 
+
 \varpi \exp(-2t) 
\Pi_{t}^{\varepsilon}
 (R_{t}^{\varepsilon})^2
  \ud t 
\\
&\hspace{15pt} +  \varepsilon \Pi_{t}^{\varepsilon} \bigl( \nabla V_{t}^{\varepsilon} - 2 \varpi \exp(-2t) X_{t}^{\varepsilon} \bigr) \cdot \ud B_{t}.
\end{split}
\end{equation*}
\vskip 4pt

\textit{Second step.}
We now let
\begin{equation*}
M_{t}^{\varepsilon} := \int_{0}^t \Pi_{s}^{\varepsilon}\bigl( \nabla V_{s}^{\varepsilon} - 2 \varpi \exp(-2s) X_{s}^{\varepsilon} \bigr) \cdot \ud B_{s}.
\end{equation*}
Clearly,
\begin{equation*}
\begin{split}
\frac12 \frac{\ud}{\ud t} [M^{\varepsilon}]_{t} &= \frac12 	\bigl( \Pi_{t}^{\varepsilon} \bigr)^2 \bigl\vert \nabla V_{t}^{\varepsilon} - 2 \varpi \exp(-2t) X_{t}^{\varepsilon} \bigr\vert^2
\\
&= (\Pi_{t}^{\varepsilon})^2 \Bigl( \frac12 \vert \nabla V_{t}^{\varepsilon} \vert^2 + 2 \varpi^2 \exp(-4t) (R_{t}^{\varepsilon})^2 - 2 \varpi  \exp(-2t) \nabla V_{t}^{\varepsilon} \cdot X_{t}^{\varepsilon} \Bigr)
\\
&\leq (\Pi_{t}^{\varepsilon})^2 \Bigl( \frac34\vert \nabla V_{t}^{\varepsilon} \vert^2 + C \varpi^2 \exp(-2t) (R_{t}^{\varepsilon})^2 \Bigr)
\leq \Pi_{t}^{\varepsilon} \Bigl( \frac34\vert \nabla V_{t}^{\varepsilon} \vert^2 + C \varpi^2 \exp(-2t) (R_{t}^{\varepsilon})^2 \Bigr), 
\end{split}
\end{equation*}
for a {universal} constant $C>0$. To derive the last bound, we used the obvious fact that $\Pi_{t}^{\varepsilon} \leq 1$. 

In the conclusion of the first step, choose now {$\varpi >0$} such that $\tfrac78 - (1+c_{0}^{-1}) \varpi > \frac34$
and $C \varpi \leq 1$, and then 
$v_{0}$ large enough such that $ \frac{v_{0}}{2} \varpi^2 - d \varpi \geq 0$
{(in other words, 
$\varpi:=\varpi(\textbf{\bf A})$
and
$v_{0} \geq v_{\star}(\textbf{\bf A})$)
}. We get, for $0 < t < \xi^{\varepsilon} \wedge \gamma^{\varepsilon}$,
\begin{equation*}
\begin{split}
&\ud \Bigl[ \Pi_{t}^{\varepsilon}   \Bigl( V_{t}^{\varepsilon} -  \varpi \exp(-2t) (R_{t}^{\varepsilon})^2 \Bigr) \Bigr] 
\geq 
\frac12 \ud [M^{\varepsilon}]_{t}   +  \varepsilon \ud M_{t}^{\varepsilon}.
\end{split}
\end{equation*}
In particular, 
\begin{equation*}
\begin{split}
&\PP \biggl( \inf_{t \in [0,\xi^{\varepsilon} \wedge \gamma^{\varepsilon}]}
\Bigl[ \Pi_{t}^{\varepsilon}   \Bigl( V_{t}^{\varepsilon} -  \varpi \exp(-2t) (R_{t}^{\varepsilon})^2 \Bigr) \Bigr]
\leq 0
\biggr)
\leq \PP \biggl( 
\inf_{t \geq 0}
\biggl[ \frac{V_{0}^{\varepsilon} - \varpi( R_{0}^{\varepsilon})^2}{\varepsilon^2} + \frac1{2 \varepsilon^2} [M^{\varepsilon}]_{t} + \frac1{\varepsilon} M_{t}^{\varepsilon} \biggr]  \leq 0 \biggr).
\end{split}
\end{equation*}
\vskip 4pt

\textit{Third step.}
We now claim that
\begin{equation*}
\begin{split}
&\PP \biggl( \inf_{t \geq 0}
\biggl[ \frac{V_{0}^{\varepsilon} - \varpi( R_{0}^{\varepsilon})^2}{\varepsilon^2} + \frac1{2 \varepsilon^2} [M^{\varepsilon}]_{t} + \frac1{\varepsilon} M_{t}^{\varepsilon} \biggr]  \leq 0 \biggr)
\\
&=
\PP \biggl( \sup_{t \geq 0}
\biggl( - \frac{1}{2 \varepsilon^2} [M^{\varepsilon}]_{t} - 
\frac1{\varepsilon} M_{t}^{\varepsilon} \biggr) \geq \frac{V_{0}^{\varepsilon} - \varpi (R_{0}^{\varepsilon})^2}{\varepsilon^2} \biggr) 
\\
&= 
\PP \biggl( \sup_{t \geq 0}
\biggl[
\exp \biggl( - \frac{1}{2 \varepsilon^2} [M^{\varepsilon}]_{t} - 
\frac1{\varepsilon} M_{t}^{\varepsilon} \biggr) \biggr] \geq 
\exp \Bigl( \frac{V_{0}^{\varepsilon} - \varpi (R_{0}^{\varepsilon})^2}{\varepsilon^2} \Bigr) \biggr)
\\
&= 
\PP \biggl( \sup_{t \geq 0}
\biggl[
\exp \biggl( - \frac{1}{2 \varepsilon^2} [M^{\varepsilon}]_{t} - 
\frac1{\varepsilon} M_{t}^{\varepsilon} \biggr) \biggr] \geq 
\exp \bigl( \frac{v_{0}}{2} \bigr) \biggr),  
\end{split}
\end{equation*}
where we used the fact that ({recalling that $\varphi$ is given in the statement})
\begin{equation*}
V_{0}^{\varepsilon} - \varpi (R_{0}^{\varepsilon})^2 \geq V_{0}^{\varepsilon} \bigl( 1 - \varphi^{-1} \varpi \bigr)
\geq \varepsilon^2 v_{0}\bigl( 1 - \varphi^{-1} \varpi \bigr) \geq \tfrac{1}2 \varepsilon^2 v_{0},
\end{equation*}
the last inequality being true for {$\varphi \geq \varphi_{\star}(\textbf{\bf A})$}. 
We conclude from Doob's theorem that 
\begin{equation*}
\PP \biggl( \inf_{t \geq 0}
\biggl[ \frac{V_{0}^{\varepsilon} - \varpi( R_{0}^{\varepsilon})^2}{\varepsilon^2} + \frac1{2 \varepsilon^2} [M^{\varepsilon}]_{t} + \frac1{\varepsilon} M_{t}^{\varepsilon} \biggr]  \leq 0 \biggr)
\leq \exp \bigl( - \frac{v_{0}}{2} \bigr).
\end{equation*}
\vskip 4pt

\textit{Fourth step.} Combining the second and third step, we have
\begin{equation*}
\begin{split}
\PP \biggl( \inf_{t \in [0,\xi^{\varepsilon} \wedge \gamma^{\varepsilon}]}
\Bigl[ \Pi_{t}^{\varepsilon}   \Bigl( V_{t}^{\varepsilon} -  \varpi \exp(-2t) (R_{t}^{\varepsilon})^2 \Bigr) \Bigr]
\leq 0
\biggr)\leq \exp \bigl( - \frac{v_{0}}{2} \bigr).
\end{split}
\end{equation*}
By Lemma 
\ref{lem:5:1}, we deduce that 
\begin{equation*}
\begin{split}
\PP \biggl( \inf_{t \geq 0}
\Bigl[ \Pi_{t}^{\varepsilon}   \Bigl( V_{t}^{\varepsilon} -  \varpi \exp(-2t) (R_{t}^{\varepsilon})^2 \Bigr) \Bigr]
\leq 0
\biggr)\leq 2 \exp \bigl( - \frac{v_{0}}{2} \bigr),
\end{split}
\end{equation*}
and then
\begin{equation*}
\begin{split}
\PP \biggl( \exists t \in [0,1] : 
 V_{t}^{\varepsilon} \leq  \exp(-2) \varpi  (R_{t}^{\varepsilon})^2 
\biggr)\leq 2 \exp \bigl( - \frac{v_{0}}{2} \bigr),
\end{split}
\end{equation*}
which completes the proof.
\end{proof}

\subsection{Escape rate of the potential}
Here is the core of the proof of Theorem 
\ref{thm:2:1}. 
We start with the following two technical lemmas on the shape of the potential:

\begin{lemma}
\label{lem:5:3}
There exists a constant $c:=c(\textbf{\bf A})>0$ such that, for any $x$ such that $g(x) \geq a_{0}$, with $a_{0}$ as in Assumption \textbf{\bf A},
\begin{equation*}
\vert \nabla V(x) \vert^2 \geq c V(x)^{\tfrac{2\alpha}{1+\alpha}}.
\end{equation*}
\end{lemma}

\begin{proof}
We recall from 
(\textbf{A3})
 that, for $g(x) >0$,
$\vert \nabla V(x) \vert^2 \geq c_{0}^2 V(x)^2/\vert x \vert^2$. 
Now, for $g(x) \geq a_{0}$, 
$V(x) \geq a_{0} \vert x \vert^{1+\alpha}$,
that is 
${\vert x \vert^{-2}} \geq a_{0}^{\frac{2}{1+\alpha}} {V(x)^{-\frac{2}{1+\alpha}}}$,
from which the proof  is easily completed. 
\end{proof}

\begin{lemma}
\label{lem:5:4}
For any $w>0$, there exists a constant $c:=c(\textbf{\bf A},w)>0$, depending on $w$, such that, for any $x$ such that 
$V(x) \geq w \vert x \vert^2$ and 
$g(x) \in (0,a_{0})$, with $a_{0}$ and $p$ as Assumption \textbf{\bf A}, 
\begin{equation*}
\vert \nabla V(x) \vert^2 \geq c V(x)^{\tfrac{2\alpha}{1+\alpha}
+ \tfrac{(1-\alpha)p}{p+1} },
\qquad \textit{\it and,\qquad if} \  \alpha<p, \quad
\vert \nabla V(x) \vert^2 \geq c V(x)^{\tfrac{\alpha + p}{p+1}}.
\end{equation*}
\end{lemma}

\begin{proof}
By assumption (\textbf{A4}), we know that, for $g(x) \in (0,a_{0})$, 
\begin{equation*}
\vert \nabla V(x) \vert^2 \geq c_{0 }^2 L(x)^{2p} \vert x \vert^{2\alpha}, \quad 
V(x) \leq C_{0} L(x)^{p+1} \vert x \vert^{1+\alpha}.
\end{equation*}
Therefore, for a constant $c:=c(\textbf{\bf A})>0$, 
\begin{equation}
\vert \nabla V(x) \vert^2 \geq c \Bigl( \frac{V(x)}{\vert x \vert^{1+\alpha}}
\Bigr)^{\tfrac{2p}{p+1}} \vert x \vert^{2\alpha} = 
c V(x)^{\tfrac{2p}{p+1}} \vert x \vert^{2 \alpha - \tfrac{2p (1+\alpha)}{p+1}}
 =
c V(x)^{\tfrac{2p}{p+1}} \vert x \vert^{\tfrac{2 (\alpha-p)}{p+1}}. \label{eq:escape_rate_lemma2}
\end{equation}
Since $V(x) \leq a_{0} \vert x \vert^{1+\alpha}$, the first inequality above yields, for a new value of the constant $c$,
\begin{equation*}
\vert \nabla V(x) \vert^2 \geq c \Bigl( \frac{V(x)}{\vert x \vert^{1+\alpha}}
\Bigr)^{\tfrac{2p}{p+1}} V(x)^{\tfrac{2\alpha}{1+\alpha}}.
\end{equation*}
Now, by assumption, $\vert x \vert \leq w^{-1/2} V(x)^{1/2}$. Hence, allowing $c$ to depend on $w$, we get
\begin{equation*}
\vert \nabla V(x) \vert^2 \geq c \Bigl( \frac{V(x)}{V(x)^{(1+\alpha)/2}}
\Bigr)^{\tfrac{2p}{p+1}} V(x)^{\tfrac{2\alpha}{1+\alpha}} = 
c V(x)^{\tfrac{(1-\alpha)p}{p+1}  + \tfrac{2\alpha}{1+\alpha}}.
\end{equation*}
If $\alpha < p$, we go back to the inequality \eqref{eq:escape_rate_lemma2} and use the fact that
$\vert x \vert^{-1} \geq w^{1/2} V(x)^{-1/2}$. We get
\begin{equation*}
\vert \nabla V(x) \vert^2 \geq 
c V(x)^{\tfrac{2p}{p+1}} \vert x \vert^{\tfrac{2 (\alpha-p)}{p+1}}
\geq  c w^{\tfrac{p-\alpha}{p+1}}
 V(x)^{\tfrac{2p}{p+1}} V(x)^{\tfrac{\alpha-p}{p+1}}
= 
c w^{\tfrac{p-\alpha}{p+1}} V(x)^{\tfrac{\alpha + p}{p+1}}.
\end{equation*}
The proof is complete.
\end{proof}

Combining the two previous lemmas, 
we get the following result:
\begin{proposition}
\label{prop:5:6}
{Recall the notations \eqref{eq:5:1}}
and 
consider a sequence of initial positions $(X_{0}^{\varepsilon})_{\varepsilon >0}$ such that, for some 
$\varphi>0$, some $v_{0}>0$ {and some} $\varepsilon >0$,
\begin{equation*}
V_{0}^{\varepsilon} \geq  \max \Bigl( v_{0} \varepsilon^2, \varphi \bigl(R_{0}^{\varepsilon}\bigr)^2 \Bigr), \quad 
g_{0}^{\varepsilon} >0. 
\end{equation*}
{Then, there exist $v_\star:=v_\star(\textbf{\bf A})>0$, $\varphi_\star:=\varphi_\star(\textbf{\bf A})>0$,
$\psi :=\psi(\alpha,p) \in (0,1)$
 such that, for $v_{0} \geq v_\star$ and $\varphi \geq \varphi_{\star}$,
 we can find $c:=c(\textbf{\bf A},\varphi)$
satisfying}
\begin{equation*}
{\mathbb P}
\biggl(\forall 
t\in [0, 1 \wedge   \Xi^{\varepsilon}], 
\ 
V_{t}^\varepsilon  \geq 
\Bigl( (1-\psi) ct \Bigr)^{\tfrac{1}{1-\psi}}
\biggr) \geq 1 - 4\exp \bigl( - \frac{v_{0}}{2} \bigr).
\end{equation*}.
\end{proposition}

\begin{remark}
\label{rem:5:6}
Observe that the simple fact that $V^{\varepsilon}$ remains positive on the event appearing in the above statement implies that 
\begin{equation*}
{\mathbb P}
\biggl(\forall 
t\in [0, 1 \wedge   \Xi^{\varepsilon}], 
\ 
V_{t}^\varepsilon  \geq 
\Bigl( (1-\psi) ct \Bigr)^{\tfrac{1}{1-\psi}}, \ g_{t}^{\varepsilon} >0
\biggr) \geq 1 - 4\exp \bigl( - \frac{v_{0}}{2} \bigr),
\end{equation*} 
\end{remark}

\begin{proof}
\textit{First step.}
Following the statement of Proposition \ref{prop:5:2}, we consider $v_{\star}$ and $\varphi_{\star}$
and then $\varpi$ {as therein}.  

Inspired by the proof of 
Proposition
\ref{prop:5:2}, we
 let
\begin{equation*}
\begin{split}
\chi^{\varepsilon} := \inf \Bigl\{ t \geq 0 : V_{t}^{\varepsilon} \leq \tfrac{\varpi}2 (R_{t}^{\varepsilon})^2 \Bigr\} \wedge 
\inf \Bigl\{t \geq 0 : V_{t}^{\varepsilon} \geq 1 \Bigr\}.
\end{split}
\end{equation*} 
Recalling the definition of $\gamma^{\varepsilon}$ in \eqref{eq:5:1}, we deduce from 
Lemmas \ref{lem:5:3} and \ref{lem:5:4} that, for $t \leq \gamma^{\varepsilon} \wedge \chi^{\varepsilon}$, 
\begin{equation}
\label{eq:5:5:1}
\bigl\vert \nabla V_{t}^{\varepsilon}
\bigr\vert^2 \geq c \bigl( V_{t}^{\varepsilon} \bigr)^\psi,
\end{equation}
where {$c:=c(\textbf{\bf A},\varpi)$ and with}
\begin{equation*}
\psi := \left\{
\begin{array}{lll}
\frac{2 \alpha}{1+\alpha} + \frac{(1-\alpha)p}{p+1} 
&<
\frac{2 \alpha}{1+ \alpha} + \frac{1 - \alpha}{2} = \frac{1 - \alpha^2 + 4\alpha}{2+2\alpha} \le 1
&\text{if } \alpha \in [ p,1) \\
\max \left\{ \frac{2 \alpha}{1+\alpha}, \frac{\alpha+p}{p+1} \right\} = \frac{\alpha+p}{p+1} &<1 &\text{if } \alpha<p.
\end{array}
\right.
\end{equation*}
(Notice that we here used the crucial fact that, 
up to time $\chi^{\varepsilon}$, $V^{\varepsilon}$ remains less than $1$.)
\vskip 4pt

\textit{Second step.}
Also, recall from the proof of 
Proposition \ref{prop:5:2}
 that, for 
{$v_{0} \geq v_{\star}(\textbf{\bf A})$}, and for 
$t \leq \gamma^{\varepsilon} \wedge \xi^{\varepsilon}$, 
\begin{equation*}
\ud V_{t}^{\varepsilon} \geq \frac78 \vert \nabla V_{t}^{\varepsilon} \vert^2 \ud t + 
\varepsilon \nabla V_{t}^{\varepsilon} \cdot \ud B_{t}.
\end{equation*}
So, for $v_{0}$ large enough and for a new value of the constant $c$ in 
\eqref{eq:5:5:1}, we deduce that, for $t \leq \gamma^{\varepsilon} \wedge \xi^{\varepsilon} \wedge \chi^{\varepsilon}$, 
\begin{equation*}
\ud V_{t}^\varepsilon  \geq \Bigl( \frac12 \vert \nabla V_{t}^\varepsilon  \vert^2 + 
c  (V_{t}^\varepsilon )^{\psi}\Bigr) \ud t + 
\varepsilon \nabla V_{t}^\varepsilon  \cdot \ud B_{t}.
\end{equation*}
In order to conclude, we let
{(following the proof of Proposition \ref{prop:5:2})}:
\begin{equation*}
N_{t}^\varepsilon  := \int_{0}^t  \nabla V_{s}^\varepsilon  \cdot \ud B_{s}.
\end{equation*}
On the event
\begin{equation}
\label{eq:5:5:2}
\biggl\{ \inf_{t \geq 0}
\biggl[
\frac{v_{0}}2 + \frac{1}{\varepsilon}N_{t}^\varepsilon  + \frac{1}{2\varepsilon^2} [N^\varepsilon ]_{t} \biggr] \geq 0 \biggr\},
\end{equation}
we get, for $t \leq \gamma^{\varepsilon} \wedge \xi^{\varepsilon} \wedge \chi^{\varepsilon}$, 
\begin{equation*}
V_{t}^\varepsilon  \geq \frac{\varepsilon^2 v_{0}}{2}  + \int_{0}^t 
c  (V_{s}^\varepsilon )^{\psi} \ud s.
\end{equation*}
We then compare $V^{\varepsilon}$ with the solution of the ODE
$\dot{y}_{t} = c y_{t}^{\psi}$, $y_{0}= \frac{\varepsilon^2 v_{0}}2$,
whose solution is given by 
\begin{equation*}
 y_{t} =  \Bigl( y_{0}^{1-\psi} + (1-\psi) ct \Bigr)^{\frac{1}{1-\psi}}, \quad t \geq 0.
\end{equation*}
By a standard comparison argument, we get that, under the condition 
\eqref{eq:5:5:2}, 
for all $t  \leq \gamma^{\varepsilon} \wedge \xi^{\varepsilon} \wedge \chi^{\varepsilon}$,
\begin{equation*}
V_{t}^\varepsilon  \geq
\Bigl( (1-\psi) ct \Bigr)^{\frac{1}{1-\psi}}. 
\end{equation*}
In particular, under the condition 
\eqref{eq:5:5:2}, we must have $\xi^{\varepsilon} \wedge \chi^{\varepsilon} <
\gamma^{\varepsilon}$. And so, the above inequality holds for
$t\leq  \xi^{\varepsilon} \wedge \chi^{\varepsilon}$. 
\vskip 4pt

\textit{Third step.}
We know that 
\begin{equation*}
{\mathbb P}
\biggl( \exists t \geq 0 : 
\frac{v_{0}}2 + \frac{1}{\varepsilon}N_{t}^\varepsilon  + \frac{1}{2\varepsilon^2} [N^\varepsilon ]_{t} \leq 0
\biggr)
=
{\mathbb P}
\biggl( \exists t \geq 0 : - \frac{1}{\varepsilon}N_{t}^\varepsilon  - \frac{1}{2\varepsilon^2} [N^\varepsilon ]_{t} \geq \frac{v_{0}}{2}
\biggr) \leq \exp \bigl( - \frac{v_{0}}{2} \bigr). 
\end{equation*}
and then, by the second step, 
\begin{equation*}
{\mathbb P}
\biggl(
\forall 
t\in [0, \xi^{\varepsilon} \wedge \chi^{\varepsilon}], 
\ 
V_{t}^\varepsilon  \geq 
\Bigl( (1-\psi) ct \Bigr)^{\frac{1}{1-\psi}}
\biggr) \geq 1 - \exp \bigl( - \frac{v_{0}}{2} \bigr).
\end{equation*}
By Lemma 
\ref{lem:5:1} (recalling that $v_{0}$ can be taken as large as needed), 
\begin{equation*}
{\mathbb P}
\biggl(
\forall 
t\in [0,   \chi^{\varepsilon}], 
\ 
V_{t}^\varepsilon  \geq 
\Bigl( (1-\psi) ct \Bigr)^{\frac{1}{1-\psi}}
\biggr) \geq 1 - 2\exp \bigl( - \frac{v_{0}}{2} \bigr).
\end{equation*}
Then,
by Proposition 
\ref{prop:5:2} ({choosing $\varphi \geq \varphi_{\star}(\textbf{\bf A})$ large enough}),
\begin{equation*}
{\mathbb P}
\biggl(\forall 
t\in [0, 1 \wedge   \chi^{\varepsilon}], 
\ 
V_{t}^\varepsilon  \geq 
\Bigl( (1-\psi) ct \Bigr)^{\frac{1}{1-\psi}}, 
\ 
V_{t}^\varepsilon  \geq 
\varpi (R_{t}^{\varepsilon})^2
\biggr) \geq 1 - 4\exp \bigl( - \frac{v_{0}}{2} \bigr).
\end{equation*}
{On the above event, it holds $1 \wedge   \chi^{\varepsilon} = 1 \wedge   \Xi^{\varepsilon}$ and so}
\begin{equation*}
{\mathbb P}
\biggl(\forall 
t\in [0, 1 \wedge   \Xi^{\varepsilon}], 
\ 
V_{t}^\varepsilon  \geq 
\Bigl( (1-\psi) ct \Bigr)^{\frac{1}{1-\psi}}
\biggr) \geq 1 - 4\exp \bigl( - \frac{v_{0}}{2} \bigr).
\end{equation*}
{The proof is complete}. 
\end{proof}

\subsection{Conclusion}
\label{subse:5:5}
We now end up the proof of Theorem \ref{thm:2:1}. 

To do so,
we consider $v_\star>0$ and $\varphi_\star>0$  as in the statement of Proposition 
\ref{prop:5:6}. 
We then invoke 
Propositions \ref{prop:hitting:potential} and 
\ref{prop:hitting:g>0} for $v_{0} \geq v_{\star}$. For $(A_{\varepsilon} := \vert \ln(\varepsilon) \vert)_{\varepsilon >0}$, we let 
$(t_{\varepsilon} := A_{\varepsilon} \varepsilon^{2(1-\alpha)/(1+\alpha)})_{\varepsilon >0}$. 
We know that 
\begin{equation*}
\liminf_{\varepsilon \searrow 0} {\mathbb P}
\bigl( \tau^{\varepsilon}(v_{0}) \leq t_{\varepsilon} \bigr) = 1,
\end{equation*} 
with $\tau^{\varepsilon}(v_{0})$ as in 
\eqref{eq:3:1}. Now, for a given ${\pi} >0$, we know that, for $\mu \leq \mu_{\star}:=\mu_{\star}(\pi,v_{0})$, 
\begin{equation*}
\liminf_{\varepsilon \searrow 0} {\mathbb P}
\bigl( g(X_{\tau^{\varepsilon}(v_{0})}^{\varepsilon}) \geq \mu \bigr) \geq 1 - \pi. 
\end{equation*} 
Hence, 
\begin{equation*}
\liminf_{\varepsilon \searrow 0} {\mathbb P}
\bigl( \tau^{\varepsilon}(v_{0}) \leq t_{\varepsilon}, \ g(X_{\tau^{\varepsilon}(v_{0})}^{\varepsilon}) \geq \mu \bigr) \geq 1 - \pi. 
\end{equation*} 
Notice now that, on the event $\{ \tau^{\varepsilon}(v_{0}) \leq t_{\varepsilon},  g(X_{\tau^{\varepsilon}(v_{0})}^{\varepsilon}) \geq \mu\}$, 
\begin{equation*}
\varepsilon^2 v_{0} = V_{\tau^{\varepsilon}(v_{0})}^{\varepsilon} \geq \mu \bigl(R_{\tau^{\varepsilon}(v_{0})}^{\varepsilon}\bigr)^{1+\alpha}.
\end{equation*}
In particular, 
\begin{equation*}
\bigl( R_{\tau^{\varepsilon}(v_{0})}^{\varepsilon} \bigr)^{-1} \geq \Bigl(  \frac{\mu}{\varepsilon^{2} v_{0}}\Bigr)^{\tfrac1{1+\alpha}},
\end{equation*}
so that
\begin{equation*}
V^{\varepsilon}_{\tau^{\varepsilon}(v_{0})} \geq \mu \bigl(R_{\tau^{\varepsilon}(v_{0})}^{\varepsilon}\bigr)^{\alpha-1} \bigl(R_{\tau^{\varepsilon}(v_{0})}^{\varepsilon}\bigr)^{2}
\geq \mu \Bigl(  \frac{\mu}{\varepsilon^{2} v_{0}}\Bigr)^{\tfrac{1-\alpha}{1+\alpha}} \bigl(R_{\tau^{\varepsilon}(v_{0})}^{\varepsilon}\bigr)^{2}.
\end{equation*}
Hence, {for 
$v_{0} \geq v_{\star}$, 
$\mu \leq \mu_{\star}$} and $\varphi>0$, 
\begin{equation*}
\liminf_{\varepsilon \searrow 0} {\mathbb P}
\Bigl( 
 \tau^{\varepsilon}(v_{0}) \leq t_{\varepsilon} , \
 g(X_{\tau^{\varepsilon}(v_{0})}^{\varepsilon}) \geq \mu, \ 
V^{\varepsilon}_{\tau^{\varepsilon}(v_{0})}\geq \varphi  \bigl(R_{\tau^{\varepsilon}(v_{0})}^{\varepsilon}\bigr)^{2}
\Bigr) \geq 1 - \pi. 
\end{equation*}
In particular, we can choose $\varphi=\varphi_{\star}$ in the above inequality. 
Combining with Proposition \ref{prop:5:6} through Markov property, this leads to 
{\begin{equation*}
\begin{split}
&\liminf_{\varepsilon \searrow 0}
{\mathbb P}
\biggl(\forall 
t\in [\tau^{\varepsilon}(v_{0}), (1+\tau^{\varepsilon}(v_{0})) \wedge   \Xi^{\varepsilon}], 
\ 
V_{t}^\varepsilon  \geq 
\Bigl( (1-\psi) c(t - \tau^{\varepsilon}(v_{0}))_{+} \Bigr)^{\frac{1}{1-\psi}}, 
\ 
g_{t}^{\varepsilon} >0
\biggr) 
\\
&\geq (1- \pi) \Bigl( 1 - 4\exp \bigl( - \frac{v_{0}}{2} \bigr) \Bigr),
\end{split}
\end{equation*}
and so $(1+t_{\varepsilon}) \wedge   \Xi^{\varepsilon}$ with $1 \wedge   \Xi^{\varepsilon}$)}
\begin{equation*}
\liminf_{\varepsilon \searrow 0}
{\mathbb P}
\biggl(\forall 
t\in [t_{\varepsilon}, 1 \wedge   \Xi^{\varepsilon}], 
\ 
V_{t}^\varepsilon  \geq 
\Bigl( (1-\psi) c(t - t_{\varepsilon})_{+} \Bigr)^{\frac{1}{1-\psi}}, 
\ 
g_{t}^{\varepsilon} >0
\biggr) \geq (1- \pi) \Bigl( 1 - 4\exp \bigl( - \frac{v_{0}}{2} \bigr) \Bigr),
\end{equation*}
where $c:=c(\textbf{\bf A})$ and where, as before,
$\Xi^{\varepsilon} = \inf \bigl\{t \geq 0 : V_{t}^{\varepsilon} \geq 1 \bigr\}$. 
Since $c$ is independent of $v_{0}$ and $\pi$ is arbitrary, we conclude that 
\begin{equation*}
\liminf_{\varepsilon \searrow 0}
{\mathbb P}
\biggl(\forall 
t\in [t_{\varepsilon}, 1 \wedge   \Xi^{\varepsilon}], 
\ 
V_{t}^\varepsilon  \geq 
\Bigl( (1-\psi) c(t - t_{\varepsilon})_{+} \Bigr)^{\frac{1}{1-\psi}}, 
\ 
g_{t}^{\varepsilon} >0
\biggr) =1. 
\end{equation*}
Reapplying Lemma 
\ref{lem:5:1}, we see that there are infinitesimal chances (as $\varepsilon \searrow 0$) for 
the potential to pass below $1/2$ after $\Xi^{\varepsilon}$. This suffices to say that 
\begin{equation*}
\liminf_{\varepsilon \searrow 0}
{\mathbb P}
\biggl(\forall 
t\in [t_{\varepsilon},1], 
\ 
V_{t}^\varepsilon  \geq 
\min \Bigl[ \tfrac12 , \Bigl( (1-\psi) c(t - t_{\varepsilon})_{+} \Bigr)^{\frac{1}{1-\psi}} \Bigr], 
\ 
g_{t}^{\varepsilon} >0
\biggr) =1. 
\end{equation*}
The conclusion easily follows.

\section{Following steepest lines}
\label{se:6}

The conclusion of the previous section is that the particle must escape from the origin and that it does 
so by staying inside the region 
$\{ g>0\}$. Meanwhile, it does not say anything regarding the typical sites of the region $\{g >0\}$ that the particle visits.
This is the question that we address in this section. To make it precise, we focus here on the typical directions that the particle follows. 
Throughout, we assume that both Assumptions \textbf{A} and \textbf{B} are in force.

\subsection{Distance to the boundary}
\label{subse:distance:boundary}
We recall the useful notation
$\bigl(
g_{t}^{\varepsilon} = g(X_{t}^{\varepsilon})
\bigr)_{t \geq 0}$. 
Similarly, we let
\begin{equation*}
\bigl(
\nabla g_{t}^{\varepsilon} := \nabla g(X_{t}^{\varepsilon}),
\Delta g_{t}^{\varepsilon} := \Delta g(X_{t}^{\varepsilon})
\bigr)_{t \geq 0}.
\end{equation*}
Since $g$ is ${\mathcal C}^{1,1}$ on any compact subset not containing $0$ and 
since the process $X^{\varepsilon}$ does not come back to $0$ with probability 1, It\^o's formula yields
\begin{equation*}
\begin{split}
\ud g_{t}^{\varepsilon} = \Bigl( \nabla g_{t}^{\varepsilon} \cdot \nabla V_{t}^{\varepsilon}
+ \frac{ \varepsilon^2}{2} \Delta g_{t}^{\varepsilon} \Bigr) \ud t 
+ \varepsilon \nabla g_{t}^{\varepsilon} \cdot \ud B_{t}, \quad t >0. 
\end{split}
\end{equation*}
Recalling the definition 
$\gamma^{\varepsilon} = \inf \bigl\{ t \geq 0 : g_{t}^{\varepsilon} = 0 \bigr\}$, 
we know that, for $t \in {(}0,\gamma^{\varepsilon})$, 
\begin{equation}
\label{eq:6:1}
\begin{split}
\nabla V_{t}^{\varepsilon} &= 
 \bigl( R_{t}^{\varepsilon} \bigr)^{1+\alpha}
\nabla g_{t}^{\varepsilon} 
+ (1+\alpha)  \bigl( R_{t}^{\varepsilon} \bigr)^{\alpha}  g_{t}^{\varepsilon} \frac{X_{t}^{\varepsilon}}{R_{t}^{\varepsilon}}
= \bigl( R_{t}^{\varepsilon} \bigr)^{1+\alpha}
\nabla g_{t}^{\varepsilon} 
+ (1+\alpha)  \bigl( R_{t}^{\varepsilon} \bigr)^{\alpha}  g_{t}^{\varepsilon} \theta_{t}^{\varepsilon},
\end{split} 
\end{equation}
where we let
\begin{equation}
\label{eq:def:thetatvarepsilon}
\theta_{t}^{\varepsilon} := \frac{X_{t}^{\varepsilon}}{R_{t}^{\varepsilon}} {\mathbf 1}_{\{ R_{t}^{\varepsilon} >0\}}, 
\quad t \geq 0.   
\end{equation}
We end up with
\begin{lemma}
\label{lem:decomposition:gt}
For any $t \in (0,\gamma^{\varepsilon})$,
\begin{equation}
\label{eq:decomposition:gt}
\begin{split}
\ud g_{t}^{\varepsilon} =  \Bigl[ \bigl( R_{t}^{\varepsilon} \bigr)^{1+\alpha} \bigl\vert \nabla g_{t}^{\varepsilon} \bigr\vert^2 
+ 
(1+\alpha)  \bigl( R_{t}^{\varepsilon} \bigr)^{\alpha}  g_{t}^{\varepsilon} 
\bigl( \nabla g_{t}^{\varepsilon} \cdot
\theta_{t}^{\varepsilon}
\bigr)
+ \frac{\varepsilon^2}{2}
\Delta g_{t}^{\varepsilon}
\Bigr] \ud t 
+ \varepsilon \nabla g_{t}^{\varepsilon} \cdot \ud B_{t}. 
\end{split}
\end{equation}
\end{lemma}
{Obviously, Lemma \ref{lem:decomposition:gt} extends to any random initial time
(given in the form of a stopping time) provided that $\gamma^{\varepsilon}$ is defined accordingly.}
The next lemma identifies the behavior of the drift in the above decomposition when the particle is inside the region 
$\{g >0\}$ but close to the boundary.

\begin{lemma}
\label{lem:tau:g}
Given $v_{0}>0$ and $a>0$, {we can find a constant $C:= {C(\textbf{\bf A},\textbf{\bf B}) }>1$ such that, 
for $a \leq a_{\star}(\textbf{\bf A},\textbf{\bf B})$ and 
 $v_{0} \geq v_{\star}(\textbf{\bf A},\textbf{\bf B})$, and for any $\varepsilon >0$ and  
$t >0$ such that 
\begin{equation*}
V_{t}^{\varepsilon} \geq  {\tfrac12} \varepsilon^2 v_{0}, \quad
\text{and}
\quad 
g_{t}^{\varepsilon} \leq a,
\end{equation*}
it holds that}
  \begin{equation*}
\begin{split}
\frac1{C} \frac{V_{t}^{\varepsilon}}{g_{t}^{\varepsilon}}  \bigl\vert \nabla g_{t}^{\varepsilon} \bigr\vert^2
&\leq
 \bigl( R_{t}^{\varepsilon} \bigr)^{1+\alpha} \bigl\vert \nabla g_{t}^{\varepsilon} \bigr\vert^2 
+ 
(1+\alpha)  \bigl( R_{t}^{\varepsilon} \bigr)^{\alpha}  g_{t}^{\varepsilon} 
\bigl( \nabla g_{t}^{\varepsilon} \cdot {\theta_{t}^{\varepsilon}}
\bigr)
+ \frac{\varepsilon^2}{2}
\Delta g_{t}^{\varepsilon} 
\leq {C} 
\frac{V_{t}^{\varepsilon}}{g_{t}^{\varepsilon}}   \bigl\vert \nabla g_{t}^{\varepsilon} \bigr\vert^2.
\end{split}
\end{equation*}
\end{lemma}
{Notice that the term in the middle in the above inequality exactly matches} the drift in the decomposition 
\eqref{eq:decomposition:gt}. 

\begin{proof}
We recall the following two bounds from (\textbf{A4}):
\begin{equation*}
\begin{split}
&c_{0} L(x)^{p+1}  \vert x \vert^{1+\alpha} \leq V(x) \leq C_{0}
 L(x)^{p+1}  \vert x \vert^{1+\alpha},
 \quad  c_{0} L(x)^p  \vert x \vert^{\alpha} \leq \bigl\vert \nabla V(x)
 \bigr\vert \leq C_{0}
 L(x)^p  \vert x \vert^{\alpha},
\end{split}
\end{equation*}
for  {$x \not =0$} such that $g(x) \in (0,a_{0})$. Obviously, the first equation says that 
\begin{equation}
\label{eq:6:2}
c_{0} L(x)^{p+1}  
 \leq g(x) \leq C_{0}
L(x)^{p+1},
\end{equation}
and, inserting {the corresponding form of} \eqref{eq:6:1}, the second one yields 
 \begin{equation*}
 \begin{split}
&\Bigl( c_{0} 
L(x)^p  
- C_{0} (1+\alpha) 
L(x)^{p+1}  
\Bigr) \vert x \vert^{\alpha}
\leq \bigl\vert \nabla g(x)
 \bigr\vert 
 \vert x \vert^{1+\alpha}
\leq 
\Bigl( C_{0} L(x)^p  
+
C_{0} (1+\alpha) 
L(x)^{p+1}  
\Bigr) \vert x \vert^{\alpha},
\end{split}
\end{equation*}
for $x$ such that $g(x) \in (0,a_{0})$. By downsizing the value of $a_{0}$, 
we can render 
$L(x)$
as small as needed. Modifying $c_{0}$ and $C_{0}$ accordingly, we obtain
 \begin{equation*}
 \begin{split}
& c_{0} 
L(x)^p  
 \vert x \vert^{\alpha}
\leq \bigl\vert \nabla g(x)
 \bigr\vert 
 \vert x \vert^{1+\alpha}
\leq 
 C_{0} L(x)^p  
 \vert x \vert^{\alpha},
 \end{split}
\end{equation*}
for $g(x) \in (0,a_{0})$. Using \eqref{eq:6:2} again, 
we get, for a constant $C=C(\textbf{A},\textbf{B}) \geq 1$,
 \begin{equation}
 \label{eq:6:4}
 \begin{split}
&\frac1{C}
\frac{1}{\vert x \vert}
g^{\tfrac{p}{p+1}}(x) 
\leq \bigl\vert \nabla g(x)
 \bigr\vert 
\leq 
 C
 \frac{1}{\vert x \vert}
g^{\tfrac{p}{p+1}}(x),
 \end{split}
\end{equation}
for $g(x) \in (0,a_{0})$, and then, allowing the constant $C$ to increase from line to line,
 \begin{equation}
 \label{eq:6:3}
 \begin{split}
&\frac1{C}
\frac{1}{\vert x \vert^2}
g^{\tfrac{2p}{p+1}}(x) 
\leq \bigl\vert \nabla g(x)
 \bigr\vert^2 
\leq 
 C
 \frac{1}{\vert x \vert^2}
g^{\tfrac{2p}{p+1}}(x) . 
 \end{split}
\end{equation}
Recalling that $V(x) = g(x) \vert x \vert^{1+\alpha}$, we deduce
 \begin{equation*}
 \begin{split}
&\frac1{C}
\frac{V(x)}{\vert x \vert^2}
g^{\tfrac{p-1}{p+1}}(x) 
\leq \bigl\vert \nabla g(x)
 \bigr\vert^2 \vert x \vert^{1+\alpha}
\leq 
 C
 \frac{V(x)}{\vert x \vert^2}
g^{\tfrac{p-1}{p+1}}(x),
 \end{split}
\end{equation*}
from which we get, by choosing $a$ accordingly in the statement, 
\begin{equation}
\label{eq:6:5}
\frac1{C} \frac{V_{t}^{\varepsilon}}{(R_{t}^{\varepsilon})^2} \bigl( g_{t}^{\varepsilon} \bigr)^{\tfrac{p-1}{p+1}}
\leq 
\bigl( R_{t}^{\varepsilon} \bigr)^{1+\alpha} \bigl\vert \nabla g_{t}^{\varepsilon} \bigr\vert^2 
\leq 
C
 \frac{V_{t}^{\varepsilon}}{(R_{t}^{\varepsilon})^2} \bigl( g_{t}^{\varepsilon} \bigr)^{\tfrac{p-1}{p+1}}, 
 \quad
 {\text{for }t >0 \, : \, 
 {g_{t}^{\varepsilon} \leq a}.}
\end{equation}
Also, by \eqref{eq:6:4}, we obtain
\begin{equation}
\label{eq:6:6}
\begin{split}
\bigl( R_{t}^{\varepsilon} \bigr)^{\alpha}  g_{t}^{\varepsilon} 
\bigl\vert \nabla g_{t}^{\varepsilon} \cdot
\theta_{t}^{\varepsilon}
\bigr\vert 
\leq 
\bigl( R_{t}^{\varepsilon} \bigr)^{\alpha}  g_{t}^{\varepsilon} 
\bigl\vert \nabla g_{t}^{\varepsilon} 
\bigr\vert 
&\leq 
C 
\frac{V_{t}^{\varepsilon}}{(R_{t}^{\varepsilon})^2}
\bigl( g_{t}^{\varepsilon} \bigr)^{\tfrac{p}{p+1}}
\leq
C
\bigl( g_{t}^{\varepsilon} \bigr)^{\tfrac{1}{p+1}}
\bigl( R_{t}^{\varepsilon} \bigr)^{1+\alpha} \bigl\vert \nabla g_{t}^{\varepsilon} \bigr\vert^2, 
\end{split}
\end{equation}
for  {$t>0$ such that 
$g_{t}^{\varepsilon} \leq a$}.
Lastly, by following \eqref{eq:6:1} and using  \eqref{eq:6:4}  {again}, we have
\begin{equation*}
\begin{split}
\Bigl\vert \Delta V(x) -
 \vert x \vert^{1+\alpha}
\Delta g(x) \Bigr\vert &\leq C \vert x \vert^{\alpha}
\bigl\vert 
\nabla g(x)
\bigr\vert  + C \vert x \vert^{\alpha-1} g(x)
\leq C  \vert x \vert^{\alpha-1} g(x)^{\tfrac{p}{p+1}},
\end{split}
\end{equation*}
for $g(x) \in (0,a_{0})$. 
By assumption (\textbf{B1}) and equation \eqref{eq:6:2}, we deduce that 
\begin{equation*}
\begin{split}
 \vert x \vert^{1+\alpha}\bigl\vert 
\Delta g(x) \bigr\vert &\leq C \vert x \vert^{\alpha-1} \Bigl(  g(x)^{\tfrac{p-1}{p+1}} +   g(x)^{\tfrac{p}{p+1}}\Bigr)
\leq C \vert x \vert^{\alpha-1}  g(x)^{\tfrac{p-1}{p+1}},
\end{split}
\end{equation*}
and then, for 
 ${t >0 \
 \text{such that} \ 
 V_{t}^{\varepsilon} \geq  {\tfrac12} \varepsilon^2 v_{0}
\ 
\text{and}
\
g_{t}^{\varepsilon} \leq a}$, 
\begin{equation*}
\begin{split}
\varepsilon^2 \, \bigl\vert \Delta g_{t}^{\varepsilon} \bigr\vert \leq C \frac{\varepsilon^2}{(R_{t}^{\varepsilon})^2}
\bigl( g_{t}^{\varepsilon} \bigr)^{\tfrac{p-1}{p+1}}
&= C \frac{\varepsilon^2}{V_{t}^{\varepsilon}} \frac{V_{t}^{\varepsilon}}{(R_{t}^{\varepsilon})^2}
\bigl( g_{t}^{\varepsilon} \bigr)^{\tfrac{p-1}{p+1}}
\leq \frac{C}{v_{0}} \frac{V_{t}^{\varepsilon}}{(R_{t}^{\varepsilon})^2}
\bigl( g_{t}^{\varepsilon} \bigr)^{\tfrac{p-1}{p+1}}.
\end{split}
\end{equation*}
By \eqref{eq:6:5}, we deduce that 
\begin{equation}
\label{eq:6:7}
\begin{split}
\varepsilon^2 \, \bigl\vert \Delta g_{t}^{\varepsilon}\bigr\vert \leq \frac{C}{v_{0}} \bigl( R_{t}^{\varepsilon} \bigr)^{1+\alpha} \bigl\vert \nabla g_{t}^{\varepsilon} \bigr\vert^2.
\end{split}
\end{equation}
We complete the proof by collecting 
\eqref{eq:6:6} and \eqref{eq:6:7}, by choosing $v_{0}$ large enough 
{and $a$ small enough}
and by using the fact that $(V_{t}^{\varepsilon}/g_{t}^{\varepsilon}) \vert \nabla g_{t}^{\varepsilon} \vert^2 = (R_{t}^{\varepsilon})^{1+\alpha} \vert \nabla g_{t}^{\varepsilon} \vert^2$. 
\end{proof}

The next step is to prove that 
the probability that the process $(g_{t}^{\varepsilon})_{t \geq 0}$ 
reaches a given threshold in infinitesimal time converges to 1 as 
$\varepsilon$ to $0$. 

\begin{proposition}
\label{prop:hitting:g:a}
Consider a tight collection of initial conditions 
$(X_{0}^{\varepsilon})_{\varepsilon >0}$ (in the sense that the collection of the laws of these random variables is tight) such that, for some $v_{0} >0$, for all $\varepsilon >0$,
$V_{0}^{\varepsilon} \geq v_{0} \varepsilon^2$ and 
$g_{0}^{\varepsilon} >0$. 
Assume also that, for any $A>0$, 
\begin{equation}
\label{eq:se:6:condition:initial:radius}
\lim_{\varepsilon \searrow 0}
{\mathbb P}
\Bigl( 
V_{0}^{\varepsilon} \geq A (R_{0}^{\varepsilon})^2 \Bigr)=1.
\end{equation}
Then, for any $\pi >0$, we can find  {two positive thresholds} $a_{\star}:=a_{\star}(\textbf{\bf A},\textbf{\bf B}) >0$ and $v_{\star}:=v_{\star}(\textbf{\bf A},\textbf{\bf B},\pi) >0$ 
and {a sequence of infinitesimal times $(t_{\varepsilon})_{\varepsilon >0}$, only depending on the parameters in assumptions \textbf{\bf A} and \textbf{\bf B}}, 
such that, for $a \leq a_{\star}$
and $v_{0} \geq v_{\star}$,
\begin{equation*}
 \liminf_{\varepsilon \searrow 0} {\mathbb P} \bigl( \exists t \in [0,t_{\varepsilon}] : g_{t}^{\varepsilon} \geq a 
  \bigr) \geq  {1-\pi}. 
 \end{equation*}
\end{proposition}

\begin{proof}
\textit{First step.}
{Let 
\begin{equation*}
\Gamma^{\varepsilon} := \inf \bigl\{ t \geq 0 : V_{t}^{\varepsilon} \leq \tfrac{1}{2} V_{0}^{\varepsilon} \bigr\} 
\wedge \inf \bigl\{ t \geq 0 : g_{t}^{\varepsilon} \geq a \bigr\}.
\end{equation*}
Recall also}
from the proof of Lemma 
\ref{lem:5:1}
 that there exists a constant {$c:=c(\textbf{\bf A},\textbf{\bf B})>0$ such that, 
for $v_{0} \geq v_{\star}(\textbf{\bf A},\textbf{\bf B})$ 
and 
$a \leq a_{\star}(\textbf{\bf A},\textbf{\bf B})$}, 
and  for $0 < t < \Gamma^{\varepsilon} \wedge \gamma^{\varepsilon}$,  
\begin{equation}
\label{eq:Vtvarepsilon:near:boundary}
\begin{split}
\ud V_{t}^{\varepsilon} &\geq \tfrac78 \vert \nabla V_{t}^{\varepsilon} \vert^2 \ud t  + \varepsilon \nabla V_{t}^{\varepsilon} \cdot \ud B_{t}
\geq c \bigl( g_{t}^{\varepsilon} \bigr)^{-\tfrac{1}{p+1}} V_{t}^{\varepsilon} \frac{\vert \nabla V_{t}^{\varepsilon} \vert }{R_{t}^{\varepsilon}} \ud t + 
\tfrac12 \vert \nabla V_{t}^{\varepsilon} \vert^2 dt + \varepsilon \nabla V_{t}^{\varepsilon} \cdot \ud B_{t},
\end{split}
\end{equation}
where we used 
(\textbf{A4}) and \eqref{eq:6:2} to pass from the first to the second line. 

We rewrite
\eqref{eq:Vtvarepsilon:near:boundary} in the form:
\begin{equation*}
\begin{split}
V_{t}^{\varepsilon} \geq \tfrac12 V_{0}^{\varepsilon} + c \int_{0}^t 
\bigl( g_{s}^{\varepsilon} \bigr)^{-\tfrac{1}{p+1}}
   \frac{\vert \nabla V_{s}^{\varepsilon} \vert}{R_{s}^{\varepsilon}} V_{s}^{\varepsilon} \ud s + {\mathcal V}_{t}^{\varepsilon},
\end{split}
\end{equation*}
with
\begin{equation*}
{\mathcal V}_{t}^{\varepsilon} := \tfrac12 V_{0}^{\varepsilon} 
 + \tfrac12
 \int_{0}^t \vert \nabla V_{s}^{\varepsilon} \vert^2 \ud s + 
 \varepsilon \int_{0}^t \nabla V_{s}^{\varepsilon}\cdot \ud B_{s}. 
\end{equation*}
Following the proof of Lemma \ref{lem:5:1}, we recall that
\begin{equation*}
{\mathbb P} \biggl( \forall t \geq 0, \ {\mathcal V}_{t}^{\varepsilon} \geq 0 \biggr) 
\geq 1 - {\mathbb E} \Bigl[ \exp \Bigl( - \frac{V_{0}^{\varepsilon}}{2 \varepsilon^2} \Bigr)
\Bigr] \geq  
1 - \exp \bigl( - \frac{v_{0}}{2} \bigr). 
\end{equation*}
So, with probability greater than $1 - \exp( - v_{0}/2)$, we have
\begin{equation*}
V_{t}^{\varepsilon} \geq \tfrac12 V_{0}^{\varepsilon} + c \int_{0}^t 
\bigl( g_{s}^{\varepsilon} \bigr)^{-\tfrac{1}{p+1}}
   \frac{\vert \nabla V_{s}^{\varepsilon} \vert}{R_{s}^{\varepsilon}} V_{s}^{\varepsilon} \ud s,
\end{equation*}
which, by a standard variant of Gronwall's lemma, yields
\begin{equation*}
V_{t}^{\varepsilon} \geq \tfrac12 V_{0}^{\varepsilon}
\exp \biggl( c \int_{0}^t 
\bigl( g_{s}^{\varepsilon} \bigr)^{-\tfrac{1}{p+1}} \frac{\vert \nabla V_{s}^{\varepsilon} \vert}{R_{s}^{\varepsilon}}
\ud s \biggr), \quad t  \in[0, \Gamma^{\varepsilon} \wedge \gamma^{\varepsilon}].
\end{equation*}
Notice from (\textbf{A4}) and 
\eqref{eq:6:2}
 and from the fact that the potential (and hence the radius) is lower bounded up until $\Gamma^{\varepsilon}$ together with the fact that $p \geq 1$ that the integrand in the above integral is well-integrable.
\vspace{5pt}

\textit{Second step.}
Recall the expansion of 
$\ud \bigl( R_{t}^{\varepsilon} \bigr)^2$
from \eqref{eq:V:R2}
and {deduce that, for two large enough positive constants $k:=k(\textbf{\bf A},\textbf{\bf B})$ and $c':=c'(\textbf{\bf A},\textbf{\bf B})>0$, we can find a constant $C:=C(\textbf{\bf A},\textbf{\bf B}) \geq 0$
 such that}
\begin{equation*}
\begin{split}
&\ud \Bigl( \exp(-c' t)  \bigl[ \bigl( R_{t}^{\varepsilon} \bigr)^2 + k \varepsilon^2 \bigr] \Bigr)
\\
&\leq 2  \exp(- c' t) X_{t}^{\varepsilon} \cdot \nabla V_{t}^{\varepsilon} \ud t
- \tfrac{1}{2} c' \exp( - c' t)  \bigl[ \bigl( R_{t}^{\varepsilon} \bigr)^2 + k \varepsilon^2 \bigr]   \ud t
+
2 \varepsilon \exp(- c' t) X_{t}^{\varepsilon} \cdot \ud B_{t}
\\
&\leq C \exp(- c' t)   \bigl( R_{t}^{\varepsilon} \bigr)^2    \frac{\vert \nabla V_{t}^{\varepsilon} \vert}{R_{t}^{\varepsilon}}  \ud t
 - \tfrac{1}{2} c' \exp( - c' t)  \bigl[ \bigl( R_{t}^{\varepsilon} \bigr)^2 + k \varepsilon^2 \bigr]   \ud t
+
2 \varepsilon \exp(- c' t) X_{t}^{\varepsilon} \cdot \ud B_{t},
\end{split}
\end{equation*}
which we rewrite in the form
\begin{equation*}
\begin{split}
\exp(-c' t)  \bigl[ \bigl( R_{t}^{\varepsilon} \bigr)^2 + k \varepsilon^2 \bigr] 
&\leq 
 2  \bigl[ \bigl( R_{0}^{\varepsilon} \bigr)^2 + k \varepsilon^2 \bigr] 
+ \int_{0}^t 
C \exp(- c' s)  \bigl[ \bigl( R_{s}^{\varepsilon} \bigr)^2 + k \varepsilon^2 \bigr]     \frac{\vert \nabla V_{s}^{\varepsilon} \vert}{R_{s}^{\varepsilon}}  \ud s
+
{\mathcal R}_{t}^{\varepsilon},
\end{split}
\end{equation*}
where
\begin{equation*}
\begin{split}
{\mathcal R}_{t}^{\varepsilon} &= -  
\bigl[ \bigl( R_{0}^{\varepsilon} \bigr)^2 + k \varepsilon^2 \bigr] - \tfrac{1}{2} c' \int_{0}^t \exp( - c' s)   \bigl[ \bigl( R_{s}^{\varepsilon} \bigr)^2 + k \varepsilon^2 \bigr]   \ud s
  + 2
\varepsilon \int_{0}^t \exp(- c' s) X_{s}^{\varepsilon} \cdot \ud B_{s}.
\end{split}
\end{equation*}
As before, we can {modify $c'$ so} that
\begin{equation*}
{\mathbb P} \biggl( \exists t \geq 0, \ {\mathcal R}_{t}^{\varepsilon} \geq 0 \biggr)
\leq {\mathbb E} \Bigl[ \exp \Bigl( - \frac{(R_{0}^{\varepsilon})^2}{\varepsilon^2} - k  \Bigr) \Bigr] \leq \exp(-k). 
\end{equation*}
On the complementary of the event in the above left-hand side, 
\begin{equation*}
\begin{split}
\exp(-c' t)  \bigl[ \bigl( R_{t}^{\varepsilon} \bigr)^2 + k \varepsilon^2 \bigr] 
&\leq 
 2 
 \bigl[ \bigl( R_{0}^{\varepsilon} \bigr)^2 + k \varepsilon^2 \bigr] 
+ \int_{0}^t 
C \exp(- c' s)   \bigl[ \bigl( R_{s}^{\varepsilon} \bigr)^2 + k \varepsilon^2 \bigr]     \frac{\vert \nabla V_{s}^{\varepsilon} \vert}{R_{s}^{\varepsilon}}  \ud s
\\
&\leq  2  \bigl[ \bigl( R_{0}^{\varepsilon} \bigr)^2 + k \varepsilon^2 \bigr] 
\exp \biggl( C \int_{0}^t 
\frac{\vert \nabla V_{s}^{\varepsilon} \vert}{R_{s}^{\varepsilon}}  \ud s
\biggr), 
\quad t  \in[0, \Gamma^{\varepsilon} \wedge \gamma^{\varepsilon}].
\end{split}
\end{equation*}
We deduce that, for any $\pi >0$, for {$k \geq k_{\star}(\textbf{\bf A},\textbf{\bf B},\pi)$}, the above is true with probability greater than $1-\pi$. 
\vspace{2pt}

\textit{Third step.}
As long as $t \leq \Gamma^{\varepsilon}$, $g^{\varepsilon}$ remains less than $a$. 
Therefore, the conclusion of the first step is that, for
$a \leq a_{\star}(\textbf{\bf A},\textbf{\bf B})$
and  
$v_{0} \geq v_{\star}(\textbf{\bf A},\textbf{\bf B})$,   
\begin{equation*}
V_{t}^{\varepsilon} \geq \tfrac12 V_{0}^{\varepsilon}
\exp \biggl( c 
a^{-\tfrac{1}{p+1}}
\int_{0}^t 
 \frac{\vert \nabla V_{s}^{\varepsilon} \vert}{R_{s}^{\varepsilon}}
ds \biggr),
\quad t  \in[0, \Gamma^{\varepsilon} \wedge \gamma^{\varepsilon}],
\end{equation*}
with probability greater than $1-\exp(-v_{0}/2)$, 
whilst the conclusion of the second step is that, 
for any $\pi >0$, we can choose $k:=k(\textbf{\bf A},\textbf{\bf B},\pi)$ and $c':=c'(\textbf{\bf A},\textbf{\bf B})$ such that 
  the following holds true with probability greater than 
$1-\pi$:
\begin{equation*}
\begin{split}
\exp(-c' t)  \bigl[ \bigl( R_{t}^{\varepsilon} \bigr)^2 + k \varepsilon^2 \bigr] 
&\leq  2  \bigl[ \bigl( R_{0}^{\varepsilon} \bigr)^2 + k \varepsilon^2 \bigr] 
\exp \biggl( C \int_{0}^t 
\frac{\vert \nabla V_{s}^{\varepsilon} \vert}{R_{s}^{\varepsilon}}  \ud s
\biggr) 
\\
&\leq 
2  \bigl[ \bigl( R_{0}^{\varepsilon} \bigr)^2 + k \varepsilon^2 \bigr] 
\exp \biggl( c a^{-\tfrac{1}{p+1}} \int_{0}^t 
\frac{\vert \nabla V_{s}^{\varepsilon} \vert}{R_{s}^{\varepsilon}}  \ud s
\biggr)^{\tfrac{C}{c} a^{1/{(p+1)}}},
\quad t  \in[0, \Gamma^{\varepsilon} \wedge \gamma^{\varepsilon}], 
\end{split}
\end{equation*}
and then, with probability greater than $1 - \pi - \exp(-v_{0}/2)$, 
\begin{equation*}
\begin{split}
\exp(-c' t)  \bigl[ \bigl( R_{t}^{\varepsilon} \bigr)^2 + k \varepsilon^2 \bigr] 
\leq 
2   \bigl[ \bigl( R_{0}^{\varepsilon} \bigr)^2 + k \varepsilon^2 \bigr] 
 \Bigl( \frac{2V_{t}^{\varepsilon}}{V_{0}^{\varepsilon}}
\Bigr)^{\tfrac{C}{c} a^{1/({p+1})}},
\quad t  \in[0, \Gamma^{\varepsilon} \wedge \gamma^{\varepsilon}].
\end{split}
\end{equation*}
Using the fact that, up until $\Gamma^{\varepsilon} \wedge \gamma^{\varepsilon}$, 
$2V^{\varepsilon}$ remains above $V_{0}^{\varepsilon}$, 
we deduce that, for any 
given $q,\pi >0$, we can 
choose 
{$c':=c'(\textbf{\bf A},\textbf{\bf B})$, 
$k:=k(\textbf{\bf A},\textbf{\bf B},\pi)$, 
$a \leq a_{\star}(\textbf{\bf A},\textbf{\bf B},q)$
and
$v_{0} \geq v_{\star}(\textbf{\bf A},\textbf{\bf B},\pi)$}
such that, with probability
greater than $1-\pi$, for $t \in [0,(\ln(2)/c')\wedge \Gamma^{\varepsilon} \wedge \gamma^{\varepsilon} ]$, 
\begin{equation*}
V_{t}^{\varepsilon} \geq \frac12 V_{0}^{\varepsilon} \biggl( 
\frac1{4} \frac{( R_{t}^{\varepsilon} )^2 + k \varepsilon^2  }{( R_{0}^{\varepsilon} )^2 + k \varepsilon^2   }\biggr)^{q},
\end{equation*}
and then, for any exponent $q' >0$, 
\begin{equation*}
\frac{V_{t}^{\varepsilon}}{(R_{t}^{\varepsilon})^{q'}} \geq
\bigl( \frac14 \bigr)^{1+q}
\bigl( R_{t}^{\varepsilon} \bigr)^{2q-q'}
 \frac{V_{0}^{\varepsilon}}{\bigl[( R_{0}^{\varepsilon} )^2 + k \varepsilon^2  \bigr]^{{q}} }.
\end{equation*}
Now, recall that, for a new value of the {constant $c$}, 
$R_{t}^{\varepsilon} \geq c \bigl( V_{t}^{\varepsilon} \bigr)^{\frac{1}{1+\alpha}}$, 
from which we deduce (allowing the value of the constant $c$ to vary from line to line
{and to depend on $q$, 
namely 
$c:=c(\textbf{\bf A},\textbf{\bf B},q, {q'})$})
\begin{equation*}
\begin{split}
\frac{V_{t}^{\varepsilon}}{(R_{t}^{\varepsilon})^{q'}} &\geq c
\bigl( V_{t}^{\varepsilon} \bigr)^{\tfrac{2q-q'}{1+\alpha}}
 \frac{V_{0}^{\varepsilon}}{\bigl[ ( R_{0}^{\varepsilon} )^2 + k \varepsilon^2   \bigr]^q}, \quad 
 t \in \bigl[0,\tfrac{\ln(2)}{c'}\wedge \Gamma^{\varepsilon} \wedge \gamma^{\varepsilon} \bigr],
\end{split}
\end{equation*}
with probability greater than $1-\pi$,  {at least if $2q-q' \geq 0$}. 
Therefore, invoking
\eqref{eq:se:6:condition:initial:radius}
together with the conclusion of Proposition
\ref{prop:5:6}
 and assuming that $2q -q' \geq 1$, 
we deduce that,
for any $\pi >0$,
 {for $v_{0} \geq v_{\star}(\textbf{\bf A},\textbf{\bf B},k,\pi)$
and for 
$\varepsilon$
small enough},  
 with probability greater than {$1 - 2\pi$},
for $t \in [0, 1 \wedge (\ln(2)/c')\wedge \Gamma^{\varepsilon} \wedge \gamma^{\varepsilon} \wedge 
\Xi^{\varepsilon}]$, 
\begin{equation*}
\begin{split}
\frac{V_{t}^{\varepsilon}}{(R_{t}^{\varepsilon})^{q'}} &\geq c t^{c(q,q')} \bigl( V_{0}^{\varepsilon}\bigr)^{1-q}, 
 \end{split}
\end{equation*}
for a constant $c(q,q'):=c(\textbf{\bf A},\textbf{\bf B},{\pi},q,q')>0$. And, then, for any $\eta >0$, 
for any $t \in [0, 1 \wedge (\ln(2)/c')\wedge \Gamma^{\varepsilon} \wedge \gamma^{\varepsilon} \wedge 
\Xi^{\varepsilon}]$,
\begin{equation*}
\begin{split}
\int_{0}^t 
\Bigl( 
\frac{V_{s}^{\varepsilon}}{(R_{s}^{\varepsilon})^{q'}} 
\Bigr)^{\eta}
\ud s &\geq c 
t^{1 + \eta c(q,q')} \bigl( V_{0}^{\varepsilon}\bigr)^{\eta(1-q)}
\geq c t^{1 + \eta c(q,q')} \bigl( v_{0} {\varepsilon}^2 \bigr)^{\eta(1-q)}.
 \end{split}
\end{equation*}
Without any loss of generality, we can assume that $q'<2$ and then choose $\eta >1$ such that 
$q' \eta =2$. Since $t \leq \Xi^{\varepsilon}$, we have 
$(V_{s}^{\varepsilon})^{\eta} \leq V_{s}^{\varepsilon}$ for 
$s \in [0,t]$, with $t$ as above. We deduce that, 
with probability greater than ${1- 2\pi}$, 
for any $t \in [0, 1 \wedge (\ln(2)/c')\wedge \Gamma^{\varepsilon} \wedge \gamma^{\varepsilon} \wedge 
\Xi^{\varepsilon}]$,
\begin{equation*}
\begin{split}
\int_{0}^t 
\frac{V_{s}^{\varepsilon}}{(R_{s}^{\varepsilon})^{2}} 
\ud s &\geq  c t^{1 + \tfrac{2 c(q,q')}{q'}} \bigl( v_{0} {\varepsilon}^2 \bigr)^{\tfrac{2(1-q)}{q'}}.
 \end{split}
\end{equation*}
By choosing $q>1$, we can easily find a collection $(t_{\varepsilon})_{\varepsilon >0}$, converging to zero with $\varepsilon$, such that, for $\varepsilon \leq \varepsilon_{\star} := \varepsilon_{\star}(c,v_{\star})$,  
\begin{equation*}
 {\int_{0}^{t_{\varepsilon}}
\frac{V_{s}^{\varepsilon}}{(R_{s}^{\varepsilon})^2} 
\ud s \geq 
 \varepsilon^{\tfrac{1-q}{q'}}},
\end{equation*}
on an event of the form $\{t_{\varepsilon} \leq \Gamma^{\varepsilon} \wedge \gamma^{\varepsilon} \wedge \Xi^{\varepsilon}\} \cap A_{\varepsilon}$, with $\PP(A_{\varepsilon}) \geq {1 - 
2\pi}$, {and for the same range of values for the various parameters as before}. For instance, we can choose
$$t_{\varepsilon} = \varepsilon^{\tfrac{q-1}{q'+2c(q,q')}},$$ 
for $\varepsilon >0$. 
\vskip 4pt

\textit{Conclusion.}
{Recall from the first lines in the previous step that,}
with probability greater than $1 - \exp(-v_{0}/2)$, 
\begin{equation*}
V_{t}^{\varepsilon} \geq \tfrac12 
v_{0} \varepsilon^2
\exp \biggl( c 
a^{-\tfrac{1}{p+1}}
\int_{0}^t 
 \frac{V_{s}^{\varepsilon}}{( R_{s}^{\varepsilon} )^2}
\ud s \biggr),
\quad t  \in[0, \Gamma^{\varepsilon} \wedge \gamma^{\varepsilon}].
\end{equation*}
Hence, the previous step says that, on an event of the form $\{ t_{\varepsilon} \leq \Gamma^{\varepsilon} \wedge \gamma^{\varepsilon} \wedge \Xi^{\varepsilon} \} \cap A_{\varepsilon}$,
with ${\mathbb P}(A_{\varepsilon}) \geq  {1 - 3\pi}$
({and for the same values of parameters as before}), 
\begin{equation*}
V_{t_{\varepsilon}}^{\varepsilon}
\geq 
\tfrac12 
v_{0} \varepsilon^2
\exp \Bigl( c 
a^{-\tfrac{1}{p+1}}
\varepsilon^{\tfrac{1-q}{q'}} \Bigr),
\end{equation*}
and so
\begin{equation*}
V_{0}^{\varepsilon} + C t_{\varepsilon} + \varepsilon
\sup_{t \in [0, t_{\varepsilon}]} \vert B_{t} \vert 
\geq 
\tfrac12 
v_{0} \varepsilon^2
\exp \Bigl( c 
a^{-\tfrac{1}{p+1}}
\varepsilon^{\tfrac{1-q}{q'}} \Bigr).
\end{equation*}
As the right hand side tends to $\infty$ as $\varepsilon$ tends to $0$ and since 
the collection of the laws of the $(V_{0}^{\varepsilon})_{\varepsilon >0}$'s is tight, we obviously have
\begin{equation*}
\lim_{\varepsilon \searrow 0}
{\mathbb P} \biggl( V_{0}^{ \varepsilon} + C t_{\varepsilon} + \varepsilon
\sup_{t \in [0, t_{\varepsilon}]} \vert B_{t} \vert 
\geq 
\tfrac12 
v_{0} \varepsilon^2
\exp \Bigl( c 
a^{-\tfrac{1}{p+1}}
\varepsilon^{\tfrac{1-q}{q'}} \Bigr) \biggr) =0,
\end{equation*}
and then
\begin{equation}
\label{eq:new:label:revision:section:6}
\limsup_{\varepsilon \searrow 0}
{\mathbb P}
\Bigl( t_{\varepsilon} \leq \Gamma^{\varepsilon} \wedge \gamma^{\varepsilon} \wedge 
\Xi^{\varepsilon} \Bigr) \leq  { 3\pi}.
\end{equation}
{Observe} now that 
$\Gamma^{\varepsilon} {=} \xi^{\varepsilon} \wedge {\zeta^{\varepsilon}}$, 
with $\xi^{\varepsilon}$ as in \eqref{eq:5:1} and 
\begin{equation*}
\begin{split}
&{\zeta^{\varepsilon}} := \inf \bigl\{ t \geq 0 : g_{t}^{\varepsilon} \geq a \bigr\}.
\end{split}
\end{equation*}
By Lemma 
\ref{lem:5:1}, we have, {for $v_{0}$ large enough,}
\begin{equation*}
\limsup_{\varepsilon \searrow 0}
{\mathbb P}
\Bigl( \xi^{\varepsilon} \vee \Xi^{\varepsilon} \vee \gamma^{\varepsilon} \leq t_{\varepsilon}
\Bigr)  \leq \exp\bigl( - \frac{v_{0}}{4} \bigr),
\end{equation*}
because $\lim_{\varepsilon \searrow 0}
{\mathbb P}(\Xi^{\varepsilon} \leq t_{\varepsilon}) = 0$, 
which shows 
{from 
\eqref{eq:new:label:revision:section:6}}
that ({for the same ranges of parameters as before})
\begin{equation*}
\liminf_{\varepsilon \searrow 0}
{\mathbb P}
\Bigl(  {\zeta}^{\varepsilon} \leq t_{\varepsilon}
\Bigr)  \ge  {1 -4\pi}.
\end{equation*}
 {This completes the proof}.

\end{proof}

We now prove more. We prove that, in fact, {once it has reached 
a given threshold 
$a$, the process 
 $g^{\varepsilon}$ remains above $a/2$ with high probability provided that $V^\varepsilon$ itself is large enough on the scale of order $\varepsilon^2$}.

\begin{lemma}
\label{lem:staying:g:above:a}
Consider a collection of initial conditions 
$(X_{0}^{\varepsilon})_{\varepsilon >0}$ such that, 
{for some $a>0$ and $v_{0}>0$, for all $\varepsilon >0$, $V_{0}^{\varepsilon} \geq {v_{0} \varepsilon^{2}}$
and $g_{0}^{\varepsilon} \geq a$. Then, 
there exist $a_{\star} := a_{\star}(\textbf{\bf A},\textbf{\bf B}) >0$
and $v_{\star} := v_{\star}(\textbf{\bf A},\textbf{\bf B})>0$
 such that, whenever $a \leq a_{\star}$ and $v_{0} \geq v_{\star}$, it holds, for a constant $c:=c(\textbf{\bf A},\textbf{\bf B}) \geq 1$,}
\begin{equation*}
{ {\mathbb P} \bigl( \forall t \geq 0, \ g_{t}^{\varepsilon} \geq \frac{a}2
\bigr) \geq 1 -  
\exp \bigl( - \frac{v_{0}}{c} \bigr).}
\end{equation*}
\end{lemma}


\begin{proof}
By Lemma \ref{lem:5:1}, 
we already know that, for  {$\varepsilon >0$} and 
for {$v_{0} \geq v_{\star}:=v_{\star}(\textbf{\bf A},\textbf{\bf B})$ (the value of $v_{\star}$ being allowed to increase from line to line)}, 
\begin{equation*}
{{\mathbb P} \Bigl( \forall t \geq 0, \ V_{t}^{\varepsilon} \geq \tfrac12 V_{0}^{\varepsilon}  \Bigr) 
\geq 1 - 
\exp \bigl( - \frac{v_{0}}{2}
\bigr).}
\end{equation*}
So, up to a small event, we can work on the event 
 {$\{ \forall t \geq 0, V_{t}^{\varepsilon}
\geq  \varepsilon^2 v_{0}/2 \} 
$}, for a given $v_{0} \geq {v_{\star}}$. We then call $\varrho^{\varepsilon} := \inf \{ t >0 : g_{t}^{\varepsilon} \leq \tfrac{a}{2} \} $. 

By Lemmas 
\ref{lem:decomposition:gt}
and
\ref{lem:tau:g}, we know that, {for 
$v_{0} \geq v_{\star}$ and $a
\leq a_{\star}:=a_{\star}(\textbf{\bf A},\textbf{\bf B})$} and if 
$\varrho^{\varepsilon} < \infty$, there exist two time instants
$s,t \in [0,\varrho^{\varepsilon}]$ ({choose $t=\varrho^{\varepsilon}$ and 
$s$ as the last time before $t$ when $g_{s}^{\varepsilon}=a$; of course, $s$ is not a stopping time}), such that 
\begin{equation*}
- \frac{a}{2}  \geq
c \int_{s}^t \frac{V_{r}^{\varepsilon}}{g_{r}^{\varepsilon}}  \bigl\vert \nabla g_{r}^{\varepsilon} \bigr\vert^2
\ud r + 
  \varepsilon \int_{s}^t 
\nabla g_{r}^{\varepsilon} \cdot \ud B_{r}
\geq 
\frac{c  v_0}{2 a}\varepsilon^2 \int_{s}^t   \bigl\vert \nabla g_{r}^{\varepsilon} \bigr\vert^2
\ud r + 
  \varepsilon \int_{s}^t 
\nabla g_{r}^{\varepsilon} \cdot \ud B_{r},
\end{equation*}
for a constant $c:=c(\textbf{\bf A},\textbf{\bf B}) >0$. 
In particular, the following holds true on the event $\{ \varrho^{\varepsilon}  < \infty\}$:
\begin{equation*}
 \sup_{t \geq 0} \biggl\{
 - 
\frac{c v_{0}}{2a} \varepsilon^2 \int_{0}^t  \bigl\vert \nabla g_{r}^{\varepsilon} \bigr\vert^2
\ud r
-
\varepsilon \int_{0}^t \nabla g_{r}^{\varepsilon} \cdot \ud B_{r}
\biggr\} \geq \frac{a}{4},
\end{equation*}
and then, multiplying by $c v_{0}/a$, 
\begin{equation*}
 \sup_{t \geq 0} \biggl\{
 - 
\frac{c^2 v_{0}^2}{2a^2} \varepsilon^2 \int_{0}^t \bigl\vert \nabla g_{r}^{\varepsilon} \bigr\vert^2
\ud r
-
\frac{c v_{0}}{a} \varepsilon \int_{0}^t \nabla g_{r}^{\varepsilon} \cdot \ud B_{r}
\biggr\} \geq \frac{c v_{0}}{4}.
\end{equation*}

Taking exponential  {in both sides} in the above inequality 
and using Doob's {inequality} for the Dol\'eans-Dade martingale
of 
$(-
a^{-1} c v_{0} \varepsilon \int_{0}^t \nabla g_{r}^{\varepsilon} \cdot \ud B_{r})_{t \geq 0}$
(see for instance Lemma \ref{lem:5:1} for a similar use), we complete the proof. 
\end{proof}

So, the conclusion of this subsection is pretty clear: From Proposition 
\ref{prop:hitting:g:a}, the particle, {whenever it starts from a potential $v_{0} \varepsilon^2$ with $v_{0}$ large enough}, reaches (with high probability) the level 
$a$ in infinitesimal time. 
Then, Lemma
\ref{lem:staying:g:above:a}
says that, 
when restarting from this point, there is almost no chance that 
$g^{\varepsilon}$ passes below $a/2$. 
\vspace{4pt}

\subsection{Deterministic approximation of the angle}
{Recall from \eqref{eq:def:thetatvarepsilon}
the definition of $\theta^{\varepsilon}$, which we call the angle of the particle.} We claim: 

\begin{lemma}
\label{lem:decomposition:angle}
Consider a collection of non-zero initial conditions 
$(X_{0}^{\varepsilon})_{\varepsilon >0}$. Then,
for any $\varepsilon >0$, 
with probability 1, 
the process $( \theta_{t}^{\varepsilon})_{0 \leq t \leq {\boldsymbol \rho}_{0}^{\varepsilon}}$, {with ${\boldsymbol \rho}_{0}^{\varepsilon} := \inf \{ t \geq 0 : R_{t}^{\varepsilon} = 0\}$}, has the decomposition:
\begin{equation}
\label{eq:ito:angle}
\begin{split}
\ud \theta_{t}^{\varepsilon}  &= 
\frac{1}{R_{t}^{\varepsilon}}
 \bigl[
  \nabla V_{t}^{\varepsilon}
- \bigl( \nabla V_{t}^{\varepsilon}
\cdot \theta_t^{\varepsilon}\bigr)
\theta_t^{\varepsilon} \bigr] \ud t -  \frac{d-1}{2} \bigl( \frac{\varepsilon}{R_{t}^{\varepsilon}}
\bigr)^2 
\theta_t^{\varepsilon}  \ud t
 +  \frac{\varepsilon}{R_{t}^{\varepsilon}} \bigl[ 
dB_{t} -
\theta_t^{\varepsilon}
\bigl( \theta_t^{\varepsilon} \cdot \ud B_{t} \bigr)
\bigr], \quad t \geq 0. 
\end{split}
\end{equation} 
\end{lemma}
The reader should not worry about the definition of the angle when the process 
$X^{\varepsilon}$ cancels. Indeed, with probability 1, the latter does not visit $0$.

Of course, {$\theta_{t}^{\varepsilon}=X_{t}^{\varepsilon}/R_{t}^{\varepsilon}$} is an element of the sphere ${\mathbb S}^{d-1}$. 
In this regard, the right-hand side in
\eqref{eq:ito:angle} must be regarded as the infinitesimal description of an It\^o process with values on the sphere. Observe in particular that the main term
\begin{equation*}
\frac{1}{R_{t}^{\varepsilon}}
 \bigl[
  \nabla V_{t}^{\varepsilon}
- \bigl( \nabla V_{t}^{\varepsilon}
\cdot \theta_t^{\varepsilon}\bigr)
\theta_t^{\varepsilon} \bigr]
\end{equation*}
is the component of 
$\nabla V_{t}^{\varepsilon}/R_{t}^{\varepsilon}$
in the tangent space to ${\mathbb S}^{d-1}$ at $\theta_{t}^{\varepsilon}$, 
which is precisely the term that appears in (\textbf{B2}). 
\begin{proof}[Proof of Lemma \ref{lem:decomposition:angle}]
{Recall the expansion of 
$\ud \bigl( R_{t}^{\varepsilon} \bigr)^2$
from \eqref{eq:V:R2}.}
Let now $f(x)=x^{-1/2}$, for $x>0$. Clearly, $f'(x)= - \frac12 x^{-3/2}$ and 
$f''(x) = \frac{3}{4} x^{-5/2}$. So, It\^o's formula yields (for 
$t \in [0,{\boldsymbol \rho}_{0}^{\varepsilon})$)
\begin{equation*}
\begin{split}
\ud (R_{t}^{\varepsilon})^{-1} &= - (R_{t}^{\varepsilon})^{-3} \bigl(  \nabla V_{t}^{\varepsilon}
\cdot X_{t}^{\varepsilon} + \frac{d}{2} \varepsilon^2 \bigr) \ud t
 + \frac{3}{2} \varepsilon^2 (R_{t}^{\varepsilon})^{-3} dt
 - \varepsilon (R_{t}^{\varepsilon})^{-3} X_{t}^{\varepsilon} \cdot \ud B_{t}
\\
&=  - (R_{t}^{\varepsilon})^{-3} \bigl(  \nabla V_{t}^{\varepsilon}
\cdot X_{t}^{\varepsilon} + \frac{d-3}{2} \varepsilon^2 \bigr) \ud t
- \varepsilon (R_{t}^{\varepsilon})^{-3} X_{t}^{\varepsilon} \cdot \ud B_{t}.
\end{split}
\end{equation*}
Hence, 
\begin{equation*}
\begin{split}
\ud \bigl( \theta_t^{\varepsilon} \bigr) &= 
 (R_{t}^{\varepsilon})^{-1} \bigl[
  \nabla V_{t}^{\varepsilon}
- \bigl( \nabla V_{t}^{\varepsilon}
\cdot \theta_t^{\varepsilon}\bigr)
\theta_t^{\varepsilon} \bigr]
 -  \frac{d-1}{2} \bigl( \frac{\varepsilon}{R_{t}^{\varepsilon}}
\bigr)^2 
\theta_t^{\varepsilon}  \ud t
+ \frac{\varepsilon}{R_{t}^{\varepsilon}} \bigl[ \ud B_{t} -
\theta_t^{\varepsilon}
\bigl( \theta_t^{\varepsilon} \cdot \ud B_{t} \bigr)
\bigr].
\end{split}
\end{equation*}
This completes the proof.
\end{proof}
%
%

Returning to Lemma
\ref{lem:decomposition:angle}, we see from (\textbf{B2}) that, when $g_{t}^{\varepsilon} >0$,  
 $  \nabla V_{t}^{\varepsilon}
- \bigl( \nabla V_{t}^{\varepsilon}
\cdot \theta_t^{\varepsilon}\bigr)
\theta_t^{\varepsilon} $
has the form
\begin{equation*}
{\DTheta} \bigl(   {\theta_{t}^{\varepsilon}}\bigr)
\bigl( R_{t}^{\varepsilon}
\bigr)^{\alpha} + {\eta_{t}^{\varepsilon,\prime}} \bigl( R_{t}^{\varepsilon} \bigr)^{\alpha},
\end{equation*} 
where {$\DTheta$} is Lipschitz continuous on ${\mathbb S}^{d-1}$
and 
where, {for $\delta_{\star} := \delta_{\star}(\textbf{A},\textbf{B})$}, there exists a constant $C:=C(\textbf{\bf A},\textbf{\bf B}) >0$ such that
 $\vert {\eta_{t}^{\varepsilon,\prime}} \vert \leq C  R_{t}^{\varepsilon}$, 
when $R_{t}^{\varepsilon} \leq \delta_{\star}$. 

We then rewrite the dynamics of the angle in the form:
\begin{equation}
\label{eq:theta:bib}
\begin{split}
\ud  \theta_t^{\varepsilon}  &= 
\frac{\DTheta \bigl( \theta_t^{\varepsilon}\bigr)}{g_{t}^{\varepsilon} }
\frac{V_{t}^{\varepsilon}}{\bigl( R_{t}^{\varepsilon} \bigr)^2}
  \ud t
+
\frac{\eta_{t}^{\varepsilon,\prime}}{g_{t}^{\varepsilon}}
\frac{V_{t}^{\varepsilon}}{\bigl( R_{t}^{\varepsilon} \bigr)^2}
  \ud t
   -  \frac{d-1}{2} \bigl( \frac{\varepsilon}{R_{t}^{\varepsilon}}
\bigr)^2 
\theta_t^{\varepsilon}  \ud t
  + \frac{\varepsilon}{R_{t}^{\varepsilon}} \bigl[ \ud B_{t} -
\theta_t^{\varepsilon}
\bigl( \theta_t^{\varepsilon} \cdot \ud B_{t} \bigr)
\bigr],
\end{split}
\end{equation}
as long as $g^{\varepsilon}$ remains positive. 
This prompts us to let
\begin{equation*}
\Sigma_{t}^{\varepsilon} := \int_{0}^t 
\frac{1}{g_{s}^{\varepsilon}}
\frac{V_{s}^{\varepsilon}}{ \bigl( R_{s}^{\varepsilon} \bigr)^2}
\ud s.
\end{equation*}
Our strategy is to compare the process $(\theta_{t}^{\varepsilon})_{t \geq 0}$ with 
 $ ( \phi_{\Sigma_{t}^{\varepsilon}} )_{t \geq 0}$, 
where $(\phi_{t})_{t \geq 0}$ is the solution of the ODE
\eqref{eq:ODE:2:2:sphere}.
{To make it easier, 
we denote by $(\phi_{t}^u)_{t \geq 0}$ the solution
to \eqref{eq:ODE:2:2:sphere} starting from $\phi_{0}^u=u \in {\mathbb S}^{d-1}$.}

\subsubsection{Local comparison between $\theta^{\varepsilon}$ and $\phi$}
Using the fact that $\DTheta$ is Lipschitz continuous, we get
\begin{proposition}
\label{prop:good:events}
Fix a positive real ${\boldsymbol \epsilon}>0$ together with an intensity $\varepsilon>0$. For a stopping time $\tau$ with values in $[0,{\boldsymbol \epsilon}]$, and  {for three reals $a>0$, $\delta \in (0,1)$ and $v_{0}>0$}, let 
\begin{equation*}
\begin{split}
&\Sigma_{\tau,t}^{\varepsilon} := \int_{0}^t 
{{\mathbf 1}_{\{s \geq \tau\}}}
\frac{V_{s}^{\varepsilon}}{g_{s}^{\varepsilon} (R_{s}^{\varepsilon})^2} \ud s
= 
{\int_{0}^t 
{\mathbf 1}_{\{s \geq \tau\}}
(R_{s}^{\varepsilon})^{\alpha-1} \ud s}, \quad t \in [0,{\boldsymbol \epsilon}] \ ; 
\\
&\varrho^{\varepsilon}(\tau) := \inf \bigl\{ t > \tau : g_{t}^{\varepsilon} \leq a \bigr\},
 \quad {\boldsymbol \rho}^{\varepsilon}(\tau) := \inf \bigl\{t > \tau : R_{t}^{\varepsilon} \geq \delta \bigr\}.
\end{split} 
\end{equation*}

i) We can find {a positive threshold 
$v_{\star}:=v_{\star}(\textbf{\bf A},\textbf{\bf B})$ and a constant $c:=c(\textbf{\bf A},\textbf{\bf B},a)$}
such that, for $v_{0} \geq v_{\star}$, the event 
\begin{equation*}
D^{0}(\tau)
:=
\Bigl\{ 
\forall t \in [\tau, 	{\varrho^{\varepsilon}(\tau)}], \quad 
\tfrac12 V_{\tau}^{\varepsilon}
\exp \bigl( c^{-1} \Sigma_{\tau,t}^{\varepsilon}
\bigr) 
\leq 
V_{t}^{\varepsilon}
\leq 
2 V_{\tau}^{\varepsilon}
\exp \bigl( c \Sigma_{\tau,t}^{\varepsilon}
\bigr)
\Bigr\}
\end{equation*}
has conditional probability
\begin{equation*}
{\mathbb P}
\Bigl( 
D^{0}(\tau)
 \, \vert \, 
{\mathcal F}_{\tau} \Bigr) \geq 
1 - {2} \exp \Bigl( - \frac{V_{\tau}^{\varepsilon}}{{2} \varepsilon^2} \Bigr),
\end{equation*}
on the event $\{ V_{\tau}^{\varepsilon} \geq v_{0} \varepsilon^2 \}$. 
\vspace{2pt}

ii) Let 
$\sigma_{0}^{\varepsilon}(\tau):=\tau$ and then, for a given $T>0$ and for all $j \geq 1$, define {the stopping time}:
\begin{equation*}
\sigma_{j}^{\varepsilon}(\tau)
:=  \inf \bigl\{ t > \sigma_{j-1}^{\varepsilon}(\tau) : \Sigma^{\varepsilon}_{\sigma_{j-1}^{\varepsilon},t}
= T \bigr\} \wedge {\boldsymbol \epsilon}.
\end{equation*} 
Then, {there exist a threshold 
$\delta_{\star}:=\delta_{\star}(\textbf{\bf A},\textbf{\bf B})$
and 
a constant $C:=C(\textbf{\bf A},\textbf{\bf B},a,T)$, such that, for 
$\delta \leq \delta_{\star}$
and
$v_{0} \geq v_{\star}$} and for any $j \geq 0$, the complementary of the event\footnote{{In $D^1_{j}(\tau)$, the supremum is null if 
$\sigma_{j}^{\varepsilon}(\tau) \geq \varrho^{\varepsilon}(\tau) \wedge {\boldsymbol \rho}^{\varepsilon}(\tau)$.}} 
\begin{equation*}
\begin{split}
D_{j}^1(\tau) &:= 
\biggl\{ \sup_{\sigma_{j}^{\varepsilon}(\tau) \leq t \leq \sigma_{j+1}^{\varepsilon}(\tau) \wedge \varrho^{\varepsilon}(\tau) \wedge {\boldsymbol \rho}^{{\boldsymbol \epsilon}}(\tau)} \biggl\vert \theta_{t}^{\varepsilon} - \phi_{\Sigma_{\sigma_{j}^{\varepsilon}(\tau),t}^{\varepsilon}}^{\theta_{\sigma_{j}^{\varepsilon}(\tau)}^{\varepsilon}} \biggr\vert \leq C \Bigl( \delta + \frac{\varepsilon^2}{V_{\tau}^{\varepsilon}} \Bigr)
\biggr\},
\end{split}
\end{equation*}
when intersected with $D^0(\tau)$,  has conditional probability  
\begin{equation*}
{\mathbb P} \Bigl(  D^0(\tau) \cap D_{j}^1(\tau)^{\complement} \, \vert  \,
{\mathcal F}_{ \tau}
\Bigr) \leq \frac{C}{\delta^4} \Bigl(  \frac{\varepsilon^2}{(j+1) V_{\tau}^{\varepsilon}} \Bigr)^2. 
\end{equation*}
\vskip 2pt

iii) Similarly, {for $\delta \leq \delta_{\star}$ and $v_{0} \geq v_{\star}$, for $j \geq 0$}, the complementary of the event 
\begin{equation*}
\begin{split}
D_{j}^2(\tau) 
& :=   \biggl\{
\forall t \in [\sigma_{j}^{\varepsilon}(\tau),\sigma_{j+1}^{\varepsilon}(\tau)],
\quad
 \Sigma_{\sigma_{j}^{\varepsilon}(\tau),t \wedge \varrho^{\varepsilon}(\tau)\wedge 
{\boldsymbol \rho}^{\varepsilon}(\tau)}^{\varepsilon} 
\\
&\hspace{-30pt} \geq {C^{-1} }
\Bigl( 
1 + \frac{\varepsilon^2}{V_{\tau}^{\varepsilon}}
\Bigr)^{-1}
\biggl(
\ln \Bigl[ (2 V_{\sigma_{j}^{\varepsilon}(\tau)}^{\varepsilon})^{\tfrac{1-\alpha}{1+\alpha}}
+ 
{C^{-1}} \bigl( t \wedge \varrho^{\boldsymbol \epsilon}(\tau)\wedge {\boldsymbol \rho}^{\varepsilon}(\tau) - \sigma_{j}^{\varepsilon}(\tau)\bigr)_{{+}}\Bigr] - \ln  \Bigl[ (2 V_{\sigma_{j}^{\varepsilon}(\tau)}^{\varepsilon})^{\tfrac{1-\alpha}{1+\alpha}} \Bigr]
\biggr)
\biggr\},
\end{split}
\end{equation*}
when intersected with $D^0(\tau)$, 
has conditional probability 
\begin{equation*}
{\mathbb P} \Bigl( D^0(\tau) \cap  D_{j}^2(\tau)^{\complement} \, \vert
\,  
{\mathcal F}_{\tau}
\Bigr) \leq C \Bigl( \frac{\varepsilon^2}{(j+1) V_{\tau}^{\varepsilon}} \Bigr)^2, 
\end{equation*}
on the event $\{ V_{\tau}^{\varepsilon} \geq v_{0} \varepsilon^2 \}$,
{where $C:=C(\textbf{\bf A},\textbf{\bf B},a,T)$}. 
\end{proposition}

%
%

\begin{proof}[Proof of  Proposition 
\ref{prop:good:events}] { \ }

\textit{First step.} We prove $i)$.  
We first let
\begin{equation*}
\xi^{\varepsilon}(\tau) := \inf \Bigl\{ t \geq \tau : V_{t}^{\varepsilon} \leq \frac{V_{\tau}^{\varepsilon}}{2}  \Bigr\}.
\end{equation*}
For a given constant $c>0$ and for $v_{0} \geq v_{\star}(\textbf{\bf A},\textbf{\bf B})$ (see for instance 
the proof of Lemma 
\ref{lem:5:1}), for any $t \in [\tau,\varrho^{\varepsilon}(\tau)  \wedge \xi^{\varepsilon}(\tau)]$, 
\begin{equation*}
\begin{split}
&\ud \bigl( V_{t}^{\varepsilon} \exp ( - c \Sigma_{\tau,t}^{\varepsilon})\bigr)
\geq \Bigl( \tfrac78 \vert \nabla V_{t}^{\varepsilon} \vert^2 - c 
\frac{(V_{t}^{\varepsilon})^2}{g_{t}^{\varepsilon} (R_{t}^{\varepsilon})^2} \Bigr) 
\exp ( - c \Sigma_{\tau,t}^{\varepsilon} )
\ud t
  + \varepsilon \exp ( - c \Sigma_{\tau,t}^{\varepsilon}) \nabla V_{t}^{\varepsilon}
\cdot \ud B_{t},
\end{split}
\end{equation*}
on the event $\{ V_{\tau}^{\varepsilon} \geq v_{0} \varepsilon^2 \}$. 
By assumption (\textbf{A3}), 
we can choose {$c := c({\mathbf A},{\mathbf B},a)$} such that, for $t \in [\tau,\varrho^{\varepsilon}(\tau) \wedge \xi^{\varepsilon}(\tau)]$, 
\begin{equation*}
\Bigl( \tfrac78 \vert \nabla V_{t}^{\varepsilon} \vert^2 - c 
\frac{(V_{t}^{\varepsilon})^2}{g_{t}^{\varepsilon} (R_{t}^{\varepsilon})^2} \Bigr) 
\exp ( - c \Sigma_{\tau,t}^{\varepsilon} )
\geq \tfrac12 
\vert \nabla V_{t}^{\varepsilon} \vert^2 \exp\bigl(- c \Sigma_{\tau,t}^{\varepsilon} \bigr). 
\end{equation*}
It then remains to see that 
\begin{equation*}
\begin{split}
&{\mathbb P}
\Bigl( \exists t \geq \tau : \frac12 V_{\tau}^{\varepsilon} +
\frac12 
\int_{\tau}^{t} 
\exp(-c \Sigma_{\tau,s}^{\varepsilon})
\vert \nabla V_{s}^{\varepsilon} \vert^2  \ud s
+
\varepsilon \int_{\tau}^t 
\exp(-c \Sigma_{\tau,s}^{\varepsilon})
\nabla V_{s}^{\varepsilon} \cdot \ud B_{s}
\leq 0 \, \vert \, {\mathcal F}_{\tau}
\Bigr) 
 \leq \exp \bigl( - 
\frac{ V_{\tau}^{\varepsilon}}{2 \varepsilon^2}
\bigr),
\end{split}
\end{equation*}
from which we get
\begin{equation}
\label{eq:proof:6:6:lower:bound:D0}
{\mathbb P}
\Bigl( 
\Bigl\{ 
\forall t \in [\tau,\varrho^{\varepsilon}(\tau)  \wedge \xi^{\varepsilon}(\tau)], \quad 
V_{t}^{\varepsilon}
\geq 
\tfrac12 V_{\tau}^{\varepsilon}
\exp \bigl( c \Sigma_{\tau,t}^{\varepsilon}
\bigr)
\Bigr\}
 \, \vert \, 
{\mathcal F}_{\tau} \Bigr) \geq 
1 - \exp \Bigl( - \frac{V_{\tau}^{\varepsilon}}{2 \varepsilon^2} \Bigr),
\end{equation}
on the event $\{ V_{\tau}^{\varepsilon} \geq v_{0} \varepsilon^2 \}$. 
{Clearly, we can remove the stopping time $\xi^{\varepsilon}$ appearing in the event in the left-hand side.} 

As for the upper bound {appearing in the definition of $D^0(\tau)$}, we must come back to the proof of 
Lemma 
\ref{lem:supermartingale}, {see in particular \eqref{eq:V:R2}}. {Working on} the event 
$\{ V_{\tau}^{\varepsilon} \geq v_{0} \varepsilon^2\}$
and using (\textbf{A2}), we deduce that 
\begin{equation}
\label{eq:additional:label:6.6:proof}
\varepsilon^2 \vert \Delta V_{t}^{\varepsilon} \vert  \leq \frac{2 {C_{0}}}{v_{0}} V_{t}^{\varepsilon} \bigl( R_{t}^{\varepsilon} 
\bigr)^{\alpha-1}, \quad t \in [{\tau},\xi^{\varepsilon}(\tau)]. 
\end{equation}
Up until $\varrho^{\varepsilon}(\tau) \wedge \xi^{\varepsilon}(\tau)$, we also have from 
(\textbf{A2})
and
(\textbf{A3}): 
\begin{equation}
\label{eq:label:en:plus:encadrement:nablaV}
{C^{-1} 
V_{t}^{\varepsilon} \bigl( R_{t}^{\varepsilon} \bigr)^{\alpha-1}
\leq 
C^{-1} \frac{\bigl( V_{t}^{\varepsilon} \bigr)^2}{\bigl( R_{t}^{\varepsilon} \bigr)^2}
\leq \vert \nabla V_{t}^{\varepsilon} \vert^2 \leq C \bigl( R_{t}^{\varepsilon} \bigr)^{2 \alpha} \leq 
C V_{t}^{\varepsilon} \bigl( R_{t}^{\varepsilon} \bigr)^{\alpha-1},}
\end{equation}
{for a constant $C=C(\textbf{\bf A},\textbf{\bf B},a)$}. For a new value of the constant 
{$c \geq c_{\star}(\textbf{\bf A},\textbf{\bf B},a,v_{\star})$}
({pay attention that the constant $c$ that appears in 
the upper bound in the definition of 
 $D^0(\tau)$ may be chosen independently of the constant that appears in the lower bound; in particular, 
 it may be chosen as large as needed)}, we deduce that, 
{on the event $\{ V_{\tau}^{\varepsilon} \geq v_{0} \varepsilon^2\}$},  
 for any $t \in [\tau,\varrho^{\varepsilon}(\tau)  \wedge \xi^{\varepsilon}(\tau)]$, 
\begin{equation*}
\begin{split}
&\ud \bigl( V_{t}^{\varepsilon} \exp ( - c \Sigma_{\tau,t}^{\varepsilon})\bigr)
\leq - \tfrac{1}{2} 
\vert \nabla V_{t}^{\varepsilon} \vert^2 
\exp ( - c \Sigma_{\tau,t}^{\varepsilon} )
\ud t
  + \varepsilon \exp ( - c \Sigma_{\tau,t}^{\varepsilon}) \nabla V_{t}^{\varepsilon}
\cdot \ud B_{t}.
\end{split}
\end{equation*}
Proceeding as for the lower bound, we see that 
\begin{equation*}
\begin{split}
&{\mathbb P}
\Bigl( \exists t \geq \tau : - V_{\tau}^{\varepsilon}- 
\frac12
\int_{\tau}^{t} 
\exp(-c \Sigma_{\tau,s}^{\varepsilon})
\vert \nabla V_{s}^{\varepsilon} \vert^2  \ud s
+
\varepsilon \int_{\tau}^t 
\exp(-c \Sigma_{\tau,s}^{\varepsilon})
\nabla V_{s}^{\varepsilon} \cdot \ud B_{s}
\geq 0 \, \vert \, {\mathcal F}_{\tau}
\Bigr) 
 \leq \exp \bigl( - 
\frac{ V_{\tau}^{\varepsilon}}{ \varepsilon^2}
\bigr),
\end{split}
\end{equation*}
and, as above, we get {the following upper bound}
\begin{equation*}
{\mathbb P}
\Bigl( 
\Bigl\{ 
\forall t \in [\tau,\varrho^{\varepsilon}(\tau)    \wedge
{\xi^{\varepsilon}(\tau)}], \quad 
V_{t}^{\varepsilon}
\leq 
2 V_{\tau}^{\varepsilon}
\exp \bigl( c \Sigma_{\tau,t}^{\varepsilon}
\bigr)
\Bigr\}
 \, \vert \, 
{\mathcal F}_{\tau}
 \Bigr)  \leq {\exp \bigl( - 
\frac{ V_{\tau}^{\varepsilon}}{ \varepsilon^2}
\bigr)}, 
\end{equation*}
{on the event $\{ V_{\tau}^{\varepsilon} \geq v_{0} \varepsilon^2\}$}. 
{By combining with \eqref{eq:proof:6:6:lower:bound:D0}},
we deduce  that, on the event $\{ V_{\tau}^{\varepsilon} \geq v_{0} \varepsilon^2 \}$,
\begin{equation*}
{\mathbb P}
\Bigl( D^0(\tau)
 \, \vert \, {\mathcal F}_{\tau}
 \Bigr) \geq 
1 - {2}\exp \Bigl( - \frac{V_{\tau}^{\varepsilon}}{{2} \varepsilon^2} \Bigr), 
\end{equation*}
which is the required inequality. 

Moreover, observe (for the sequel) that, on the event $D^0(\tau) \cap \{ \sigma_{j}^{\varepsilon}(\tau) < {\boldsymbol \epsilon}\}$, 
\begin{equation}
\label{eq:prop:6:6:proof:conclusion:first:step}
\tfrac12 V_{\tau}^{\varepsilon} \exp \bigl( c^{-1} j T \bigr) 
 \leq V_{t}^{\varepsilon} \leq
2 V_{\tau}^{\varepsilon} \exp \bigl( c (j+1) T \bigr), 
\quad t \in \bigl[ \sigma_{j}^{\varepsilon}(\tau),\sigma_{j+1}^{\varepsilon}(\tau) \wedge \varrho^{\varepsilon}(\tau)\bigr).
\end{equation}
{Of course, it must be stressed that the time interval on which the inequality holds true is empty
{if 
$\sigma_{j}^{\varepsilon}(\tau) > \varrho^{\varepsilon}(\tau)$}.}
\vspace{4pt}

\textit{Second step.} 
We now prove $ii)$. 
To simplify, we just write
$\Sigma^{j,\varepsilon}_{t}$ for 
$\Sigma^{\varepsilon}_{\sigma_{j}^{\varepsilon}(\tau),t}$
and $\phi_{t}^{j,\varepsilon}$ for 
 $\phi_{t}^{\theta^{\varepsilon}_{\sigma^{\varepsilon}_{j}(\tau)}}$. 
Then, letting 
\begin{equation*}
\hat{\phi}^{j,\varepsilon}_{t} := \phi^{j,\varepsilon}_{\Sigma_{t}^{j,\varepsilon}}, \quad t \in 
\bigl[ \sigma_{j}^{\varepsilon}(\tau),{\boldsymbol \epsilon} \bigr],
\end{equation*}
we have
\begin{equation*}
\hat{\phi}^{j,\varepsilon}_{t} = \theta^{\varepsilon}_{\sigma_{j}^{\varepsilon}(\tau)}
+ \int_{\sigma_{j}^{\varepsilon}(\tau)}^t
\DTheta\bigl(\hat{\phi}^{j,\varepsilon}_{s}\bigr)
\frac{V_{s}^{\varepsilon}}{g_{s}^{\varepsilon}(R_{s}^{\varepsilon})^2}
\ud s, 
\end{equation*}
for $t \in [\sigma_{j}^{\varepsilon}(\tau),{\boldsymbol \epsilon}]$. Then,
by \eqref{eq:theta:bib} and by the Lipschitz property of $\DTheta$,
for $t \in [\sigma_{j}^{\varepsilon}(\tau),
\varrho^{\varepsilon}(\tau) \wedge 
{\boldsymbol \rho}^{\varepsilon}(\tau)]$, 
\begin{equation*}
\begin{split}
\bigl\vert \theta_{t}^{\varepsilon} -
\hat{\phi}^{j,\varepsilon}_{t} \bigr\vert &\leq
C \int_{\sigma_{j}^{\varepsilon}(\tau)}^t
\bigl\vert \theta_{s}^{\varepsilon} -
\hat{\phi}^{j,\varepsilon}_{s} \bigr\vert
\frac{V_{s}^{\varepsilon}}{g_{s}^{\varepsilon}(R_{s}^{\varepsilon})^2}
ds
 + C \int_{\sigma_{j}^{\varepsilon}(\tau)}^t 
 {\eta_{s}^{\varepsilon,\prime}} 
\frac{V_{s}^{\varepsilon}}{g_{s}^{\varepsilon}(R_{s}^{\varepsilon})^2}
\ud s + C \int_{\sigma_{j}^{\varepsilon}(\tau)}^t
\frac{\varepsilon^2}{V_{s}^{\varepsilon}} \frac{V_{s}^{\varepsilon}}{g_{s}^{\varepsilon} (R_{s}^{\varepsilon})^2} \ud s
\\
&\hspace{15pt} + C \biggl\vert 
\int_{\sigma_{j}^{\varepsilon}(\tau)}^t 
\frac{\varepsilon}{\sqrt{V_{s}^{\varepsilon}}} 
\sqrt{\frac{V_{s}^{\varepsilon}}{(R_{s}^{\varepsilon})^2}} 
\Bigl[ \ud B_{s} -
\theta_s^{\varepsilon}
\Bigl( \theta_s^{\varepsilon} \cdot \ud B_{s} \Bigr)
\Bigr]
\biggr\vert,
\end{split}
\end{equation*}
where {$C:=C(\textbf{\bf A},\textbf{\bf B})$ is allowed to increase from line to line}. 
So,
{for $\delta \leq \delta_{\star}$ as in 
\eqref{eq:theta:bib}},
 on the event $D^{0}(\tau)$, 
for $t \in [\sigma_{j}^{\varepsilon}(\tau),
\varrho^{\varepsilon}(\tau) \wedge 
{\boldsymbol \rho}^{\varepsilon}(\tau)]$, 
\begin{equation*}
\begin{split}
\bigl\vert \theta_{t}^{\varepsilon} -
\hat{\phi}^{j,\varepsilon}_{t} \bigr\vert &\leq
C \int_{\sigma_{j}^{\varepsilon}(\tau)}^t
\bigl\vert \theta_{s}^{\varepsilon} -
\hat{\phi}^{j,\varepsilon}_{s} \bigr\vert
\frac{V_{s}^{\varepsilon}}{g_{s}^{\varepsilon}(R_{s}^{\varepsilon})^2}
\ud s
  + C  
 {  \bigl( \delta + \frac{\varepsilon^2}{V_{\tau}^{\varepsilon}}
  \bigr)}
 \int_{\sigma_{j}^{\varepsilon}(\tau)}^t 
\frac{V_{s}^{\varepsilon}}{g_{s}^{\varepsilon}(R_{s}^{\varepsilon})^2}
\ud s
\\
&\hspace{15pt} + C \biggl\vert 
\int_{\sigma_{j}^{\varepsilon}(\tau)}^t 
\frac{\varepsilon}{\sqrt{V_{s}^{\varepsilon}}} 
\sqrt{\frac{V_{s}^{\varepsilon}}{(R_{s}^{\varepsilon})^2}} 
\Bigl[ \ud B_{s} -
\theta_s^{\varepsilon}
\Bigl( \theta_s^{\varepsilon} \cdot \ud B_{s} \Bigr)
\Bigr]
\biggr\vert. 
\end{split}
\end{equation*}
Allowing the constant $C$ to depend on $T$ ({that is $C:=C(\textbf{\bf A},\textbf{\bf B},T)$}), 
using Gronwall's lemma and recalling that 
$\Sigma^{j,\varepsilon}_{t} \leq T$ for $t \in [\sigma_{j}^{\varepsilon}(\tau),\sigma_{j+1}^{\varepsilon}(\tau)]$, we get
\begin{equation*}
\begin{split}
&\sup_{\sigma_{j}^{\varepsilon}(\tau) \leq t \leq \sigma_{j+1}^{\varepsilon}(\tau) \wedge 
\varrho^{\varepsilon}(\tau) \wedge {\boldsymbol \rho}^{\varepsilon}(\tau)}
\bigl\vert \theta_{t}^{\varepsilon} -
\hat{\phi}^{j,\varepsilon}_{t} \bigr\vert \leq
C 
\bigl( 
\delta + \frac{\varepsilon^2}{V_{\tau}^{\varepsilon}}
\bigr) 
\\
&\hspace{60pt} + C
\sup_{\sigma_{j}^{\varepsilon}(\tau) \leq t \leq \sigma_{j+1}^{\varepsilon}(\tau) \wedge 
\varrho^{\varepsilon}(\tau) \wedge {\boldsymbol \rho}^{\varepsilon}(\tau)}
 \biggl\vert 
\int_{\sigma_{j}^{\varepsilon}(\tau)}^t 
\frac{\varepsilon}{\sqrt{V_{s}^{\varepsilon}}} 
\sqrt{\frac{V_{s}^{\varepsilon}}{(R_{s}^{\varepsilon})^2}} 
\Bigl[ \ud B_{s} -
\theta_s^{\varepsilon}
\Bigl( \theta_s^{\varepsilon} \cdot \ud B_{s} \Bigr)
\Bigr]
\biggr\vert. 
\end{split}
\end{equation*}
Now, by Markov inequality and by B\"urkholder-Davis-Gundy  inequality, we get, on the event 
$\{ \sigma_{j}^{\varepsilon}(\tau) < {\boldsymbol \epsilon}\}$, 
\begin{equation*}
\begin{split}
&{\mathbb P}
\biggl( 
\biggl\{
\sup_{\sigma_{j}^{\varepsilon}(\tau) \leq t \leq \sigma_{j+1}^{\varepsilon}(\tau) \wedge 
\varrho^{\varepsilon}(\tau) \wedge {\boldsymbol \rho}^{\varepsilon}(\tau)}
 \biggl\vert 
\int_{\sigma_{j}^{\varepsilon}(\tau)}^t 
\frac{\varepsilon}{\sqrt{V_{s}^{\varepsilon}}} 
\sqrt{\frac{V_{s}^{\varepsilon}}{(R_{s}^{\varepsilon})^2}} 
\Bigl[ \ud B_{s} -
\theta_s^{\varepsilon}
\Bigl( \theta_s^{\varepsilon} \cdot \ud B_{s} \Bigr)
\Bigr]
\biggr\vert \geq \delta \biggr\} \cap D^0(\tau) \, \vert \, {\mathcal F}_{\tau} \biggr)
\\
&\leq \frac{C}{\delta^4}{\mathbb E} \biggl[ \biggl( \int_{\sigma_{j}^{\varepsilon}(\tau)}^{\sigma_{j+1}^{\varepsilon}(\tau) \wedge 
\varrho^{\varepsilon}(\tau) \wedge 
{\boldsymbol \rho}^{\varepsilon}(\tau)}
{\mathbf 1}_{\{V_{s}^{\varepsilon} \geq V_{\tau}^{\varepsilon} \exp(C^{-1}j)\}}  
\frac{\varepsilon^2}{V_{s}^{\varepsilon}}
{\frac{V_{s}^{\varepsilon}}{(R_{s}^{\varepsilon})^2}} 
\ud s \biggr)^2 \, \vert \, {\mathcal F}_{\tau} \biggr]
\\
&\leq \frac{C}{\delta^4}\Bigl( \frac{\varepsilon^2}{V_{\tau}^{\varepsilon} \exp(C^{-1} j)}
\Bigr)^2
{\mathbb E} \biggl[ \biggl( \int_{\sigma_{j}^{\varepsilon}(\tau)}^{\sigma_{j+1}^{\varepsilon}(\tau) \wedge 
\varrho^{\varepsilon}(\tau) \wedge {\boldsymbol \rho}^{\varepsilon}(\tau)}
{\frac{V_{s}^{\varepsilon}}{(R_{s}^{\varepsilon})^2}} \ud s \biggr)^2 \, \vert \, {\mathcal F}_{\tau} \biggr]
\\
&\leq \frac{C}{\delta^4}\Bigl( \frac{\varepsilon^2}{V_{\tau}^{\varepsilon} \exp(C^{-1} j)}
\Bigr)^2,
\end{split}
\end{equation*}
{where we used \eqref{eq:prop:6:6:proof:conclusion:first:step} in the second line and $C:=C(\textbf{A},\textbf{\bf B},a,T)$.}
Obviously, the bound remains true when $\sigma_{j}^{\varepsilon}(\tau)={\boldsymbol \epsilon}$ because, in that case, the sup in the first line reduces to $0$. 
The upper bound for ${\mathbb P}(D^1_{j}(\tau)^{\complement} \cap D^0(\tau) \, \vert \, {\mathcal F}_{\tau})$
easily follows. 
\vskip 4pt

\textit{Third step.} We now proceed in the same way to upper bound 
${\mathbb P}( D^2_{j}(\tau)^{\complement} \cap D^0(\tau) \, \vert \, {\mathcal F}_{\tau})$. 
\vskip 4pt

We start with the following observation. 
Following 
\eqref{eq:additional:label:6.6:proof}
and
\eqref{eq:label:en:plus:encadrement:nablaV}, we can find a new constant $C:=C(\textbf{\bf A},\textbf{\bf B},a)$ such that, for $t \in [\tau,\varrho^{\varepsilon}(\tau)\wedge {\boldsymbol \rho}^{\varepsilon}(\tau)\wedge \xi^{\varepsilon}(\tau)]$, 
\begin{equation*}
\begin{split}
\ud V_{t}^{\varepsilon}  &\leq   \vert \nabla V_{t}^{\varepsilon} \vert^2  \ud t +
C \frac{\varepsilon^2}{V_{\tau}^{\varepsilon}}
 V_{t}^{\varepsilon}
\bigl( R_{t}^{\varepsilon} \bigr)^{\alpha-1} \ud t + \varepsilon \nabla V_{t}^{\varepsilon} \cdot \ud B_{t}
 \leq C \Bigl( 1+  
 \frac{\varepsilon^2}{ V_{\tau}^{\varepsilon}}
\Bigr)
\bigl( V_{t}^{\varepsilon} \bigr)^{\tfrac{2\alpha}{1+\alpha}} \ud t -
\tfrac12 \vert \nabla V_{t}^{\varepsilon} \vert^2  \ud t + \varepsilon \nabla V_{t}^{\varepsilon} \cdot \ud B_{t}.
\end{split}
\end{equation*}
We then observe that the event 
\begin{equation*}
D^3_{j}(\tau) :=
\biggl\{
\exists t \geq \sigma_{j}^{\varepsilon}(\tau) : \varepsilon \int_{\sigma_{j}^{\varepsilon}(\tau)}^t \nabla V_{s}^{\varepsilon} \cdot \ud B_{s} 
- \tfrac12 \int_{\sigma_{j}^{\varepsilon}(\tau)}^t \bigl\vert \nabla V_{s}^{\varepsilon} \bigr\vert^2 \ud s
\geq V_{\sigma_{j}^{\varepsilon}(\tau)}^{\varepsilon}
\biggr\},
\end{equation*}
has conditional probability
\begin{equation*}
{\mathbb P}
\bigl( D^3_{j}(\tau)  \, \vert  \, {\mathcal F}_{\sigma_{j}^{\varepsilon}(\tau)}
\bigr) \leq \exp\biggl(- \frac{V_{\sigma_{j}^{\varepsilon}(\tau)}^{\varepsilon}}{\varepsilon^2} \biggr).
\end{equation*}
On $D^0(\tau) \cap \{ \sigma_{j}^{\varepsilon}(\tau) < {\boldsymbol \epsilon} {
\wedge \varrho^{\varepsilon}(\tau)\wedge {\boldsymbol \rho}^{\varepsilon}(\tau)}\}$,
$V^{\varepsilon}_{\sigma^{\varepsilon}_{j}(\tau)} \geq C^{-1} {(j+1)} V_{\tau}^{\varepsilon}$
{(see \eqref{eq:prop:6:6:proof:conclusion:first:step})}, 
 {with $C:=C(\textbf{A},\textbf{\bf B},a,T)$. Then,}
\begin{equation}
\label{eq:P:D3j}
\begin{split}
&{\mathbb P}
\Bigl( D^3_{j}(\tau) \cap {D^0(\tau)} \cap \{ \sigma_{j}^{\varepsilon}(\tau) < {\boldsymbol \epsilon} 
\wedge \varrho^{\varepsilon}(\tau)\wedge {\boldsymbol \rho}^{\varepsilon}(\tau)\}  \, \vert  \, {\mathcal F}_{\tau}
\Bigr)
\\
&\leq {\mathbb P}
\Bigl( D^3_{j}(\tau) \cap \bigl\{ V^{\varepsilon}_{\sigma^{\varepsilon}_{j}(\tau)} \geq C^{-1} {(j+1)} V_{\tau}^{\varepsilon} \bigr\}  \, \vert  \, {\mathcal F}_{\tau}
\Bigr)
\\
&= {\mathbb E}
\Bigl[
{\mathbb P}
\bigl( D^3_{j}(\tau)  \, \vert  \, {\mathcal F}_{\sigma_{j}^{\varepsilon}(\tau)}
\bigr)
{\mathbf 1}_{\{
 V^{\varepsilon}_{\sigma^{\varepsilon}_{j}(\tau)} \geq C^{-1} {(j+1)} V_{\tau}^{\varepsilon}
\}}
\, \vert  \, {\mathcal F}_{\tau}
\Bigr]
\leq \exp\biggl(- C^{-1}\frac{{(j+1)} V_{\tau}^{\varepsilon}}{\varepsilon^2} \biggr)
\leq C \Bigl( \frac{\varepsilon^2}{(j+1) V_{\tau}^{\varepsilon}} \Bigr)^2.
\end{split}
\end{equation}
On $D^3_{j}(\tau)^{\complement} \cap D^0(\tau)$, we have
\begin{equation*}
\forall t \in [\sigma_{j}^{\varepsilon}(\tau),\sigma_{j+1}^{\varepsilon}(\tau) \wedge \varrho^{\varepsilon}(\tau)\wedge {\boldsymbol \rho}^{\varepsilon}(\tau) \wedge \xi^{\varepsilon}(\tau)], \quad V_{t}^{\varepsilon} \leq 2 V_{\sigma_{j}^{\varepsilon}(\tau)}^{\varepsilon} + C 
\Bigl( 1 + \frac{\varepsilon^2}{V_{\tau}^{\varepsilon}}
\Bigr)
\int_{\sigma_{j}^{\varepsilon}(\tau)}^t 
\bigl( V_{s}^{\varepsilon} \bigr)^{\tfrac{2\alpha}{1+\alpha}} \ud s, 
\end{equation*}
and then, by a standard comparison argument with the solution of the 
ODE $\dot y_{t} = C(1+\varepsilon^2/V_{\tau}^{\varepsilon}) y_{t}^{{2\alpha}/{1+\alpha}}$,
$t \geq \tau$, $y_{\tau} > 2 V_{\sigma_{j}^{\varepsilon}(\tau)}^{\varepsilon}$, 
we have
\begin{equation*}
V_{t}^{\varepsilon} \leq \Bigl[ \bigl(2 V_{\sigma_{j}^{\varepsilon}(\tau)}^{\varepsilon} \bigr)^{\tfrac{1-\alpha}{1+\alpha}}
+ C \frac{1-\alpha}{1+\alpha}
\Bigl( 1 + \frac{\varepsilon^2}{V_{\tau}^{\varepsilon}}
\Bigr) \bigl( t - \sigma_{j}^{\varepsilon}(\tau)\bigr) \Bigr]^{\tfrac{1+\alpha}{1-\alpha}}. 
\end{equation*}
And then, for $t \in [\sigma_{j}^{\varepsilon}(\tau),\sigma_{j+1}^{\varepsilon}(\tau) \wedge \varrho^{\varepsilon}(\tau)\wedge {\boldsymbol \rho}^{\varepsilon}(\tau) \wedge \xi^{\varepsilon}(\tau)]$, 
\begin{equation*}
\begin{split}
\int_{\sigma_{j}^{\varepsilon}(\tau)}^t \frac{V_{s}^{\varepsilon}}{(R_{s}^{\varepsilon})^2} ds
&\geq 
 C^{-1} \int_{\sigma_{j}^{\varepsilon}(\tau)}^t \bigl( V_{s}^{\varepsilon} \bigr)^{- \tfrac{1-\alpha}{1+\alpha}}
 ds
 \\
 &\geq C^{-1} \int_{\sigma_{j}^{\varepsilon}(\tau)}^t 
 \Bigl[ \Bigl(2 V_{\sigma_{j}^{\varepsilon}(\tau)}^{\varepsilon} \Bigr)^{\tfrac{1-\alpha}{1+\alpha}}
+ C \frac{1-\alpha}{1+\alpha}
\Bigl( 1 + \frac{\varepsilon^2}{V_{\tau}^{\varepsilon}}
\Bigr) \bigl( s - \sigma_{j}^{\varepsilon}(\tau)\bigr) \Bigr]^{-1} 
 ds
\\
&\geq C^{-1} 
\Bigl( 
1 + \frac{\varepsilon^2}{V_{\tau}^{\varepsilon}}
\Bigr)^{-1}
\biggl( \ln \Bigl[ \Bigl(2 V_{\sigma_{j}^{\varepsilon}(\tau)}^{\varepsilon} \Bigr)^{\tfrac{1-\alpha}{1+\alpha}}
+ C^{-1}
\bigl( t - \sigma_{j}^{\varepsilon}(\tau) \bigr) \Bigr] -
\ln \Bigl[ \Bigl(2 V_{\sigma_{j}^{\varepsilon}(\tau)}^{\varepsilon} \Bigr)^{\tfrac{1-\alpha}{1+\alpha}}
 \Bigr]
\biggr),
\end{split}
\end{equation*}
{with $C:=C(\textbf{A},\textbf{\bf B},a,T)$}.
In fact, it is easily checked that $\varrho^{\varepsilon}(\tau)\wedge {\boldsymbol \rho}^{\varepsilon}(\tau) \wedge \xi^{\varepsilon}(\tau) = 
\varrho^{\varepsilon}(\tau)\wedge {\boldsymbol \rho}^{\varepsilon}(\tau)
$ on $D^0(\tau)$. Hence, on $D^3_{j}(\tau)^{\complement} \cap D^0(\tau)$, the above is true 
for any $t \in [\sigma_{j}^{\varepsilon}(\tau),\sigma_{j+1}^{\varepsilon}(\tau) \wedge \varrho^{\varepsilon}(\tau)\wedge {\boldsymbol \rho}^{\varepsilon}(\tau) ]$. This says that 
$D^3_{j}(\tau)^{\complement} \cap D^0(\tau)$
is included in $D^2_{j}(\tau)$. 
In particular, 
$D^2_{j}(\tau)^{\complement} \cap D^0(\tau) \subset 
D^3_{j}(\tau) \cap D^0(\tau)$.
 In order to complete the proof,
it suffices to see
that {$D^2_{j}(\tau)^{\complement} \subset \{ \sigma_{j}^{\varepsilon}(\tau)
< {\boldsymbol \epsilon}
\wedge \varrho^{\varepsilon}(\tau)\wedge {\boldsymbol \rho}^{\varepsilon}(\tau)\}$} and then to use \eqref{eq:P:D3j}. 
\end{proof}

\subsubsection{Global comparison between $\theta^{\varepsilon}$ and $\phi$}
We have the following corollary of Proposition 
\ref{prop:good:events}:

\begin{corollary}
\label{coro:6:7}
With the same assumptions and notations as in the statement of Proposition \ref{prop:good:events}, there exists a constant 
{$C:=C(\textbf{\bf A},\textbf{\bf B},a,T)$}, such that, for 
any stopping time $\tau$ with values in $[0,{\boldsymbol \epsilon}]$
with ${\mathbb P}(V_{\tau}^{\varepsilon} \geq v_{\star} \varepsilon^2)>0$, the following holds true on the event $\{ V_{\tau}^{\varepsilon} \geq v_{\star} \varepsilon^2 \}$:
\begin{equation*}
{\mathbb P}
\biggl( D^0(\tau) \cap \Bigl(
\bigcup_{j \geq 0} D^1_{j}(\tau)^{\complement} \cup \bigcup_{j \geq 0} 
D^2_{j}(\tau)^{\complement} \Bigr) \, \vert \, {\mathcal F}_{\tau}
\biggr) \leq \frac{C \varepsilon^4}{\delta^4 (V_{\tau}^{\varepsilon})^2}. 
\end{equation*}
In particular, if we let 
\begin{equation*}
{\mathcal D}(\tau) := \biggl( \bigcap_{j \geq 0} D^1_{j}(\tau) 
\biggr) {\cap} \biggl( 
\bigcap_{j \geq 0} D^2_{j}(\tau) 
\biggr),
\end{equation*}
then, on the event $\{ V_{\tau}^{\varepsilon} \geq v_{\star} \varepsilon^2 \}$,
\begin{equation*}
{\mathbb P}\Bigl( {\mathcal D}(\tau) \, \vert \, {\mathcal F}_{\tau}
\Bigr) \geq 1 - \frac{C \varepsilon^4}{\delta^4 (V_{\tau}^{\varepsilon})^2}.
\end{equation*}
\end{corollary}

\begin{proof}
The first claim is obtained by summing over $j$ in the upper bounds for 
${\mathbb P}(D^1_{j}(\tau)^{\complement} \cap 
D^0(\tau) \, \vert \, {\mathcal F}_{\tau})$
and 
${\mathbb P}(D^2_{j}(\tau)^{\complement} \cap D^0(\tau) \, \vert \, {\mathcal F}_{\tau})$
{in the statement of Proposition \ref{prop:good:events}}.
By taking the complementary, 
we get 
\begin{equation*}
{\mathbb P}
\Bigl( D^0(\tau)^{\complement} \cup 
{\mathcal D}(\tau)  \, \vert \, {\mathcal F}_{\tau} \Bigr) \geq 1 - \frac{C}{\delta^4 \varepsilon^4} (V_{\tau}^{\varepsilon})^2. 
\end{equation*}
We then use the lower bound for 
${\mathbb P}(D^0(\tau)  \, \vert \, {\mathcal F}_{\tau})$ in order to complete the proof. 
\end{proof}

\subsection{Escaping from the well formed by a local minimum}
\label{subse:escape:well}

{In the previous subsection, we addressed the distance between the angle $\theta^\varepsilon$ and the 
solution $\phi$ to the equation \eqref{eq:ODE:2:2:sphere}. As we show below, this preliminary analysis
turns out to be really useful when the particle starts from the same set ${\mathcal B}_{a}$
as in (\textbf{B3}) and (\textbf{\bf B4}). Recall indeed that 
${\mathcal B}_{a}$ is attracted to ${\mathcal S}$ by the flow $\phi$. In this subsection, we provide 
some preliminary computations when the flow starts away from ${\mathcal B}_{a}$, namely from 
a neighborhood of the local minima ${\mathcal L}$. 
Notice in this regard} from assumption (\textbf{\bf B4}) that 
the
local minima
of
$\Theta$ (on the region where it is positive) are located at bottoms of uniformly convex wells. 
Part of the proof of Theorem \ref{thm:2:2} below is to prove that the particle
leaves such wells in infinitesimal time whenever it starts from the interior of one of them. 

To make it clear, we consider $u_{w}$ the minimizer of $\Theta$ on a given well ${\mathcal W}_{w}$. 
Thanks to $(\textbf{\bf B4})$, ${\mathcal W}_{w}$ may be written in the form of a level set 
$\{u \in {\mathbb S}^{d-1} : \Theta(u) \leq a_{w}, \  \vert u - u_{w} \vert < e_{w}\}$ for some 
$e_{w}>0$ and the intersection of the well with ${\mathcal B}_{a}$ is given by 
$\{ u \in {\mathbb S}^{d-1} : \Theta(u) = a_{w}, \ \vert u - u_{w} \vert <e_{w}\}$,  {${\mathcal B}_{a}$ containing 
$\{ u \in {\mathbb S}^{d-1} :  \vert u - u_{w} \vert =e_{w}\}$}.
In other words, the particle leaves the well if    
$\Theta(\theta^{\varepsilon})$ becomes greater than $a_{w}$. 
This prompts us to define, 
for any stopping time $\tau$, 
the exit time
\begin{equation*}
{\boldsymbol e}_{w}^{\varepsilon}(\tau) := \inf \{ t \geq \tau : \theta_{t}^{\varepsilon} \not \in {\mathcal W}_{w} \}.
\end{equation*}
Whenever
$\theta_{\tau}^{\varepsilon}$ belongs to ${\mathcal W}_{w}$, 
${\boldsymbol e}_{w}^{\varepsilon}(\tau)$ is also equal to 
${\boldsymbol e}_{w}^{\varepsilon}(\tau)
= \inf \{t \geq \tau : \Theta(\theta_{t}^{\varepsilon}) \geq a_{w}\}$.


Following the proof of Lemma 
\ref{lem:decomposition:gt}, we have

\begin{lemma}
\label{lem:decomposition:Thetat}
For a given stopping time $\tau$, for any $t \in [\tau,{\boldsymbol e}_{w}^{\varepsilon}(\tau))$, 
\begin{equation}
\label{eq:decomposition:Thetat}
\begin{split} 
\ud \Theta \bigl( \theta_{t}^{\varepsilon} \bigr) &=
 \Bigl(
 \bigl( {R_{t}^{\varepsilon}}  \bigr)^{\alpha} \DTheta\bigl( \theta_{t}^{\varepsilon}\bigr) \cdot \nabla g_{t}^{\varepsilon}
+ \frac{ \varepsilon^2}{2}
 \frac1{(R_{t}^{\varepsilon})^2}
 \textrm{\rm Trace} \bigl( \HTheta\bigl( \theta_{t}^{\varepsilon}\bigr) \bigr) \Bigr) \ud t 
+ \varepsilon
 \frac1{R_{t}^{\varepsilon}}
\DTheta\bigl( \theta_{t}^{\varepsilon}\bigr)  \cdot \ud B_{t}, 
\end{split}
\end{equation}
\end{lemma}

\begin{proof}
The proof is a standard application of It\^o's formula
 {to $X^{\varepsilon}$ and the function $\R^d \setminus \{0\} \ni x \mapsto \Theta(\tfrac{x}{\vert x \vert})$.} 
It is based upon the connection between Euclidean and spherical derivatives, 
 {see 
\eqref{eq:DTheta}
for the first order terms and 
\eqref{eq:Laplace:spheric} (replacing 
$g$ by $\Theta$ and $1+\alpha$ by $0$)
together with
 \eqref{eq:Hess:theta} for the second order derivatives}.
In addition, we make use of 
\eqref{eq:6:1} 
and of the fact that the last term therein
has a zero
contribution 
since $\DTheta(\theta_{t}^{\varepsilon}) \cdot \theta_{t}^{\varepsilon}=0$.
Alternatively, we can also apply It\^o's formula on the sphere using 
$\theta^{\varepsilon}$ as underlying process, see \eqref{eq:theta:bib}, and $\Theta$ as function. 
\end{proof}

We deduce that 
\begin{lemma}
\label{lem:6:9:b}
There exist two positive thresholds
 $a_{\star}:=a_{\star}(\textbf{\bf A},\textbf{\bf B})$
and
 $\delta_{\star}:=\delta_{\star}(\textbf{\bf A},\textbf{\bf B})$
 such that, for any $a \leq a_{\star}$
 and $\delta \leq \delta_{\star}$, 
 we can find constants $c:=c(\textbf{\bf A},\textbf{B},a)$ and 
$C:=C(\textbf{\bf A},\textbf{\bf B},a)$ 
such that, 
 for any stopping time $\tau$, 
the following holds true on 
the event $D^{0}(\tau) \cap \{Ê\theta_{\tau}^{\varepsilon} \in 
{\mathcal W}_{w}\}
$
and
for $t \in [\tau, \varrho^{\varepsilon}(\tau) \wedge {\boldsymbol \rho}^{\varepsilon}(\tau)
\wedge 
{\boldsymbol e}_{w}^{\varepsilon}(\tau))$: 
\begin{equation}
\label{eq:6:9:b}
\begin{split}
\bigl( \Theta(\theta_{t}^{\varepsilon})
-
\Theta(u_{w})\bigr) &\geq \int_{\tau}^t
  \frac{c}2 \frac{V_{s}^{\varepsilon}}{g_{s}^{\varepsilon} (R_{s}^{\varepsilon})^2}
\bigl[
\Theta\bigl( \theta_{s}^{\varepsilon} \bigr) - \Theta(u_{w})
\bigr] \ud s
+
c \frac{\varepsilon^2}{V_{\tau}^{\varepsilon}}
  \Bigl( 1 - 
  \exp \bigl( - C \Sigma_{\tau,t}^{\varepsilon}
 \bigr)
 \Bigr)
-
C \int_{\tau}^t  
V_{s}^{\varepsilon} \ud s
 \\
&\hspace{15pt}+
V_{\tau}^{\varepsilon}
\int_{\tau}^t 
\frac{c}{(R_{s}^{\varepsilon})^2}
\bigl\vert \DTheta(\theta_{s}^{\varepsilon})
 \bigr\vert^2
 \ud s
 + 
  \varepsilon
 \int_{\tau}^t \frac1{R_{s}^{\varepsilon}}
 \DTheta(\theta_{s}^{\varepsilon}) \cdot \ud B_{s}.
\end{split}
\end{equation}
\end{lemma}
\begin{proof}
We address the two terms entering the absolutely continuous part in the expansion  
\eqref{eq:decomposition:Thetat}. We know from
 (\textbf{B2}) that, for $\delta \leq \delta_{\star}:=\delta_{\star}(\textbf{\bf A},\textbf{\bf B})$
 and  $C:=C(\textbf{\bf A},\textbf{\bf B})$,
\begin{equation*}
\begin{split}
\bigl( R_{t}^{\varepsilon} \bigr)^{\alpha} \DTheta\bigl(\theta_{t}^{\varepsilon}\bigr) \cdot \nabla g_{t}^{\varepsilon}
&\geq 
\bigl( R_{t}^{\varepsilon} \bigr)^{\alpha-1}
\bigl\vert\DTheta\bigl({\theta_{t}^{\varepsilon}}\bigr)
 \bigr\vert^2 - C \bigl(R_{t}^{\varepsilon} \bigr)^{\alpha}
\bigl\vert \DTheta\bigl({\theta_{t}^{\varepsilon}}\bigr) \bigr\vert
\\ 
&= 
\frac{V_{t}^{\varepsilon}}{g_{t}^{\varepsilon} (R_{t}^{\varepsilon})^2}
\bigl\vert \DTheta\bigl(\theta_{t}^{\varepsilon} \bigr)
 \bigr\vert^2 - C \bigl(R_{t}^{\varepsilon} \bigr)^{\alpha} \bigl\vert \DTheta\bigl(\theta_{t}^{\varepsilon} \bigr) \bigr\vert,
\end{split}
\end{equation*}
with $C:=C(\textbf{\bf A},\textbf{\bf B})$. 
We now make use of  
the local strict convexity property of $\Theta$ stated in (\textbf{\bf B4}). 
Obviously, it permits to lower bound the Trace term in 
\eqref{eq:decomposition:Thetat}. It also permits to lower bound 
$\vert \DTheta\bigl(\theta_{t}^{\varepsilon} \bigr)
\vert^2$ by 
$c[\Theta(\theta_{t}^{\varepsilon}) - \Theta(u_{w})]$
for $c:=c(\textbf{\bf A},\textbf{\bf B})$, see 
\eqref{eq:convexity:2}. 
Hence, we have, 
for
$t \in [\tau,{\boldsymbol \rho}^{\varepsilon}(\tau)
\wedge 
{\boldsymbol e}_{w}^{\varepsilon}(\tau))$,
\begin{equation*}
\begin{split}
\ud \Theta(\theta_{t}^{\varepsilon}) &\geq 
\biggl[ \frac{c}2 \frac{V_{t}^{\varepsilon}}{g_{t}^{\varepsilon} (R_{t}^{\varepsilon})^2}
\bigl( 
\Theta(\theta_{t}^{\varepsilon}) - \Theta(u_{w})
\bigr)
+
\frac{\varepsilon^2}{2} \frac{c}{(R_{t}^{\varepsilon})^2}
-
 C \bigl(R_{t}^{\varepsilon} \bigr)^{\alpha} \bigl\vert \DTheta\bigl(\theta_{t}^{\varepsilon} \bigr)
 \bigr\vert
+
\frac12 \frac{V_{t}^{\varepsilon}}{g_{t}^{\varepsilon} (R_{t}^{\varepsilon})^2}
\bigl\vert \DTheta\bigl(\theta_{t}^{\varepsilon} \bigr)
 \bigr\vert^2
 \biggr]
 \ud t
 \\
&\hspace{15pt} + 
  \varepsilon
 \frac1{R_{t}^{\varepsilon}}
 \DTheta\bigl(\theta_{t}^{\varepsilon} \bigr) \cdot \ud B_{t},
\end{split}
\end{equation*}
where $c:=c(\textbf{\bf A},\textbf{\bf B})$. 
We then notice that 
\begin{equation*}
\begin{split}
C \bigl( R_{t}^{\varepsilon} \bigr)^{\alpha} 
\bigl\vert \DTheta\bigl(\theta_{t}^{\varepsilon} \bigr)
 \bigr\vert
 &= C \frac{V_{t}^{\varepsilon}}{g_{t}^{\varepsilon} ( R_{t}^{\varepsilon})^2}  R_{t}^{\varepsilon} 
\bigl\vert \DTheta\bigl(\theta_{t}^{\varepsilon} \bigr)
 \bigr\vert
 \\
&   \leq \frac{1}{4}
 \frac{V_{t}^{\varepsilon}}{g_{t}^{\varepsilon} ( R_{t}^{\varepsilon})^2} 
 \bigl\vert \DTheta\bigl(\theta_{t}^{\varepsilon} \bigr)
 \bigr\vert^2 
 +
C^2 \frac{V_{t}^{\varepsilon}}{g_{t}^{\varepsilon}}
= \frac{1}{4}
 \frac{V_{t}^{\varepsilon}}{g_{t}^{\varepsilon} (R_{t}^{\varepsilon})^2} 
 \bigl\vert \DTheta\bigl(\theta_{t}^{\varepsilon} \bigr)
 \bigr\vert^2 
 +
C^2 \bigl( R_{t}^{\varepsilon} \bigr)^{1+\alpha}. 
\end{split}
\end{equation*}
Using the fact 
$\Theta(\theta_{t}^{\varepsilon}) \geq 
\Theta(u_{w})$
 {for
all
$t \in [\tau,
{\boldsymbol e}_{w}^{\varepsilon}(\tau))$}, we get, on the event
$ \{Ê\theta_{\tau}^{\varepsilon} \in 
{\mathcal W}_{w}\}$, 
\begin{equation*}
\begin{split}
 \Theta(\theta_{t}^{\varepsilon})
-
\Theta(u_{w})  &\geq \int_{\tau}^t
  \frac{c}2 \frac{V_{s}^{\varepsilon}}{g_{s}^{\varepsilon} (R_{s}^{\varepsilon})^2}
\bigl[
\Theta\bigl( \theta_{s}^{\varepsilon} \bigr) - \Theta(u_{w})
\bigr] \ud s
+
\int_{\tau}^t \Bigl( \frac{c}2  \frac{\varepsilon^2}{(R_{s}^{\varepsilon})^2} -
 C^2  \bigl(R_{s}^{\varepsilon}\bigr)^{1+\alpha}
\Bigr) \ud s
 \\
&\hspace{15pt}+
\int_{\tau}^t 
\frac{1}{4} \frac{V_{s}^{\varepsilon}}{ g_{s}^{\varepsilon} (R_{s}^{\varepsilon})^2}
\bigl\vert \DTheta\bigl(\theta_{s}^{\varepsilon} \bigr)
 \bigr\vert^2
 \ud s
 + 
  \varepsilon
 \int_{\tau}^t \frac1{R_{s}^{\varepsilon}}
 \DTheta\bigl(\theta_{s}^{\varepsilon} \bigr) \cdot \ud B_{s}.
\end{split}
\end{equation*}
If we work on $D^{0}(\tau)$ (see the definition in 
Proposition \ref{prop:good:events})
 and if we require $t$ to be less than $\varrho^{\varepsilon}(\tau)$, then we have the following two lower bounds (for new values of 
$c:=c(\textbf{\bf A},\textbf{\bf B},a)$
and
$C:=C(\textbf{\bf A},\textbf{\bf B},a)$ that may now depend on $a$):
\begin{equation*}
\begin{split}
&\int_{\tau}^t 
\frac{V_{s}^{\varepsilon}}{ g_{s}^{\varepsilon} (R_{s}^{\varepsilon})^2}
\bigl\vert \DTheta\bigl(\theta_{s}^{\varepsilon} \bigr)
 \bigr\vert^2
 \ud s \geq
 V_{\tau}^{\varepsilon}
  \int_{\tau}^t 
\frac{c}{(R_{s}^{\varepsilon})^2}
\bigl\vert \DTheta\bigl(\theta_{s}^{\varepsilon} \bigr)
 \bigr\vert^2
 \ud s, 
 \\
 &\int_{\tau}^t  \frac{\varepsilon^2}{ (R_{s}^{\varepsilon})^2} \ud s
 \geq 
\frac{{c} \varepsilon^2}{V_{\tau}^{\varepsilon}}
 \int_{\tau}^t  \frac{V_{s}^{\varepsilon}}{g_{s}^{\varepsilon} (R_{s}^{\varepsilon})^2}
 \exp \bigl( - C \Sigma_{\tau,s}^{\varepsilon}
 \bigr)
  \ud s
  = \frac{ {c}\varepsilon^2}{C V_{\tau}^{\varepsilon}}
  \Bigl( 1 - 
  \exp \bigl( - C \Sigma_{\tau,t}^{\varepsilon}
 \bigr)
 \Bigr),
\end{split}
\end{equation*}
where we used the fact that $(\ud/\ud t) \Sigma_{\tau,t}^{\varepsilon}=
V_{t}^{\varepsilon}/(g_{t}^{\varepsilon} (R_{t}^{\varepsilon})^2)$.
\end{proof}

We now focus on the second line in 
\eqref{eq:6:9:b}.
Using the same notations as in the statement of Proposition 
\ref{prop:good:events}, we obtain
\begin{lemma}
\label{lem:6:10:b}
Consider 
 $a_{\star}:=a_{\star}(\textbf{A},\textbf{B})$
 as in Lemma \ref{lem:6:9:b}
and the same two constants $c:=c(\textbf{\bf A},\textbf{B},a)$ and 
$C:=C(\textbf{\bf A},\textbf{\bf B},a)$ 
as therein, for a given $a \leq a_{\star}$. 
Then, there exists a positive constant $K(c,C)$ (only depending on $c$ and $C$) such that,
for any stopping time $\tau$, 
\begin{equation*}
{\mathbb P}
\biggl(
\forall t \geq \tau, \ 
V_{\tau}^{\varepsilon}
\int_{\tau}^t 
\frac{c}{(R_{s}^{\varepsilon})^2}
\bigl\vert \DTheta\bigl(\theta_{s}^{\varepsilon} \bigr)
 \bigr\vert^2
 \ud s
 + 
  \varepsilon
 \int_{\tau}^t \frac1{R_{s}^{\varepsilon}}
 \DTheta\bigl(\theta_{s}^{\varepsilon} \bigr) \cdot \ud B_{s}
\geq 
-\frac{c}{2} \frac{\varepsilon^2}{V_{\tau}^{\varepsilon}}
  \bigl( 1 - 
  \exp ( - C)
 \bigr)
 \, \big\vert 
 \, {\mathcal F}_{\tau}
 \biggr) \geq K(c,C).
\end{equation*}
\end{lemma}
\begin{proof}
The proof is rather straightforward. It suffices to multiply both terms
in the event appearing in the left-hand side 
by $c V_{\tau}^{\varepsilon}/\varepsilon^2$ 
and then to apply Doob's inequality to the resulting Dol\'eans-Dade martingale, 
see for instance Lemma \ref{lem:5:1}
for a similar use. 
\end{proof}

Combining the last two lemmas, we finally deduce:
\begin{proposition}
\label{prop:6:11}
 {There exist  {three} positive thresholds
 $a_{\star}:=a_{\star}(\textbf{\bf A},\textbf{\bf B})$,
 $\delta_{\star}:=\delta_{\star}(\textbf{\bf A},\textbf{\bf B})$
 and
 $v_{\star}:=v_{\star}(\textbf{\bf A},\textbf{\bf B})$
 such that, for any $a \leq a_{\star}$, 
 we can find constants $c:=c(\textbf{\bf A},\textbf{B},a)$ and 
$C:=C(\textbf{\bf A},\textbf{\bf B},a)$ 
satisfying, 
 for any stopping time $\tau$, any reals 
 $\delta \leq \delta_{\star}$ and $v_{0} \geq v_{\star}$, 
any integer $j \geq 1$, 
and on the event
 {$\{V_{\tau}^{\varepsilon} \geq v_{0} \varepsilon^2 \}$},}
\begin{equation*}
{\mathbb P}\bigl(D_{j}'(\tau) \, \vert \, {\mathcal F}_{\tau} \bigr) \geq  {\tfrac12} K(c,C),
\end{equation*}
with
\begin{equation}
\label{eq:D:j:prime:tau}
\begin{split}
D_{j}'(\tau)
&:= 
\biggl\{
\forall t \in \bigl[ {\sigma_{1}^{\varepsilon}(\tau)},
\sigma_{j}^{\varepsilon}(\tau) \wedge \varrho^{\varepsilon}(\tau) \wedge {\boldsymbol \rho}^{\varepsilon}(\tau)
\wedge 
{\boldsymbol e}_{w}^{\varepsilon}(\tau)\bigr), 
\\
&\hspace{30pt} \Theta(\theta_{t}^{\varepsilon})
-
\Theta(u_{w}) \geq 
\Bigl(
\frac{c}{2} \frac{\varepsilon^2}{V_{\tau}^{\varepsilon}}
\bigl( 1 - \exp(-C) \bigr) - C \int_{\tau}^{t} V_{s}^{\varepsilon} \ud s
\Bigr) \exp \Bigl( \frac{c}{2} \bigl( \Sigma_{\tau,t}-1\bigr) \Bigr) 
\biggr\},
\end{split}
\end{equation}
$\sigma_{j}^{\varepsilon}(\tau)$ being defined with $T=1$ in 
Proposition 
\ref{prop:good:events}.
\end{proposition}
\begin{proof}
{The first step is 
to apply 
\eqref{eq:6:9:b}
for $t \geq \sigma_{1}^{\varepsilon}(\tau)$. 
If the latter is strictly less than 
$\sigma_{j}^{\varepsilon}(\tau)$
(which is fact impossible when $j=1$ proving that 
$D_{1}'(\tau)=\Omega$ and that $j$ can be taken greater than or equal to 2), it is thus 
strictly less than ${\boldsymbol \epsilon}$. 
In particular $\Sigma^{\varepsilon}_{\tau,\sigma^{\varepsilon}_{1}(\tau)}$ is then equal to $1$. 
By \eqref{eq:6:9:b}, we get, 
on
$D^{0}(\tau)$ 
and
for $t \in [\sigma_{1}^{\varepsilon}(\tau),\sigma_{j}^{\varepsilon}(\tau) \wedge \varrho^{\varepsilon}(\tau) \wedge {\boldsymbol \rho}^{\varepsilon}(\tau)
\wedge 
{\boldsymbol e}_{w}^{\varepsilon}(\tau))$,  
\begin{equation*}
\begin{split}
\bigl( \Theta(\theta_{t}^{\varepsilon})
-
\Theta(u_{w})\bigr) &\geq \int_{\sigma_{1}^{\varepsilon}(\tau)}^t
  \frac{c}2 \frac{V_{s}^{\varepsilon}}{g_{s}^{\varepsilon} (R_{s}^{\varepsilon})^2}
\bigl[
\Theta\bigl( \theta_{s}^{\varepsilon} \bigr) - \Theta(u_{w})
\bigr] \ud s
+
c \frac{\varepsilon^2}{V_{\tau}^{\varepsilon}}
  \Bigl( 1 - 
  \exp \bigl( - C  
 \bigr)
 \Bigr)
-
C \int_{\tau}^t  
V_{s}^{\varepsilon} \ud s
 \\
&\hspace{15pt}+
V_{\tau}^{\varepsilon}
\int_{\tau}^t 
\frac{c}{(R_{s}^{\varepsilon})^2}
\bigl\vert \DTheta(\theta_{s}^{\varepsilon})
 \bigr\vert^2
 \ud s
 + 
  \varepsilon
 \int_{\tau}^t \frac1{R_{s}^{\varepsilon}}
 \DTheta(\theta_{s}^{\varepsilon}) \cdot \ud B_{s}.
\end{split}
\end{equation*}
Intersecting with the event appearing 
in the statement of 
Lemma \ref{lem:6:10:b},
we obtain for $t \in [\sigma_{1}^{\varepsilon}(\tau),\sigma_{j}^{\varepsilon}(\tau) \wedge \varrho^{\varepsilon}(\tau) \wedge {\boldsymbol \rho}^{\varepsilon}(\tau)
\wedge 
{\boldsymbol e}_{w}^{\varepsilon}(\tau))$,  
\begin{equation*}
\begin{split}
\bigl( \Theta(\theta_{t}^{\varepsilon})
-
\Theta(u_{w})\bigr) &\geq \int_{\sigma_{1}^{\varepsilon}(\tau)}^t
  \frac{c}2 \frac{V_{s}^{\varepsilon}}{g_{s}^{\varepsilon} (R_{s}^{\varepsilon})^2}
\bigl[
\Theta\bigl( \theta_{s}^{\varepsilon} \bigr) - \Theta(u_{w})
\bigr] \ud s
+
\frac{c}2 \frac{\varepsilon^2}{V_{\tau}^{\varepsilon}}
  \Bigl( 1 - 
  \exp \bigl( - C  
 \bigr)
 \Bigr)
-
C \int_{\tau}^t  
V_{s}^{\varepsilon} \ud s.
\end{split}
\end{equation*}
By
Gronwall's lemma, we deduce that (notice that the lower bound below holds true even though the factor in front of the exponential is negative)
\begin{equation*}
\begin{split}
\Theta(\theta_{t}^{\varepsilon})
-
\Theta(u_{w}) &\geq 
\Bigl(
\frac{c}{2} \frac{\varepsilon^2}{V_{\tau}^{\varepsilon}}
\bigl( 1 - \exp(-C) \bigr) - C \int_{\tau}^{t} V_{s}^{\varepsilon} \ud s
\Bigr) \exp \Bigl( \frac{c}{2} \bigl( \Sigma_{\sigma_{1}^{\varepsilon}(\tau),t}\bigr) \Bigr)
\\
&=
\Bigl(
\frac{c}{2} \frac{\varepsilon^2}{V_{\tau}^{\varepsilon}}
\bigl( 1 - \exp(-C) \bigr) - C \int_{\tau}^{t} V_{s}^{\varepsilon} \ud s
\Bigr) \exp \Bigl( \frac{c}{2} \bigl( \Sigma_{\tau,t}-1\bigr) \Bigr).
\end{split}
\end{equation*}
It remains to invoke
Proposition 
\ref{prop:good:events} and Lemma \ref{lem:6:10:b} 
in order to lower bound the probability on which the above holds true.}
\end{proof}

The meaning of Proposition 
\ref{prop:6:11} is pretty clear: If 
$\sigma_{j}^{\varepsilon}(\tau) \leq \varrho^{\varepsilon}(\tau) \wedge {\boldsymbol \rho}^{\varepsilon}(\tau)
\wedge 
{\boldsymbol e}_{w}^{\varepsilon}(\tau)$, then we can 
choose $t = 
\sigma_{j}^{\varepsilon}(\tau)$ in the above statement. Provided that 
$\sigma_{j}^{\varepsilon}(\tau)$ is infinitesimal (and is thus less than ${\boldsymbol \epsilon}$), we get 
\begin{equation*}
\biggl( \frac{c}{2} \frac{\varepsilon^2}{V_{\tau}^{\varepsilon}}
\bigl( 1 - \exp(-C) \bigr) - C \int_{\tau}^{\sigma_j^{\varepsilon}(\tau)} V_{s}^{\varepsilon} \ud s
\biggr) \exp \bigl( \frac{c}{2} (j-1) \bigr),
\end{equation*}
as lower bound for $ \Theta\bigl(\theta_{\sigma_{j}^{\varepsilon}(\tau)}^{\varepsilon}\bigr)
-
\Theta(u_{w})$ on the event $D_{j}'(\tau)$. The simple fact that $\sigma_{j}^{\varepsilon}(\tau)$ is infinitesimal allows to get rid of the integral quite easily. Then, provided that $j$ can be chosen large enough, we can make
the above lower bound greater than the required threshold $a_{w}$, which implies that 
$\sigma_{j}^{\varepsilon}(\tau)$ is in fact greater than ${\boldsymbol e}_{w}^{\varepsilon}(\tau)$. This means in particular that the particle leaves the well in infinitesimal time. 
We make this argument clear in the next subsection.

\subsection{End of the proof of Theorem \ref{thm:2:2}}
We now complete the proof of Theorem \ref{thm:2:2}.

\begin{proof}[Proof of Theorem \ref{thm:2:2}]
The main idea is to divide the analysis in two mains cases according to the
value of the potential whenever the process $g^{\varepsilon}$ hits for the {last} time a given threshold $a>0$ (and then remains above this threshold forever). 
It may happen that $g^{\varepsilon}$ becomes of order 1 once for all quite quickly even though the potential is pretty small or
it may happen 
that the potential is already large (on a given scale) whenever the process $g^{\varepsilon}$ reaches the given threshold for the last time. 
\vskip 4pt

\textit{First step.}
{We here introduce the tools that are needed for the proof. 
We fix some $\pi \in (0,1)$ (standing for the probability that the particle behaves differently from what we claimed in the statement 
of 
Theorem \ref{thm:2:2}).}
We are also given some {${\boldsymbol \epsilon} \in (0,1)$ and 
$\delta \in (0, \delta_{\star}]$
with $\delta_{\star}$ being the minimum of the two 
$\delta_{\star}$'s given in the statements of
Propositions \ref{prop:good:events} and \ref{prop:6:11}, but the values of both 
${\boldsymbol \epsilon}$
and  
$\delta$
are allowed to decrease in the proof}. Moreover, 
following the proof of Theorem \ref{thm:2:1}
in Subsection \ref{subse:5:5}, we can assume without any loss of generality that 
{$g_{0}^{\varepsilon} >0$, 
$V_{0}^{\varepsilon} = v_{0} \varepsilon^2$, for some $v_{0} >0$}, and, for any $A>0$, 
\begin{equation*}
\lim_{\varepsilon \searrow \infty} {\mathbb P}
\Bigl(  
V_{0}^{\varepsilon}
\geq 
A (R_{0}^{\varepsilon})^2 
\Bigr) = 1. 
\end{equation*} 
We know from
{Lemma \ref{lem:5:1}}
and
 Proposition 
\ref{prop:hitting:g:a} that there exist two thresholds {$a_{\star}:=a_{\star}(\textbf{\bf A},\textbf{\bf B})>0$ and 
$v_{\star}:=v_{\star}(\textbf{\bf A},\textbf{\bf B},\pi) >0$} ({the values of which may vary from line to line in the proof as long as they only depend on the parameters specified in parentheses}) and a sequence of infinitesimal times 
$(t_{\varepsilon})_{\varepsilon>0}$ such that, 
{for 
$a \in (0,a_{\star}]$ and $v_{0} \geq v_{\star}$},
\begin{equation}
\label{eq:liminf:supplementaire}
\liminf_{\varepsilon \searrow 0}
{\mathbb P}
\Bigl(
{ \bigl\{ \zeta^{\varepsilon} \leq t_{\varepsilon} \bigr\} \cap \bigl\{ \forall t \geq 0,
\ V_{t}^{\varepsilon} \geq  \frac{v_{0}}2 \varepsilon^2 \bigr\}} \Bigr) \geq 1 - { \pi} - \exp\bigl(-\frac{v_{0}}8 \bigr),
\end{equation}
where $\zeta^{\varepsilon}$ denotes the stopping time:
\begin{equation*}
\zeta^{\varepsilon}
:=
\inf \bigl\{t \geq 0 : g_{t}^{\varepsilon} \geq 2a
 \bigr\}. 
\end{equation*}
By {Lemma
 \ref{lem:staying:g:above:a}}, by Markov's property
%
{and by \eqref{eq:liminf:supplementaire}, we get, for $a \in (0,a_{\star}]$, $v_{0} \geq v_{\star}$
and $\varepsilon$ small enough},
\begin{equation*}
{\mathbb P}
\Bigl( 
{\bigl\{
\forall t \in [\zeta^{\varepsilon},{\boldsymbol \epsilon}], 
\ g_{t}^{\varepsilon} \geq a 
\bigr\}
\cap
\bigl\{ \zeta^{\varepsilon} \leq t_{\varepsilon} \bigr\} \cap \bigl\{ \forall t \geq 0,
\ V_{t}^{\varepsilon} \geq  \frac{v_{0}}2 \varepsilon^2 \bigr\}}
\Bigr) \geq 1 - {2} \pi - C \exp \bigl( - \frac{v_{0}}{c} \bigr),
\end{equation*} 
where $c:=c(\textbf{\bf A},\textbf{B})$ and $C:=C(\textbf{\bf A},\textbf{\bf B})$. 
Now, by Proposition \ref{prop:5:6}, there exist $\psi':=\psi'(\textbf{\bf A},\textbf{\bf B}) >0$ and a new value of $C:=C(\textbf{\bf A},\textbf{\bf B})>1$, such that, for $v_{0} \geq v_{\star}(\textbf{\bf A},\textbf{\bf B},\pi)$
and $\varepsilon$ small enough, 
\begin{equation*}
{\mathbb P}
\Bigl( \forall t \in [0,{\boldsymbol \epsilon}], \ { C^{-1} t^{\psi'}} \leq V_{t}^{\varepsilon} \leq 
C t + \varepsilon^{1/2}
\Bigr) \geq 1 - \pi,
\end{equation*}
{where, to get the upper bound in the above event, we used the fact that $\nabla V$ is bounded on the whole space}. 
Also, allowing the value of the constant $C$ to increase from line to line, we have, 
\begin{equation*}
\begin{split}
{\mathbb P} \Bigl( \forall t \in [0,{\boldsymbol \epsilon}], \ R_{t}^{\varepsilon} < \delta \Bigr) 
&\geq {\mathbb P} \Bigl( \forall t \in [0,{\boldsymbol \epsilon}], \ R_{0}^{\varepsilon} + \varepsilon \vert B_{t}^{\varepsilon} \vert < \delta - C {\boldsymbol \epsilon} \Bigr)
\\
&\geq {\mathbb P} \Bigl( R_{0}^{\varepsilon}  < \tfrac12 \delta - \tfrac12 C {\boldsymbol \epsilon} \Bigr) +
{\mathbb P} \Bigl( \forall t \in [0,{\boldsymbol \epsilon}], \  \varepsilon \vert B_{t}^{\varepsilon} \vert < \tfrac12 \delta - \tfrac12 C {\boldsymbol \epsilon} \Bigr)
- 1.
\end{split}
\end{equation*}
Choosing ${\boldsymbol \epsilon}< {\boldsymbol \epsilon}_{\star}(\textbf{\bf A},\textbf{\bf B},\delta)$ and then choosing $\varepsilon$ small enough, the right-hand side can be made greater than $1-\pi$. 
Call now
\begin{equation*}
\begin{split}
{\mathcal E}_{0}^{\varepsilon} &:=
\Bigl\{ \zeta^{\varepsilon} \leq t_{\varepsilon} \Bigr\}  
\cap
\Bigl\{
\forall t \in [\zeta^{\varepsilon},{\boldsymbol \epsilon}], 
\ g_{t}^{\varepsilon} \geq a 
\bigr\}
\cap
\Bigl\{
\forall t \in [0,{\boldsymbol \epsilon}], \ \max(\tfrac12 v_{0} \varepsilon^2, C^{-1}t^{\psi'}) \leq V_{t}^{\varepsilon} \leq 
C t + \varepsilon^{1/2}
\Bigr\} 
\\
&\hspace{15pt} 
\cap
\Bigl\{  \forall t \in [0,{\boldsymbol \epsilon}], R_{t}^{\varepsilon} < \delta \Bigr\},
\end{split}
\end{equation*}
so that, for $v_{0} \geq v_{\star}(\textbf{\bf A},\textbf{\bf B},\pi)$
and
${\boldsymbol \epsilon}< {\boldsymbol \epsilon}_{\star}(\textbf{\bf A},\textbf{\bf B},\delta)$,
and for $\varepsilon$ small enough,
 {${\mathbb P} \bigl( {\mathcal E}_{0}^{\varepsilon} \bigr) \geq 1 -  {5} \pi$}, 
and, on ${\mathcal E}_{0}^{\varepsilon}$,
\begin{equation*}
{ \max \Bigl( \tfrac12 v_{0} \varepsilon^2,C^{-1} (\zeta^{\varepsilon})^{\psi'}\Bigr) \leq V_{\zeta^{\varepsilon}}^{\varepsilon} \leq C \zeta^{\varepsilon} + \varepsilon^{1/2} \leq C t_{\varepsilon} +\varepsilon^{1/2} },
\end{equation*}
the term $t_{\varepsilon}':= C t_{\varepsilon} + \varepsilon^{1/2}$ standing
for a new infinitesimal term. 

On the event ${\mathcal E}_{0}^{\varepsilon}$, 
 we have, with the same notations as in the statement of 
 Proposition \ref{prop:good:events},
{ $\varrho^{\varepsilon}(\zeta^{\varepsilon}) \geq {\boldsymbol \epsilon}$}
and {$V^{\varepsilon}_{\zeta^{\varepsilon}}$} is infinitesimal. 
We also have ${\boldsymbol \rho}^{{\boldsymbol \epsilon}}(\tau) \geq {\boldsymbol \epsilon}$ for any stopping time $\tau$ with values in 
$[0,{\boldsymbol \epsilon}]$. 
\vskip 4pt

{We now study the behavior of the particle
on each of the two events ${\mathcal E}_{0}^{\varepsilon} \cap \{\zeta^{\varepsilon} > 0\}$ and 
${\mathcal E}_{0}^{\varepsilon} \cap \{ \zeta^{\varepsilon} = 0\}$ using therein similar 
ranges of values for the parameters $a$, $\delta$, ${\boldsymbol \epsilon}$ and $v_{0}$.}
\vskip 4pt

\textit{Second step.}
We first address the case $\zeta^{\varepsilon} > 0$. 
Using the same notation as in Proposition
\ref{prop:good:events}
and Corollary
\ref{coro:6:7}, 
we consider the event
\begin{equation*}
{{\mathcal E}^{\varepsilon} := {\mathcal E}^{\varepsilon}_{0} \cap D^2_{0}(t_{\varepsilon}) \cap 
 {\mathcal D}(\zeta^{\varepsilon})}.
\end{equation*}
%
%
By Proposition
\ref{prop:good:events}
{(see $(i)$ and $(iii)$ therein)}
and
Corollary \ref{coro:6:7}, 
it has probability greater than $1 - {6} \pi$
{for the same range of parameters as before except that 
 $v_{0}$ is now taken greater than $v_{\star}(\textbf{\bf A},\textbf{\bf B},\delta,\pi)$}.

Below, we make use of assumptions 
(\textbf{B3-a}), (\textbf{B3-b}) and (\textbf{B3-c}). For a given value of $r>0$, we can choose $r':=r'(r)$
as in (\textbf{B3-b}). For such an $r'$, we choose $T \geq T_{0}$ where $T_{0}:=T_{0}({r'/2})$ is given by (\textbf{B3-a}) but with $r=r'/2$ therein. 

Work now on ${\mathcal E}^{\varepsilon}$ and call $r''$ a (small) positive real. 
{Since $\zeta^{\varepsilon} >0$, we have $g_{\zeta^{\varepsilon}}^{\varepsilon} = a$}.
Recalling from (\textbf{B2}) that 
$\vert \Theta(\theta_{\zeta^{\varepsilon}}^{\varepsilon}) - g_{\zeta^{\varepsilon}}^{\varepsilon} \vert 
\leq C R_{\zeta^{\varepsilon}}^{\varepsilon} \leq C \delta$ (the value of $C$ being allowed to increase from one inequality to another), 
we can choose $\delta$ small enough ({and so ${\boldsymbol \epsilon}$ and $v_{0}$ accordingly as they depend on $\delta$}) such that 
$${\theta^{\varepsilon}_{\zeta^{\varepsilon}}
 \in \Theta^{-1} \bigl(  {[ a,3a ]} \bigr)
 \subset {\mathcal B}_{a}
 },$$
{the inclusion following from (\textbf{B3-c})}.
Then, (\textbf{B3-a}) says that
\begin{equation}
\label{eq:proof:2.6:s1}
\forall t \geq \sigma_{{1}}^{\varepsilon}({\zeta^{\varepsilon}}),
\quad 
\textrm{\rm dist}\Bigl( \phi_{\Sigma^{\varepsilon}_{{\zeta^{\varepsilon}},t \wedge {\boldsymbol \epsilon}}}^{\theta^{\varepsilon}_{{\zeta^{\varepsilon}}}},{\mathcal S} \Bigr) \leq \tfrac12 r',
\end{equation}
provided that $\sigma_{{1}}^{\varepsilon}({\zeta^{\varepsilon}}) < {\boldsymbol \epsilon}$ (which we check later on)
Also, since we are on 
{$D^1_{0}(\zeta^{\varepsilon})$ and 
${\varrho}^{\varepsilon}(\zeta^{\varepsilon}) \wedge 
{\boldsymbol \rho}^{\varepsilon}(\zeta^{\varepsilon}) \geq {\boldsymbol \epsilon}$},
we have 
\begin{equation}
\label{eq:proof:2.6:s2}
\Bigl\vert 
\theta^{\varepsilon}_{\sigma_{1}^{\varepsilon}(\zeta^{\varepsilon})}
-
\phi_{\Sigma^{\varepsilon}_{\zeta^{\varepsilon},
\sigma_{1}^{\varepsilon}(\zeta^{\varepsilon})}}^{\theta^{\varepsilon}_{\zeta^{\varepsilon}}} 
\Bigr\vert
\leq C \Bigl( \delta + \frac{\varepsilon^2}{V^{\varepsilon}_{\zeta^{\varepsilon}}} \Bigr) \leq 2 C
\Bigl( \delta + \frac1{v_{0}} \Bigr),
\end{equation}
for $C:=C(\textbf{\bf A},\textbf{\bf B},a,T)$. 
If we require 
{the right-hand side to be less than $r''$, {namely 
$\delta \leq \delta_{\star}(\textbf{\bf A},\textbf{\bf B},{a},r'',T)$ and 
$v_{0} \geq v_{\star}(\textbf{\bf A},\textbf{\bf B},{a},\delta,\pi,r'',T)$}, and then if we take $r'' < r'/2$}, we get, 
by combining
\eqref{eq:proof:2.6:s1}
and
\eqref{eq:proof:2.6:s2}, 
\begin{equation*}
\textrm{\rm dist}\Bigl( \theta^{\varepsilon}_{\sigma_{{1}}^{\varepsilon}({\zeta^{\varepsilon}})},{\mathcal S} \Bigr) \leq r',
\end{equation*}
{Hence, by (\textbf{B3-b}),  $\theta^{\varepsilon}_{\sigma_{{1}}^{\varepsilon}({\zeta^{\varepsilon}})} \in 
{\mathcal B}_{a}$ and, for all $t \geq \sigma_{1}^{\varepsilon}(\zeta^{\varepsilon})$, 
\begin{equation*}
\textrm{\rm dist}\Bigl( \phi_{t}^{\theta^{\varepsilon}_{\sigma_{{1}}^{\varepsilon}({\zeta^{\varepsilon}})}},{\mathcal S} \Bigr) \leq r.
\end{equation*}
Now, invoking the definition of $D^1_{1}(\zeta^{\varepsilon})$, we deduce that 
\begin{equation*}
\sup_{\sigma_{1}^{\varepsilon}(\zeta^{\varepsilon}) \leq t \leq 
\sigma_{2}^{\varepsilon}(\zeta^{\varepsilon})}
\textrm{\rm dist}\bigl( \theta_{t}^{\varepsilon},{\mathcal S} \bigr)
\leq \bigl( r'' + r \bigr). 
\end{equation*}
Without any loss of generality, we can take $r'' \leq r$, in which case the above right-hand side is less than 
$2r$.}
{Observing that we obtained this conclusion by using the facts that 
we are on ${\mathcal E}^{\varepsilon}$, that 
{$\theta^{\varepsilon}_{\zeta^{\varepsilon}} \in 
{\mathcal B}_{a}$} and that
{$\sigma^{\varepsilon}_{1}(\zeta^{\varepsilon}) < {\boldsymbol \epsilon}$}, 
we deduce by induction that, for any $j \geq 0$ such that 
{$\sigma_{j+1}^{\varepsilon}(\zeta^{\varepsilon}) < {\boldsymbol \epsilon}$}, 
{$ \theta^{\varepsilon}_{\sigma_{j+1}^{\varepsilon}(\zeta^{\varepsilon})}
 \in {\mathcal B}_{a}$
 and
$\sup_{\sigma_{j+1}^{\varepsilon}(\zeta^{\varepsilon}) \leq t \leq 
\sigma_{j+2}^{\varepsilon}(\zeta^{\varepsilon})}
\textrm{\rm dist} ( \theta_{t}^{\varepsilon},{\mathcal S})
\leq 2r$. We deduce that, for all $t \in [\sigma_{1}^{\varepsilon}(\zeta^{\varepsilon}) \wedge {\boldsymbol \epsilon},{\boldsymbol \epsilon})$},
\begin{equation}
\label{eq:dist:S}
\textrm{\rm dist}\bigl( \theta^{\varepsilon}_{t},{\mathcal S} \bigr) \leq 2r.
\end{equation}}
(Notice that, on the event ${\mathcal E}^{\varepsilon}$, 
the path $[{\zeta^{\varepsilon}} \wedge {\boldsymbol \epsilon}, {\boldsymbol \epsilon}] \ni t \mapsto \Sigma_{{\zeta^{\varepsilon}},t}^{\varepsilon}$ is continuous {and increasing}, and so $\sigma_{j'}^{\varepsilon}({\zeta^{\varepsilon}})$ is equal to ${\boldsymbol \epsilon}$ for some (random) $t \geq \zeta^{\varepsilon}$.) 
 
It remains to prove that ${\sigma_{1}^{\varepsilon}(\zeta^{\varepsilon})}$ can be bounded by an  infinitesimal (deterministic) time.  To do so, we notice that, 
since
we work on ${\mathcal E}^{\varepsilon} \subset {\mathcal E}_{0}^{\varepsilon}$, we have
(see the {definition of ${\mathcal E}_{0}^{\varepsilon}$ in the first step})
\begin{equation*}
\tfrac12 v_{0} \varepsilon^2 \leq V^{\varepsilon}_{t_{\varepsilon}} \leq t_{\varepsilon}'. 
\end{equation*}
Moreover, we observe that, if ${\sigma_{1}^{\varepsilon}(\zeta^{\varepsilon})} > t_{\varepsilon}$, then 
\begin{equation*}
T \geq 
\Sigma^{\varepsilon}_{{\zeta^{\varepsilon}},
{\sigma_{1}^{\varepsilon}(\zeta^{\varepsilon})}}
\geq 
\Sigma^{\varepsilon}_{t_{\varepsilon},{\sigma_{1}^{\varepsilon}(\zeta^{\varepsilon})}}.
\end{equation*}
On ${\mathcal E}^{\varepsilon} \subset {\mathcal E}_{0}^{\varepsilon}$, 
$\varrho^{\varepsilon}(t_{\varepsilon}) \geq {\boldsymbol \epsilon}$ and 
${\boldsymbol \rho}^{\varepsilon}(t_{\varepsilon}) \geq {\boldsymbol \epsilon}$. Also, 
${\sigma_{1}^{\varepsilon}(\zeta^{\varepsilon})}
\leq \sigma_{1}^{\varepsilon}(t_{\varepsilon})$. 
 Since
 we work on $D^2_{0}(t_{\varepsilon})$, we get
 \begin{equation*}
\begin{split}
 T &\geq  C^{-1} 
\Bigl( 
1 + \frac{\varepsilon^2}{V_{t_{\varepsilon}}^{\varepsilon}}
\Bigr)^{-1}
\ln \Bigl[ 1 + C^{-1} (2 V_{t_{\varepsilon}}^{\varepsilon})^{-\tfrac{1-\alpha}{1+\alpha}}
\bigl( {\sigma_{1}^{\varepsilon}(\zeta^{\varepsilon})} \wedge {\boldsymbol \epsilon} - t_{\varepsilon}\bigr)_{+}\Bigr]
\\
&\geq C^{-1} 
\Bigl( 
1 + \frac{2}{v_{0}}
\Bigr)^{-1}
\ln \Bigl[ 1 + C^{-1} (2 t_{\varepsilon}')^{-\tfrac{1-\alpha}{1+\alpha}}
\bigl( {\sigma_{1}^{\varepsilon}(\zeta^{\varepsilon})} \wedge {\boldsymbol \epsilon} - t_{\varepsilon}\bigr)_{+}\Bigr].
\end{split}
 \end{equation*}
 Without any loss of generality, we can assume that $v_{0} \geq 2$, from which we obtain
 \begin{equation*}
\begin{split}
 2CT &\geq 
\ln \Bigl[ 1 + C^{-1} (2 t_{\varepsilon}')^{-\tfrac{1-\alpha}{1+\alpha}}
\bigl( {\sigma_{1}^{\varepsilon}(\zeta^{\varepsilon})} \wedge {\boldsymbol \epsilon} - t_{\varepsilon}\bigr)_{+}\Bigr],
\end{split}
 \end{equation*} 
and then
\begin{equation*}
2 C  \bigl( 2t_{\varepsilon}' \bigr)^{\tfrac{1-\alpha}{1+\alpha}}  \Bigl( \exp \bigl( 2 C T\bigr) - 1 \Bigr) \geq  
{\sigma_{1}^{\varepsilon}(\zeta^{\varepsilon})} \wedge {\boldsymbol \epsilon} - t_{\varepsilon},
\end{equation*}
which shows that 
{ $\sigma_{1}^{\varepsilon}(\zeta^{\varepsilon})$} is indeed less than 
 {some} infinitesimal (deterministic) time $t_{\varepsilon}''$.
Returning to 
\eqref{eq:dist:S}, this says that, on the event ${\mathcal E}^{\varepsilon} \cap \{ {\zeta^{\varepsilon}}>0\}$, 
\begin{equation*}
\forall t \in [t_{\varepsilon}'',{\boldsymbol \epsilon}], \quad \textrm{\rm dist}\bigl(\theta_{t}^{\varepsilon},{\mathcal S}\bigr) \leq {2r}, 
\end{equation*}
provided we choose in a sequential manner $a=a(\textbf{\bf A},\textbf{\bf B})$, 
$r'=r'(\textbf{\bf A},\textbf{\bf B},r)$, 
$T=T(r')$,
$r''=r''(\textbf{\bf A},\textbf{\bf B},a,r,r',T)$,
$\delta=\delta(\textbf{\bf A},\textbf{\bf B},{a},r'',T)$,
${\boldsymbol \epsilon}={\boldsymbol \epsilon}(\textbf{\bf A},\textbf{\bf B},\delta)$,
$v_{\star} = v_{\star}(\textbf{\bf A},\textbf{\bf B},{a},\delta,\pi,r'',T)$
and 
 $\varepsilon$ small enough. This shows that 
 \begin{equation*}
 {\mathbb P}
 \Bigl( \bigl\{ \zeta^{\varepsilon} >0 \bigr\}
\cap
\bigl\{ \exists t \in [t_{\varepsilon}'',{\boldsymbol \epsilon}] : \textrm{\rm dist}\bigl(\theta_{t}^{\varepsilon},{\mathcal S}\bigr) 
> 2r \bigr\}\Bigr) \leq 6 \pi.
\end{equation*}

\vskip 4pt

\textit{Third step.} We now treat the case $\zeta^{\varepsilon} =0$.
{Then, on 
 ${\mathcal E}_{0}^{\varepsilon}$, the process $g^{\varepsilon}$ never passes below $a$. 
So, even though we work on 
${\mathcal E}^{\varepsilon}_{0}$ as in the previous step,  
we need to proceed differently in order to prove that the particle hits ${\mathcal B}_{a}$. While we used assumption 
(\textbf{B3}) in the previous step, we use here assumption (\textbf{B4}).  
Once the particle has been proved to reach ${\mathcal B}_{a}$ with high probability, 
it may be shown to converge to ${\mathcal S}$ by using the same arguments as in the previous step. Hence, we only prove here that, with (asymptotically) high probability,
the particle hits ${\mathcal B}_{a}$ in infinitesimal time. 
This suffices for our purpose.}

  {Throughout the proof, we use the same parameters $a$, $\delta$, ${\boldsymbol \epsilon}$, $\varepsilon$, $\pi$ and $v_{0}$ as before,
  but for possibly smaller ranges of values, which is indeed licit provided that there is no conflict in the order that is used to fix one parameter in terms of the others. Here, we do not make use of $r$ and we use another value of $T$, which is indeed possible since $T$ is a free parameter.} 

{The idea of the proof is as follows. If $\zeta^{\varepsilon} =0$ and $\theta_{0}^{\varepsilon} \in {\mathcal B}_{a}$, 
the proof is over. If $\zeta^{\varepsilon} =0$ and $\theta_{0}^{\varepsilon} \not \in {\mathcal B}_{a}$,
we deduce from 
(\textbf{\bf B4}) that $\theta_{0}^{\varepsilon}$ belongs to the well formed by a local minimum
of the function $\Theta$ on the sphere. Then, in order to complete the proof, it suffices to prove that, conditional 
on the fact that {$\theta_{0}^{\varepsilon}$ belongs to a well}, $\theta^{\varepsilon}$ exits 
the well in infinitesimal time
{with high probability on 
${\mathcal E}_{0}^{\varepsilon}$}.
Since we assumed the number of wells to be finite, this obviously implies that 
$\theta^{\varepsilon}$
reaches ${\mathcal B}_{a}$ in infinitesimal time 
 {with high probability} 
 on the event
${\mathcal E}_{0}^{\varepsilon} \cap \{ \zeta^{\varepsilon}=0\}$.

Following 
Subsection \ref{subse:escape:well}, we consider $u_{w}$ the minimizer of $\Theta$ on a given well ${\mathcal W}_{w}$. 
We know that ${\mathcal W}_{w}$ may be written in the form of a level set 
$\{u \in {\mathbb S}^{d-1} : \Theta(u) \leq a_{w}, \ \vert u - u_{w} \vert < e_{w}\}$ for some 
$a_{w},e_{w}>0$ and the intersection of the well with ${\mathcal B}_{a}$ is  {in the form}
$\{ u \in {\mathbb S}^{d-1} : \Theta(u) = a_{w}, \ \vert u - u_{w} \vert<e_{w}\}$, 
 {${\mathcal B}_{a}$ containing 
$\{ u \in {\mathbb S}^{d-1} : \ \vert u - u_{w} \vert =e_{w}\}$}.
In other words, it suffices to show, for our purposes, that   
$\Theta(\theta^{\varepsilon})$ becomes greater than $a_{w}$ in infinitesimal time. 
As we already alluded to, this  {turns out to be} a consequence of Proposition 
\ref{prop:6:11}. 

In order to proceed, we use the same two constants $c$ and $C$ as in the statement of Proposition \ref{prop:6:11}. 
 {We recall that $c$ and $C$ depend on $a$ and that Proposition 
\ref{prop:6:11} is valid if 
$a \leq a_{\star}$, $\delta \leq \delta_{\star}$ and $v_{0} \geq v_{\star}$, 
for $a_{\star}$, $\delta_{\star}$ and $v_{\star}$ as therein.}
For such a $v_{0}$, we call ${\mathcal J}(v_{0})$ the smallest integer $j \geq 1$ such that 
\begin{equation}
\label{eq:j(v0)}
\frac{c}{8 v_{0}}
\bigl( 1 - \exp(-C) \bigr) 
 \exp \Bigl( \frac{c}2 \bigl( j-1 \bigr) \Bigr) \geq a_{w}. 
\end{equation}
Inductively, we define the following two sequences of integers $(j_{n})_{n \geq 0}$
and $(J_{n})_{n \geq 0}$:
\begin{equation*}
\begin{split}
&j_{0}=0, \quad j_{1}={\mathcal J}(v_{0}), \quad j_{n+1}={\mathcal J} \bigl(v_{0} \exp\bigl(C (j_{1}+\cdots+j_{n})\bigr), \quad n \geq 1;
\\
&J_{0}=0, \quad J_n=j_{1}+\cdots+j_{n}, \quad n \geq 1. 
\end{split}
\end{equation*}
Observe in particular that 
$j_{n+1}={\mathcal J}(v_{0} \exp(CJ_{n}))$. 
Accordingly, 
we consider the events
$(D_{j_{n+1}}'(\sigma_{J_{n}}^{\varepsilon}(0)))_{n \geq 0}$, 
see 
Proposition \ref{prop:good:events}
for the definition of $\sigma_{j}^{\varepsilon}(0)$
({with $T=1$ therein})
and
\eqref{eq:D:j:prime:tau} for the definition of $D_{j}'(\tau)$. 
Noticing 
that 
$\sigma_{J_{n+1}}^{\varepsilon}(0)=
\sigma^{\varepsilon}_{j_{n+1}-j_{n}}(\sigma_{J_{n}}^{\varepsilon}(0))$, we deduce that 
each $D_{j_{n+1}}'(\sigma_{J_{n}}^{\varepsilon}(0))$ belongs to 
${\mathcal F}_{\sigma_{J_{n+1}}^{\varepsilon}(0)}$. Moreover, by 
Proposition \ref{prop:6:11},
\begin{equation*}
{\mathbb P} \bigl( D_{j_{n+1}}'(\sigma_{J_{n}}^{\varepsilon}(0))
\, \vert \, {\mathcal F}_{\sigma_{J_{n}}^{\varepsilon}(0)} \bigr) \geq {\tfrac12} K(c,C). 
\end{equation*}
Therefore, by a straightforward induction, 
\begin{equation*}
 {{\mathbb P} \biggl( \bigcap_{\ell=0}^n \Bigl[ D_{j_{\ell+1}}'(\sigma_{J_{\ell}}^{\varepsilon}(0))
\Bigr]^{\complement}
 \biggr) \leq \bigl( 1- \tfrac12 K(c,C)\bigr)^{n+1}}, 
\end{equation*}
and then, for a given $\pi>0$, we can find 
$n_{\star}:=n_{\star}(\textbf{\bf A},\textbf{\bf B},a, {\pi})$ such that 
\begin{equation*}
{\mathbb P} \Bigl( \bigcup_{n=0}^{n_{\star}} D_{j_{n+1}}'(\sigma_{J_{n}}^{\varepsilon}(0))
\Bigr) \geq 1 -\pi.
\end{equation*}
Choosing $a \leq a_{\star}:=a_{\star}(\textbf{\bf A},\textbf{\bf B})$, $\delta \leq \delta_{\star}:=\delta_{\star}(\textbf{\bf A},\textbf{\bf B})$ and $v_{0} \geq v_{\star}(\textbf{\bf A},\textbf{\bf B},a,\delta,\pi)$
(which is compatible with the previous step), we may assume without any loss of generality that  
\begin{equation*}
{\mathbb P} \biggl( 
{\mathcal E}_{0}^{\varepsilon}
\cap 
{\mathcal D}(0) 
\cap 
D^{0}(0)
\cap \Bigl(\bigcup_{n=0}^{n_{\star}} D_{j_{n+1}}'(\sigma_{J_{n}}^{\varepsilon}(0))
\Bigr) \biggr)\geq 1 - {7} \pi. 
\end{equation*}
Work now on 
the event appearing in the left-hand side, namely 
${\mathcal E}_{0}^{\varepsilon}
\cap 
{\mathcal D}(0) 
\cap 
D^{0}(0)
\cap (\bigcup_{n=0}^{n_{\star}} D_{j_{n+1}}'(\sigma_{J_{n}}^{\varepsilon}(0)))$. 
Since $V_{0}^{\varepsilon} = v_{0} \varepsilon^2$ is infinitesimal, 
we deduce
by the same argument as in the second step that, on the above event, $\sigma_{1}^{\varepsilon}(0)$
is less than some (deterministic) infinitesimal time.
Since we are on $D^0(0)$, see the definition in Proposition \ref{prop:good:events}, 
we deduce that $V_{\sigma_{1}^{\varepsilon}(0)}^{\varepsilon}$ is also infinitesimal and then, by induction (since 
$J_{n_{\star}+1}$ is a deterministic integer), 
$\sigma_{J_{n_{\star}+1}}^{\varepsilon}(0)$ is also less than some 
infinitesimal $t_{{\varepsilon}}'''$. In particular (at least for $\varepsilon$ small enough), 
$\sigma_{J_{n_{\star}+1}}^{\varepsilon}(0)$ is less than 
${\boldsymbol \rho}^{\varepsilon}(0) \wedge \varrho^{\varepsilon}(0)$ since the latter is greater than 
${\boldsymbol \epsilon}$ on ${\mathcal E}_{0}^{\varepsilon}$. 
Moreover, again by definition of 
$D^0(0)$ in Proposition \ref{prop:good:events}, we have, for all $n \in \{0,\cdots,n_{\star}\}$,
\begin{equation*}
\frac12 
\varepsilon^2 v_{0} \exp \bigl(c J_{n}\bigr)
\leq
V_{\sigma_{J_{n}}^{\varepsilon}(0)}^{\varepsilon} \leq 2 \varepsilon^2 v_{0} \exp \bigl(C J_{n}\bigr).
\end{equation*}
Take now an integer $n \in \{0,\cdots,n_{\star}\}$ 
and work on 
${\mathcal E}_{0}^{\varepsilon}
\cap 
{\mathcal D}(0) 
\cap 
D^{0}(0)
\cap D_{j_{n+1}}'(\sigma^{\varepsilon}_{J_{n}}(0))$.
For a new constant $C':=C'(\textbf{\bf A},\textbf{B})$, 
we get 
from 
\eqref{eq:D:j:prime:tau}
and (again) from the fact that 
$\sigma_{J_{n+1}}^{\varepsilon}(0)=
\sigma_{j_{n+1}-j_{n}}^{\varepsilon}(\sigma_{J_{n}}^{\varepsilon}(0))$
that,  
for all $t \in [\sigma_{{J_{n}+1}}^{\varepsilon}(0),\sigma_{J_{n+1}}^{\varepsilon}(0) \wedge {\boldsymbol e}_{w}(\sigma_{J_{n}}^{\varepsilon}(0)
)]$, 
\begin{equation*}
\begin{split}
\Theta(\theta_{t}^{\varepsilon})
-
\Theta(u_{w}) 
&\geq
\Bigl(
\frac{c}{2} \frac{\varepsilon^2}{V_{\sigma_{J_{n}}^{\varepsilon}(0)}^{\varepsilon}}
\bigl( 1 - \exp(-C) \bigr) - C \int_{\sigma_{J_{n}}^{\varepsilon}(0)}^{t} V_{s}^{\varepsilon} \ud s
\Bigr) \exp \Bigl( \frac{c}{2} \bigl( \Sigma^{\varepsilon}_{\sigma_{J_{n}}^{\varepsilon}(0),t} -1 \bigr) \Bigr) 
\\
&\geq 
\Bigl(
\frac{c}{4 v_{0}} 
\bigl( 1 - \exp(-C) \bigr) \exp \bigl( - C J_{n} \bigr) - C  t_{\varepsilon}'''
\Bigr) \exp \Bigl( \frac{c}{2} \bigl( \Sigma_{\sigma_{J_{n}}^{\varepsilon}(0),t}^{\varepsilon} -1 \bigr) \Bigr),
\end{split}
\end{equation*}
where used the fact that, for a suitable choice of $C'$, $C V_{s}^{\varepsilon} \leq C'$, which follows from the 
bound $s \leq {\boldsymbol \rho}^{\varepsilon}(0)$;  {we also made use of the upper bound 
$V^{\varepsilon}_{\sigma^{\varepsilon}_{J_{n}}(0)} \leq 2 v_{0} \varepsilon^2 \exp (C J_{n})$, which follows from 
the definition of $D^0(0)$.}
Assuming that 
$\sigma_{J_{n+1}}^{\varepsilon}(0) \leq {\boldsymbol e}_{w}(\sigma_{J_{n}}^{\varepsilon}(0))$, 
we may choose $t = \sigma_{J_{n+1}}^{\varepsilon}(0)$
 in the above inequality ({notice that, by definition, 
 $\sigma_{J_{n+1}}^{\varepsilon}(0)
 \geq 
 \sigma_{J_{n}+1}^{\varepsilon}(0)$}). Observing that 
$\Sigma^{\varepsilon}_{\sigma^{\varepsilon}_{J_{n}}(0),\sigma^{\varepsilon}_{J_{n+1}}(0)}
=j_{n+1}$, we deduce that 
\begin{equation*}
\Theta \bigl( \theta_{\sigma_{J_{n+1}}^{\varepsilon}(0)}^{\varepsilon} \bigr)
-
\Theta(u_{w}) \geq 
\Bigl(
\frac{c}{4 v_{0}} 
\bigl( 1 - \exp(-C) \bigr) \exp \bigl( - C J_{n} \bigr) - C  t_{\varepsilon}'''
\Bigr) \exp \Bigl( \frac{c}2   {\bigl(j_{n+1}-1 \bigr)}\Bigr).
\end{equation*}
For $\varepsilon$ small enough and by 
\eqref{eq:j(v0)}, we obtain
\begin{equation*}
\begin{split}
\Theta \bigl( \theta_{\sigma_{J_{n+1}}^{\varepsilon}(0)}^{\varepsilon} \bigr)
-
\Theta(u_{w}) &\geq 
\frac{c}{8 v_{0}} 
\bigl( 1 - \exp(-C) \bigr) \exp \bigl( - C J_{n} \bigr)  \exp \Bigl( \frac{c}2  {\bigl( j_{n+1} -1\bigr)}\Bigr)
\\
&= \frac{c}{8 v_{0} \exp(CJ_{n})} 
\bigl( 1 - \exp(-C) \bigr)   \exp \Bigl( \frac{c}2   {\bigl( {\mathcal J}(v_{0} \exp(CJ_{n})) -1 \bigr)}\Bigr)
\geq a_{w},
\end{split}
\end{equation*}
which shows that 
${\boldsymbol e}_{w}(\sigma_{J_{n}}^{\varepsilon}(0)) \leq \sigma_{J_{n+1}}^{\varepsilon}(0)$. 
So, in any case, we have 
${\boldsymbol e}_{w}(0) \leq \sigma_{J_{n_{\star}+1}}^{\varepsilon}(0)$
on the event  
${\mathcal E}_{0}^{\varepsilon}
\cap 
{\mathcal D}(0) 
\cap 
D^{0}(0)
\cap (\bigcup_{n=0}^{n_{\star}} D_{j_{n+1}}'(\sigma_{J_{n}}^{\varepsilon}(0)))$.
This proves that the particle reaches ${\mathcal B}_{a}$
on the latter event. As a result,
 \begin{equation*}
 {\mathbb P}
 \Bigl( \bigl\{ \zeta^{\varepsilon} =0 \bigr\}
\cap
\bigl\{ \forall t \in [0,t_{\varepsilon}'''],  \ \theta_{t}^{\varepsilon}
\not \in {\mathcal B}_{a}
\bigr\}
 \Bigr) \leq 7 \pi,
\end{equation*}
which completes the proof.
}
\end{proof}

\section{Numerical example}
\label{se:7}
In this section, we provide a numerical example that illustrates our theoretical results. 
To make things simpler, we focus on a two-dimensional example, of the same type 
as in Subsection 
\ref{subse:2:2}, namely we take
$V$ of the form
$V(x) = g(\theta) \vert x \vert^{1+\alpha}$, $x \in {\mathbb R}^2$,
where $g$ is a function on ${\mathbb R}/(2 \pi {\mathbb Z})$ and $\theta$ is the angle in the polar decomposition of 
$x$.  
We choose $\alpha=0.5$ and 
\begin{equation*}
g(\theta) =  \left\{ 
\begin{array}{ll}
0 \quad &\textrm{\rm if} \ \theta \in \bigl(- \pi, - \tfrac{ \pi}{2} \bigr] \cup 
\bigl[ \frac{\pi}{2},\pi \bigr]. 
\vspace{1pt}
\\
\frac{2}{1.5+(\tan(\theta))^2}
-2 \cos (\theta)^4 + \cos(\theta)^2 
 &\textrm{\rm if} \ \theta 
\in \bigl( - \frac{\pi}{2}, \frac{\pi}{2} \bigr). 
\end{array}
\right.
\end{equation*}
The plot of $g$ together with 
the plot of a numerical approximation of $(1+\alpha)^2 g + g''$ (which is equal to the Laplacian, up to a scaling factor, and which is computed by finite differences)
are given in 
Figure \ref{fig:g:3:0} below. 
\begin{figure}[htb]
\includegraphics[scale=.35]{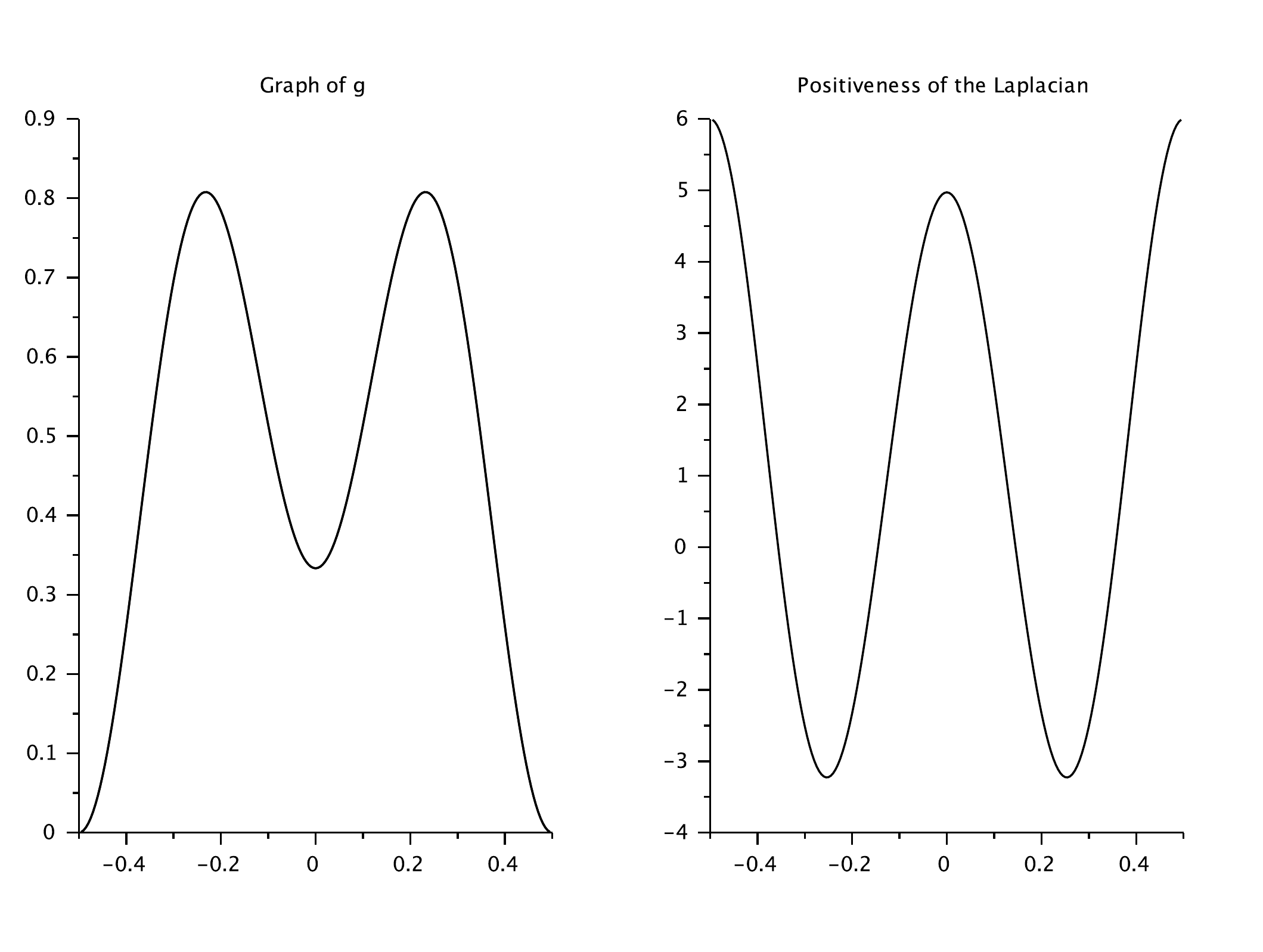}
\caption{Graph of $g$ (on its positive part) and positiveness of the Laplacian at the boundary of $\{ g>0\}$. 
Ticks in abscissa represent fractions of $\pi$ at which the function $g$ and its Laplacian are computed. }
\label{fig:g:3:0}
\end{figure}
Numerically, we find that $g$ has two maxima, at -0.23$\pi$ and 0.23$\pi$ approximatively. 

We then simulate (using a standard Euler scheme with step size $h=\varepsilon/100$)
 twelve paths of the process for the following values of $\varepsilon$: $\varepsilon=0.01$, 
 $\varepsilon=0.005$, $\varepsilon=0.002$ and $\varepsilon=0.001$. Paths are represented on 
 Figure \ref{fig:simu:3:0} below. 
\begin{figure}[htb]
\includegraphics[scale=.18]{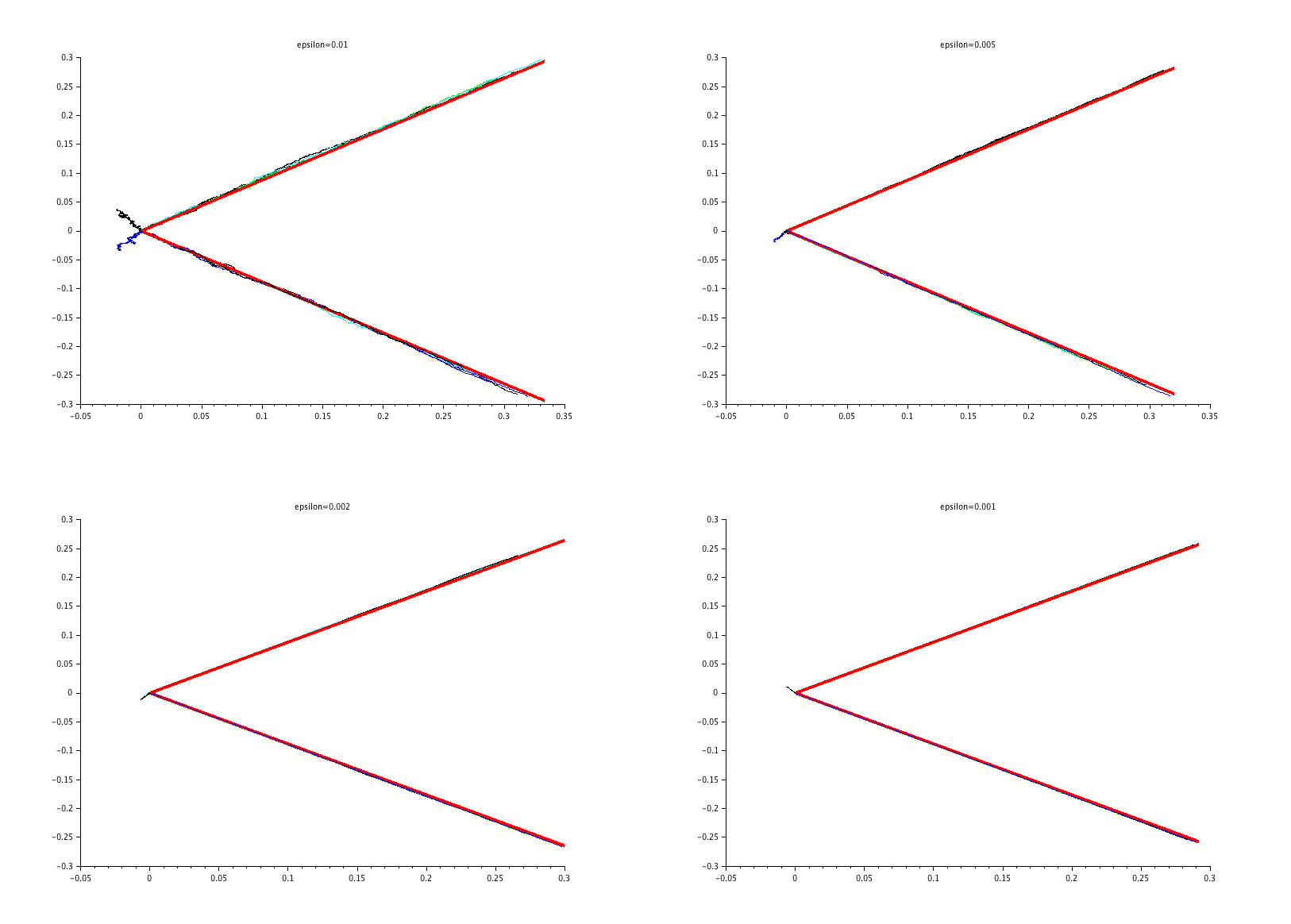}
\caption{Simulations for different values of $\varepsilon$}
\label{fig:simu:3:0}
\end{figure}

On Figure \ref{fig:simu:3:0}, thick red paths correspond to the expected exit directions, based on our theoretical results. Thin paths are the simulated ones. 
On each pane, the ``small excursions'' around the origin correspond to paths that failed to exit. There are two of them in the case $\varepsilon=.01$ and one in all other cases. On the bottom row, we hardly distinguish 
the simulated paths from the expected exit directions, which sounds as a numerical evidence 
of the fact that the particle indeed follows the maxima of $g$.

\section*{Appendix}

{\begin{lemma}
\label{lem:boundary:zero:measure}
Under Assumption \textbf{\bf A}, the boundary of $\{g=0\}$ has zero Lebesgue measure. 
\end{lemma}

\begin{proof}
It suffices to prove that, for any open ball $U$ such that $U \cap \partial \{g>0\}$ is not empty, 
$U \cap \partial \{g>0\}$ has zero measure. Without any loss of generality, we may assume 
that the closure $\overline U$ of $U$ does not contain zero and that the radius of $U$ is as small as needed.  
 
Indeed, for such an $U$,  
(\textbf{A4}) says that, provided that the radius of $U$ is small enough, 
\begin{equation}
\label{eq:zero:measure:level:set}
\begin{split}
0 < \tfrac1{p+1}
\inf_{y \in U \cap \{ g >0\}}
 \frac{\vert \nabla V(y) \vert}{V^{p/(p+1)}(y)} 
& =
\inf_{y \in U \cap \{ g >0\}}
 \bigl\vert \nabla \bigl( V^{\frac{1}{p+1}} \bigr)(y)\bigr\vert
 \\
&\leq 
\sup_{y \in U \cap \{ g >0\}}
 \bigl\vert \nabla \bigl( V^{\frac{1}{p+1}} \bigr) (y) \bigr\vert 
 = 
 \tfrac1{p+1}
 \sup_{y \in U \cap \{ g >0\}}
 \frac{\vert \nabla V(y) \vert}{V^{p/(p+1)}(y)}
< \infty.
\end{split}
 \end{equation}
Hence, we can find a unitary vector $e$ such that $\inf_{y \in U \cap \{ g >0\}}
\nabla (V^{\frac{1}{p+1}})(y) \cdot e >0$. 
 Therefore, 
 for another open ball $U' \subset U$, with the same center as $U$ but with a strictly smaller radius, 
 the fact that $V$ is ${\mathcal C}^{1,1}$ on $U$ guarantees that, for
 any $y \in U' \cap \{ g >0\}$ and for any $\lambda >0$ small enough such that $[y,y + \lambda e] \subset U \cap \{ g >0\}$, 
 \begin{equation}
\label{eq:zero:measure:level:set:2} 
 V^{\frac1{p+1}} \bigl( y + \lambda e \bigr) = V^{\frac1{p+1}}(y) + \lambda \int_{0}^1 s \nabla \bigl( V^{\frac1{p+1}} \bigr) (y + s \lambda e) \cdot e \, ds   \geq  c \lambda,
 \end{equation}
 where $c$ only depends on $\inf_{z \in U \cap \{ g >0\}}
\nabla (V^{\frac{1}{p+1}})(z) \cdot e$. 
 In particular, 
 if we choose $\lambda$ small enough such that  
 $U' + \lambda e \subset U$, then we must have $V^{1/(p+1)}(y+\lambda e) \geq c \lambda$
 for any $y \in U' \cap \{g>0\}$ 
 (if not consider $\lambda_{\star} := \inf \{ \lambda' >0 : g(y + \lambda' e) =0 \}$
 and apply \eqref{eq:zero:measure:level:set:2}  to $\lambda = \lambda_{\star} - {\epsilon}$, for ${\epsilon}$ small enough, 
 and get a contradiction by letting ${\epsilon}$ tend to zero). 
In particular, for another ball $U'' \subset U'$, with the same center but with a strictly smaller radius, 
for any $x \in U'' \cap \partial \{g>0\}$, 
it holds (by approximating $x$ by a sequence $(y_{n})_{n \geq 1}$ in $U' \cap \{g >0\}$ such that $y_{n} \rightarrow x$)
$V^{1/(p+1)}(x + \lambda e) \geq c\lambda$.  
Recalling from 
\eqref{eq:zero:measure:level:set}
that 
$V^{1/(p+1)}$ is Lipschitz continuous on $U \cap \{ g >0\}$
and 
assuming that $U'' + \lambda z \subset U'$ for any $z \in {\mathbb R}$ with $\vert z \vert \leq 1$, we deduce that there exists 
$\varrho \in (0,1)$ such that, for the same values of $\lambda$ as before, 
$V^{1/(p+1)}(x + \lambda e + \lambda \varrho z) \geq (c/2)\lambda$, for any $z \in \R^d$ with 
$\vert z \vert \leq 1$. In particular, the ball $B(x+\lambda e, \lambda \varrho)$
(of center $x + \lambda e$ and of radius $\lambda \varrho$) is included in 
$\{ g >0\}$ and is thus disjoint from $\{g=0\}$, which proves by standard arguments of porosity
that $U'' \cap \partial \{ g=0\}$ 
has zero measure, see \cite{Zajcek}.
For completeness, we provide a sketch of proof of the latter claim. If   $U'' \cap \partial \{ g>0\}$ has positive measure, then
by Lebesgue differentiation theorem, we can find some $x \in U'' \cap \partial \{ g>0\}$
such that
\begin{equation}
\label{eq:zero:measure:level:set:3}
\lim_{\lambda \searrow 0} \frac1{\vert B(x,2\lambda) \vert} \bigl\vert B(x,2\lambda) \cap  U'' \cap \partial \{ g>0\}
\bigr\vert =1,
\end{equation}
where $\vert \cdot \vert$ is here used to denote the Lebesgue measure. 
Now, for $\lambda$ small enough, 
\begin{equation*}
\bigl\vert B(x,2\lambda) \cap  U'' \cap \partial \{ g>0\}
\bigr\vert \leq \bigl\vert B(x,2\lambda) \setminus 
B(x+\lambda e, \lambda \varrho)
\bigr\vert = \bigl\vert  B(x,2\lambda) \bigr\vert 
\bigl( 1 - 2^{-d}\varrho^d\bigr),
\end{equation*}
which yields a contradiction with \eqref{eq:zero:measure:level:set:3}.
Now that $\vert U'' \cap \partial \{ g>0\} \vert$ is known to  be zero, it is easy to see that the same must hold true for 
$\vert U \cap \partial \{ g>0\} \vert$.
\end{proof}
}

\bibliographystyle{plain}
\bibliography{zeronoise}

\end{document}